\documentclass[12pt,reqno]{amsart}
\def\Ab{\underline{A}}
\usepackage{amsthm,amsfonts,amssymb,euscript,fancyhdr}

\makeatletter
\def\l@section{\@tocline{1}{0pt}{1pc}{}{}}
\renewcommand{\tocsection}[3]{%
\indentlabel{{\@ifnotempty{#2}{\ignorespaces#1 #2.\quad} #3}}}
\def\l@subsection{\@tocline{2}{0pt}{1pc}{5pc}{}}
\renewcommand{\tocsubsection}[3]{%
  \indentlabel{\hspace*{2.3em}\@ifnotempty{#2}{\ignorespaces#1 #2.\quad}}#3}
\renewcommand{\tocappendix}[3]{%
\indentlabel{#1\@ifnotempty{#2}{ #2}.\quad}#3}
\makeatother

\usepackage[colorlinks,linkcolor=blue,citecolor=blue]{hyperref}

\newtheorem{theorem}{Theorem}[section]
\newtheorem{lemma}[theorem]{Lemma}
\newtheorem{proposition}[theorem]{Proposition}

\newtheorem{definition}[theorem]{Definition}
\newtheorem{remark}[theorem]{Remark}
\newtheorem{conjecture}[theorem]{Conjecture}

\numberwithin{equation}{section}

\def\a{{\alpha}}

\def\be{{\beta}}
\def\ga{\gamma}
\def\Ga{\Gamma}
\def\de{\delta}
\def\De{\Delta}

\def\la{\lambda}

\def\rh{\rho}
\def\si{\sigma}
\def\Si{\Sigma}
\def\om{\omega}
\def\Om{\Omega}

\def\Th{\Theta}
\def\ze{\zeta}

\def\CC{{\mathcal C}}

\def\MM{{\mathcal M}}
\def\NN{{\mathcal N}}
\def\II{{\mathcal I}}

\def\HH{{\mathcal H}}

\def\OO{{\mathcal O}}

\def\NN{{\mathcal N}}

\def\PP{{\mathcal P}}
\def\RR{{\mathcal R}}
\def\QQ{{\mathcal Q}}

\def\D{{\bf D}}

\def\g{{\bf g}}

\def\SSS{{\mathbb S}}
\def\RRR{{\mathbb R}}
\def\ZZZ{{\mathbb Z}}

\def\NNN{{\mathbb N}}

\def\Lb{{\,\underline{L}}}

\def\chih{{\widehat \chi}}
\def\chib{{\underline \chi}}
\def\chibh{{\underline{\chih}}}

\def\etab{{\underline \eta}}

\def\alphab{{\underline{\alpha}}}
\def\betab{{\underline{\beta}}}

\newcommand{\Divd}{\Div \mkern-17mu /\ }
\newcommand{\Curld}{\Curl \mkern-17mu /\ }
\newcommand{\Nd}{\nabla \mkern-13mu /\ }

\newcommand{\Ld}{\triangle \mkern-12mu /\ }
\newcommand{\iin}{\in \mkern-16mu /\ \mkern-5mu}
\newcommand{\gd}{{g \mkern-8mu /\ \mkern-5mu }}

\newcommand{\LIE}{\mathcal{L} \mkern-11mu /\  \mkern-3mu }
\newcommand{\cDd}{\mathcal{D} \mkern-11mu /\  \mkern-3mu}

\newcommand{\RRt}{{\mathcal{R}}}
\newcommand{\IIt}{{\mathcal{I}}}

\newcommand{\Lt}{{{L}}}
\newcommand{\Lbt}{{{\Lb}}}
\newcommand{\Omt}{{{\Om}}}
\newcommand{\ubt}{{{\ub}}}
\newcommand{\Gat}{{{\Ga}}}
\newcommand{\etA}{{{e}_A}}
\newcommand{\etB}{{{e}_B}}

\newcommand{\chit}{{{\chi}}}
\DeclareRobustCommand{\chibt}{{\underline{\chi}}}

\newcommand{\zet}{{{\ze}}}
\newcommand{\etabt}{{{\etab}}}

\newcommand{\chiht}{{\widehat{{\chi}}}}
\newcommand{\chibht}{{\widehat{{\underline{\chi}}}}}
\newcommand{\trchit}{{\tr \chit}}

\newcommand{\trchibt}{{\tr \chibt}}
\newcommand{\alphat}{{{\alpha}}}
\newcommand{\betat}{{{\beta}}}
\newcommand{\rhot}{{{\rho}}}
\newcommand{\sigmat}{{{\sigma}}}
\newcommand{\betabt}{{{\betab}}}
\newcommand{\alphabt}{{{\alphab}}}
\newcommand{\mut}{{{\mu}}}
\newcommand{\Kt}{{{K}}}

\newcommand{\rhotc}{\check{\rhot}}
\newcommand{\OOt}{{{\OO}}}

\newcommand{\etatt}{{\tilde{\eta}}}

\providecommand{\norm}[1]{ \left\Vert #1 \right\Vert }

\DeclareMathOperator{\Div}{\mathrm{div}}
\DeclareMathOperator*{\Curl}{\mathrm{curl}}

\def\nab{\nabla}
\def\varep{\varepsilon}

\def\pr{{\partial}}
\def\les{\lesssim}

\def\f12{{\frac 1 2}} 
\def\dual{{{}^\ast}} 
\def\tr{\mathrm{tr}}

\def\f{\widetilde{f}}

\def\otimesh{{\widehat{\otimes}}}
\def\half{\frac{1}{2}}
\newcommand{\RRRic}{\mathrm{Ric}}

\newcommand{\Rbf}{\mathbf{R}}

\newcommand{\lab}{\label}

\newcommand{\R}{{\bf R}}

\newcommand{\trchi}{\tr \chi}
\newcommand{\trchib}{\tr \chib}
\newcommand{\rhotco}{{\overline{\rhotc}}}
\newcommand{\sigmatc}{{\check{\sigmat}}}

\newcommand{\Dd}{\cDd}
\newcommand{\Ddi}{\cDd^{-1}}

\renewcommand{\ubt}{v}
\newcommand{\betabtc}{\check{\betab}}

\newcommand{\vtt}{{\tilde{v}}}

\newcommand{\Stt}{{\tilde{S}}}

\newcommand{\Dds}{{{}^\ast\Dd}}

\newcommand{\vt}{{\tilde{v}}}

\newcommand{\PPo}{{\PP^0_v}}

\newcommand{\muc}{{\mu}}
\newcommand{\betabc}{{\check{\betab}}}
\newcommand{\rhoc}{{\check{\rho}}}
\renewcommand{\Ga}{{S}}
\newcommand{\Ups}{\Upsilon}
\newcommand{\dg}{{\dagger}}
\newcommand{\ddg}{{\ddagger}}

\newcommand{\rhoco}{\rhotco}
\newcommand{\muo}{{\overline{\mu}}}
\newcommand{\muco}{{\overline{\mu}}}
\newcommand{\Nda}{{^{(0)}\Nd}}

\newcommand{\Ndc}{{^{(n+1)}\Nd}}
\newcommand{\Lda}{{^{(0)}\Ld}}

\newcommand{\Ldc}{{^{(n+1)}\Ld}}
\newcommand{\QQo}{{\QQ^{1/2}_v}}
\begin{document}

\title[The canonical foliation on null hypersurfaces in low regularity]{The canonical foliation on null hypersurfaces \\ in low regularity}

\address[Stefan Czimek]{Department of Mathematics, University of Toronto, Canada}

\author[Stefan Czimek]{Stefan Czimek}
\email{stefan.czimek@utoronto.ca}

\address[Olivier Graf]{Laboratoire Jacques-Louis Lions, Sorbonne University, France}

\author[Olivier Graf]{Olivier Graf}
\email{grafo@ljll.math.upmc.fr}

\begin{abstract} Let $\HH$ denote the future outgoing null hypersurface emanating from a spacelike 2-sphere $S$ in a vacuum spacetime $(\MM,\g)$. In this paper we study the so-called \emph{canonical foliation} on $\HH$ introduced in~\cite{KlainermanNicolo},~\cite{Nicolo} and show that the corresponding geometry is controlled locally only in terms of the initial geometry on $S$ and the $L^2$ curvature flux through $\HH$. In particular, we show that the ingoing and outgoing null expansions $\tr \chi$ and $\tr \chib$ are both locally uniformly bounded. The proof of our estimates relies on a generalisation of the methods of \cite{KlRod1} \cite{KlRod2} \cite{KlRod3} and \cite{Alexakis} \cite{AlexShaoGeom} \cite{ShaoBesov} \cite{Wang} where the geodesic foliation on null hypersurfaces $\HH$ is studied. The results of this paper, while of independent interest, are essential for the proof of the spacelike-characteristic bounded $L^2$ curvature theorem \cite{CzimekGraf2}.
\end{abstract}

\date{\today}
\maketitle
\setcounter{tocdepth}{2}
\tableofcontents


\section{Introduction}

\subsection{Einstein vacuum equations}

A Lorentzian $4$-manifold $(\mathcal{M},{\bf g})$ is called a \emph{vacuum spacetime} if it solves the Einstein vacuum equations
\begin{align} \label{EinsteinVacuumEquationsIntroJ}
\mathbf{Ric} = 0,
\end{align}
where $\mathbf{Ric}$ denotes the Ricci tensor of the Lorentzian metric ${ \bf g}$. Expressed in general coordinates, \eqref{EinsteinVacuumEquationsIntroJ} is a non-linear coupled system of partial differential equations of order $2$ for ${\bf g}$. In so-called \emph{wave coordinates}, it can be shown that \eqref{EinsteinVacuumEquationsIntroJ} is a system of nonlinear wave equations. Hence it admits an initial value formulation. Moreover, the characteristic hypersurfaces of these equations are the \emph{null} hypersurfaces of the spacetime $(\MM,\g)$.

\subsection{The \emph{weak cosmic censorship conjecture} and the \emph{bounded $L^2$ theorem}}
The global behaviour of solutions to~(\ref{EinsteinVacuumEquationsIntroJ}) is subject to the celebrated conjecture of \emph{weak cosmic censorship} formulated by Penrose~\cite{PenroseConjecture}.
\begin{conjecture}[Weak cosmic censorship conjecture (rough version), \cite{PenroseConjecture}]
  For generic initial data, all singularities are hidden in a black hole region.
\end{conjecture}
In the seminal work~\cite{Chr9}, it is shown that the conjecture holds true in the case of spherical symmetry for Einstein equations coupled with a scalar field. The result relies on the sharp breakdown criterion and local existence result proved in~\cite{Chr7} at the level of data with bounded variation, which is adapted to the $(1+1)$-setting of spherical symmetry.\\

In the case of Einstein vacuum equations (\ref{EinsteinVacuumEquationsIntroJ}) without symmetry, local existence results are naturally formulated in terms of $L^2$-based functional spaces (see the discussion in the introduction of~\cite{KRS}). In this context, the sharpest known local existence result in term of regularity of the initial data is the celebrated \emph{bounded $L^2$ curvature theorem} 
 (see~\cite{KRS} and the companion papers~\cite{J1}-\cite{J5}). The following is a rough statement of that result.
\begin{theorem}[Bounded $L^2$ theorem (rough version), \cite{KRS}]\label{thm:bddL2rough}
  For initial data to the Einstein equations~\eqref{EinsteinVacuumEquationsIntroJ} on a spacelike hypersurface $\Si$ such that the spacetime curvature tensor $\R$ is bounded in $L^2(\Si)$, there exists a local Cauchy development that satisfies Einstein equations~\eqref{EinsteinVacuumEquationsIntroJ}.  
\end{theorem}

In the proof~\cite{Chr9} of the weak cosmic censorship conjecture in spherical symmetry, it is crucial that the local existence result in~\cite{Chr7} is formulated on null hypersurfaces, in order to highlight a trapped surface formation mechanism (see~\cite{Chr6},~\cite{Chr9} and also~\cite{LiLiu} for further discussion).\\

The aim of the present paper is to initiate the proof of a local existence result for initial data on null hypersurfaces with no symmetry assumption, assuming only finite $L^2$ curvature. Together with the companion paper~\cite{CzimekGraf2}, this will amount to a proof of a spacelike-characteristic bounded $L^2$ theorem that generalises the bounded $L^2$ Theorem~\ref{thm:bddL2rough} to the case of initial data posed on a characteristic hypersurface instead of a spacelike hypersurface (see Section~\ref{SECintroCauchyProblem} for further discussion).



\subsection{Null hypersurfaces, foliations and  geometry}
In various problems, foliating vacuum spacetimes by null hypersurfaces is a powerful tool to capture the propagation features of the Einstein equations. We refer the reader for example to \cite{ChrKl93}, \cite{KRS}, \cite{J1}, \cite{J2}, \cite{J3}, \cite{J4} where the spacetimes are foliated by a mixed spacelike-null foliation (one family of null hypersurfaces and one family of spacelike hypersurfaces), and for example to \cite{ChrFormationNonSpherical}, \cite{KlainermanNicolo}, \cite{LukRodnianski} where the spacetimes are foliated by a double null foliation (two families of transversely intersecting null hypersurfaces). \\

When using mixed spacelike-null foliations, the family of null hypersurfaces is typically determined by prescribing the corresponding induced foliation on an initial spacelike hypersurface. This is equivalent to prescribing the values of the optical function $u$, whose level sets are the null hypersurfaces, on the initial spacelike hypersurface. In particular, the regularity of the induced foliation on the initial spacelike hypersurface determines the regularity of the corresponding foliation by null hypersurfaces and hence needs to be carefully picked depending on the situation (see the different constructions of the optical functions in \cite{ChrKl93} and in \cite{J1} for example).\\

The family of spacelike hypersurfaces is itself typically determined by defining them to be maximal hypersurfaces and fixing their asymptotics towards spacelike infinity or prescribing their finite boundary. In case of spacelike-null foliations these boundaries can be naturally prescribed by choosing the induced foliation on an initial null hypersurface, see \cite{BartnikExistence}. Similarly, in double null foliations, the two families of null hypersurfaces are entirely determined by the foliation they induce on two transverse initial null hypersurfaces. \\

A standard choice of foliation on initial null hypersurfaces is the \emph{geodesic foliation} (see below for a definition and \cite{ChrFormationNonSpherical}, \cite{Luk}, \cite{LukRodnianski} for examples where this foliation is used as an initial foliation). In more specific situations, other foliations have to be considered: the so-called \emph{canonical foliation} on null hypersurfaces in \cite{KlainermanNicolo} and \cite{Nicolo} (see also Definition~\ref{def:canfol}) for its additional regularity features, the so-called \emph{constant expansion} and \emph{constant mass aspect function} foliations in~\cite{Sauter} to obtain monotonicity properties for the Hawking mass (see also~\cite{Roesch}).\\



In this paper, we consider foliations on an outgoing truncated null hypersurface $\HH$ emanating from a spacelike $2$-sphere $S$, given by the level sets $S_v$ of a scalar function $v\in[1,2]$ and we assume that the first leaf of this foliation coincides with $S$, \emph{i.e.} $S = S_{v=1}$. Given $L$ a null geodesic generator of $\HH$, we define the \emph{null lapse} $\Om$ of the foliation $(S_v)$ to be
\begin{align*}
  \Om := Lv.
\end{align*}
The \emph{geodesic foliation} corresponds to $\Om=1$ and we call $s$ its parameter. Note that it depends on the choice of $L$, which we assume to be fixed, here and in the rest of the paper.

We denote by $\Lb$ the null vector field orthogonal to $S_v$ and transverse to $\HH$ such that $\g(L,\Lb) = -2$. The geometry of the foliation $(S_v)$ on $\HH$ is described by the induced metric $\gd$ on the $2$-spheres $S_v$ by $\g$, the \emph{intrinsic null second fundamental form} $\chi$, the \emph{torsion} $\ze$ and the \emph{extrinsic null second fundamental form} $\chib$, respectively defined by
\begin{align*}
  \chi(X,Y):=\g(\D_XL,Y), \,\,\,\,\, \zet(X) :=\half\g(\D_XL,\Lb), \,\,\,\,\, \chib(X,Y) := \g(D_X\Lb,Y),
\end{align*}
where $X,Y$ are $S_v$-tangent vectors and $\D$ denotes the covariant derivative on $(\MM,\g)$. The quantities $\chi,\ze$ and $\chib$ are called the \emph{null connection coefficients}. In the following, we often split up $\chi$ and $\chib$ into their trace and tracefree parts,
\begin{align*}
\begin{aligned}
\trchit & := \gd^{AB}\chit_{AB},\,\,\,\,\,\,\,& \chiht_{AB} & := \chit_{AB} -\half \trchit \gd_{AB}, \\
\trchibt & := \gd^{AB}\chibt_{AB},\,\,\,\,\,\,\,& \chibht_{AB} & := \chibt_{AB} - \half \trchibt \gd_{AB}.
\end{aligned}
\end{align*}

Geometrically, the intrinsic second fundamental form $\chi$ measures how the spheres $S_v$ and their first fundamental form $\gd$ change along $\HH$ and the extrinsic second fundamental form $\chib$ measures how the $2$-spheres $S_v$ change in the null direction transverse to $\HH$ given by $\Lb$. In particular, we see that the geometry of hypersurfaces emanating from the $2$-spheres $S_v$ transversely to $\HH$ must critically depend on the regularity of $\chib$ on $S_v$.\\

We also have the following decomposition of the spacetime curvature tensor $\R$ into the \emph{null curvature components} relative to $L$ and $\Lb$.
\begin{align*}
& \a(X,Y) := \Rbf(X,L,Y,L), & \beta(X) :=& \half \Rbf(X,L,\Lb,L), & \rh :=& \frac{1}{4}\Rbf(L, \Lb, L, \Lb), \\
& \si := \frac{1}{4} \dual\Rbf(\Lb,L,\Lb,L), & \betab(X) :=& \half \Rbf(X,\Lb,\Lb,L), & \alphab(X,Y) :=& \Rbf(\Lb, X, \Lb,Y),
\end{align*}
where $X$ and $Y$ are $S_v$-tangent vectors.

We define the \emph{$L^2$ curvature flux through $\HH$} by  
\begin{align}\label{eq:L2flux}
\RR_{\HH} := \left( \Vert \a \Vert_{L^2(\HH)}^2+\Vert \beta \Vert_{L^2(\HH)}^2+\Vert \rh \Vert_{L^2(\HH)}^2+\Vert \si \Vert_{L^2(\HH)}^2+\Vert \betab \Vert_{L^2(\HH)}^2 \right)^{1/2}.
\end{align}


\subsection{Regularity of the foliation on $\HH$.}
Our goal in this paper is to provide a local initial foliation $(S_v)$ on $\HH$ such that the geometry of a family of hypersurfaces transverse to $\HH$ emanating from the $2$-spheres $S_v$ can be locally controlled under the assumption of finite $L^2$ curvature flux. One needs to control the intrinsic and extrinsic geometry of the foliation $(S_v)$, that is, to provide bounds on $\HH$ for the null lapse $\Om$, the induced metric $\gd$ and the null connection coefficients $\chi,\ze$ and $\chib$ of the foliation $(S_v)$ assuming only a control on the $L^2$ curvature flux. \\

Here and in the rest of the paper, all quantities specific to the geodesic foliation will be noted with a prime.

In \cite{KlRod1}, the following groundbreaking result is proved for the geodesic foliation.
\begin{theorem}[Control of the geodesic foliation, rough version, \cite{KlRod1}] \label{THMKlRodrough} Let $(\MM,\g)$ be a vacuum spacetime. Let $\HH$ be an outgoing null hypersurface emanating from a spacelike $2$-sphere $(S,\gd)$ foliated by the geodesic foliation associated to the affine parameter $s$ going from $s \vert_{S} = 1$ to $s\simeq 2$. Let $\II'_{S}$ denote low regularity norms on $\chi',\zet'$ and $\chib'$ on $S$ (see Section~\ref{SECdefinitionOfNorms} for a definition). Assume that $\II'_{S}$ and the $L^2$-curvature flux $\RR'_{\HH}$ are sufficiently small. Then
\begin{align*}
\left\Vert \tr \chi' - \frac{2}{s} \right\Vert_{L^\infty(\HH)} &\les \II'_{S} + \RR'_{\HH}, \\
\Vert \chih' \Vert_{L^2(\HH)} + \Vert \nab' \chih' \Vert_{L^2(\HH)} &\les\II'_{S} + \RR'_{\HH}, \\
\Vert \zet' \Vert_{L^2(\HH)} + \Vert \nab'\zet' \Vert_{L^2(\HH)} &\les  \II'_{S} + \RR'_{\HH}, \\
\left\Vert \tr \chib' + \frac{2}{s} \right\Vert_{L^2(\HH)} + \left\Vert \chibh' \right\Vert_{L^2(\HH)} &\les \II'_{S} +\RR'_{\HH},
\end{align*}
where $\nab' \in \{\Nd',\Nd'_L\}$ (see Definition~\ref{def:projnab}). This holds together with additional, more specific estimates for $\chi'$, $\zet'$ and $\chib'$ (see Section \ref{SECdefinitionOfNorms} for more precise norms).
\end{theorem}

{\bf Remarks.}
\begin{enumerate}
\item The smallness assumptions on $\II'_{S}$ imply that the $2$-sphere $S$ is close to the euclidian $2$-sphere of radius $1$ in a weak sense, and the smallness assumption on $\RR'_{\HH}$ imply that the null curvature components are close to the their (trivial) value in Minkowski spacetime in $L^2$-norm.  
\item The proof of Theorem \ref{THMKlRodrough} is obtained by analysing the so-called \emph{null structure equations} and \emph{null Bianchi equations} which are consequences of geometric constraints and the Einstein equations~\eqref{EinsteinVacuumEquationsIntroJ} and relate $\gd',\chi',\ze',\chib'$ to $\alpha',\beta',\rho',\sigma',\betab'$, see Sections~\ref{sec:nullstructureeqOLD} and~\ref{sec:bianchi} and \cite{ChrKl93} and~\cite{KlRod1} for details.
\item It uses sharp bilinear and trace estimates and geometric Littlewood-Paley theory, see also \cite{KlRod2} \cite{KlRod3} and the subsequent \cite{Alexakis} \cite{AlexShaoGeom} \cite{ShaoBesov} \cite{J1} \cite{J2} \cite{J3} \cite{J4} \cite{Wang}.
\item This theorem gives a local control of the geodesic foliation in terms of the $L^2$ curvature flux. In~\cite{Alexakis} and~\cite{AlexShaoGeom}, a global control on the geodesic foliation was obtained provided that the (weighted) $L^2$ curvature flux is sufficiently close to the Schwarzschild data. In~\cite{Wang}, a local control result was obtained when $\HH$ is the null cone emanating from a point.
\item The control of the extrinsic coefficients $\trchib'$ and $\chibh'$ is significantly weaker than the control of $\trchi'$ and $\chih'$ on $\HH$.
\end{enumerate}

We will need bounds for $\trchib$ and $\chibh$ comparable to the ones for $\trchi$ and $\chih$ in order to control transversely emanating hypersurfaces. As outlined above, this does not hold for the geodesic foliation.\\

In the next section we turn to the study of the \emph{canonical foliation} on $\HH$, which was defined in~\cite{KlainermanNicolo} and~\cite{Nicolo} for its improved regularity features for $\trchib$ and $\chibh$ (see the discussion in~\cite{KlainermanNicolo2}). This regularity improvement is sufficient for the spacelike-characteristic bounded $L^2$ theorem in our companion paper~\cite{CzimekGraf2} (see also Section \ref{SECintroCauchyProblem} for more details).
                                                                                                             

\subsection{The canonical foliation and first version of the main result} In this section, we define the canonical foliation on $\HH$. Notation and more precise definitions are given in Section \ref{SECmainresults}.

\begin{definition}[Canonical foliation]\label{def:canfol} A foliation $(S_v)$ on $\HH$ is called \emph{canonical foliation}, if the null lapse $\Om$ satisfies
\begin{align*} 
  \Ld (\log \Om) & = -\Divd \zeta + \rho - \half \chih \cdot \chibh - \overline{\rho} + \half\overline{\chih\cdot\chibh},\\
  \int_{S_v}\log\Om & = 0.
\end{align*}
where $\Ld$ denotes the induced Laplace-Beltrami operator on $S_v$, and $\Divd$ the divergence operator acting on $S_v$-tangent vector fields, $\overline{\rho}$ and $\overline{\chih\cdot\chibh}$ denote the average of $\rho$ and $\chih\cdot\chibh$ on $S_v$ respectively.
\end{definition}
\begin{remark}
  In this paper, we use the definitions for the connection coefficients from~\cite{KlRod1}, for this choice makes the dependency in the null lapse $\Om$ disappear in the null structure equations (see Section~\ref{sec:nullstructureeqOLD}). This accounts for the apparent discrepancy with the original canonical foliation definition (see Definition 3.3.2 in~\cite{KlainermanNicolo}).
\end{remark}

The following is a first version of the main result of this paper, see Theorem \ref{TheoremMassLapseFoliationControl} for the precise version.
\begin{theorem}[Existence and control of the canonical foliation, version 1] \label{THMmainRoughIntro} Let $(\MM,\g)$ be a vacuum spacetime. Let $\HH$ be an outgoing null hypersurface emanating from a spacelike $2$-sphere $(S,\gd)$ and foliated by a smooth geodesic foliation associated to the affine parameter $s$ taking values between $s \vert_{S} = 1$ and $s=5/2$. Assume that the initial data norm $\II_{S}$ at $S_1$ and the $L^2$ curvature flux $\RR_{\HH}'$ (with respect to the geodesic foliation) are sufficiently small. Then:

\begin{enumerate}
\item \textbf{$L^2$-regularity.} The canonical foliation $(S_v)$ on $\HH$ is well-defined from $v=1$ to $2$ and
\begin{align*}
\left\Vert \tr \chi - \frac{2}{v} \right\Vert_{L^\infty(\HH)} + \left\Vert \tr \chib + \frac{2}{v} \right\Vert_{L^\infty(\HH)}& \les \II'_{S} + \RR_{\HH}', \\
\Vert \chih \Vert_{L^2(\HH)} + \Vert \nab \chih \Vert_{L^2(\HH)} & \les \II'_{S} + \RR_{\HH}',\\
\Vert \zet \Vert_{L^2(\HH)} + \Vert \nab \zet \Vert_{L^2(\HH)} & \les \II'_{S} + \RR_{\HH}',\\
\Vert \chibh \Vert_{L^2(\HH)} + \Vert \nab \chibh \Vert_{L^2(\HH)} & \les \II'_{S} + \RR_{\HH}', \\
\Vert \Om-1 \Vert_{L^\infty(\HH)} + \Vert \nab\Om \Vert_{L^2(\HH)} + \Vert \nab\Nd\Om \Vert_{L^2(\HH)} & \les \II'_{S} +\RR_{\HH}',
\end{align*}
where $\nab \in \{\Nd,\Nd_L\}$ (see Definition~\ref{def:projnab}).
Additional, more specific estimates hold for $\chi$, $\zet$, $\Om$ and $\chib$ (see Theorem \ref{TheoremMassLapseFoliationControl}).
\item \textbf{Higher regularity.} The smoothness of the geodesic foliation implies smoothness of the canonical foliation.
\end{enumerate}
\end{theorem}

{\bf Remarks.}
\begin{enumerate}
\item In Theorem \ref{THMmainRoughIntro}, the regularity of $\trchib$ and $\chibh$ is improved compared to Theorem \ref{THMKlRodrough}. In particular, the regularity of $\chib$ is sufficient for the spacelike-characteristic bounded $L^2$ theorem in the companion paper~\cite{CzimekGraf2} (see also Section \ref{SECintroCauchyProblem}).
\item The canonical foliation displays better regularity features for $\chib$ than the geodesic foliation because of a simplified transport equation for $\tr \chib$. More precisely, while in the geodesic foliation it holds that
\begin{align*}
\Lt(\trchibt') +\half\trchit' \trchibt' =& -2\Divd\zeta' +2\left( \rh' - \half \chih' \cdot \chibh' \right)+2| \zeta' |^2
\end{align*}
where a low regularity curvature term is present on the right-hand side, in the canonical foliation we have
\begin{align*}
\Lt(\trchibt) +\half\trchit\trchibt & =  2\overline{\rho} -\overline{\chih\cdot\chibh} +2|\Nd \Om - \zeta|^2,
\end{align*}
where the right-hand side has improved tangential regularity (see Lemma \ref{prop:transportequationfortrchibt}). This allows for an improved control of $\tr \chib$ and subsequently $\chibh$ on $\HH$.
\item The methods in the proof of Theorem \ref{THMmainRoughIntro} are generalisations of \cite{KlRod1} \cite{KlRod2} \cite{KlRod3} and the subsequent \cite{Alexakis} \cite{AlexShaoGeom} \cite{ShaoBesov} \cite{Wang} where the geodesic foliation is studied (see also~\cite{J1}-\cite{J5}). A new difficulty that arises in our analysis is that, in contrast to the geodesic foliation where $\Om \equiv 1$, the null lapse $\Om$ has only low regularity and hence must be treated with care (see Sections~\ref{SECmainresults},~\ref{sec:aprioriL2}, \ref{sec:apriorihigher} and \ref{sec:thmloc}).
\item The functional analysis tools are listed in Section~\ref{sec:calculus} and are mostly taken from~\cite{ShaoBesov}, which is the latest version of the ideas from the groundbreaking~\cite{KlRod1},~\cite{KlRod2} and~\cite{KlRod3} (see also~\cite{Wang} and~\cite{J1}-\cite{J5}). 
\item The quantity $\II_{S}$ contains the same norms as $\II'_{S}$ in Theorem~\ref{THMKlRodrough} together with additional norms on $\trchib'$ and $\chib'$ in order for the new bounds to hold. These additional norms are at the same level of regularity as the required norms for $\trchi'$ and $\chih'$. 
\item We use in its full extent that the geodesic connection coefficients are controlled by Theorem~\ref{THMKlRodrough} and that a small change of foliation leaves the null curvature components, the second fundamental form $\chi$ and some geometric norms essentially invariant (see Sections~\ref{SECmainresults},~\ref{sec:aprioriL2}, \ref{sec:apriorihigher} and \ref{sec:thmloc}).
\item The proof of Theorem \ref{THMmainRoughIntro} implies in particular that if a given canonical foliation on $\HH$ has small initial norm $\II_{S}$ at $S$ and small $L^2$ curvature flux $\RR_{\HH}$ with respect to the canonical foliation on $\HH$, then the foliation geometry on $\HH$ is controlled as stated in Theorem \ref{THMmainRoughIntro} with $\II_{S}$ and $\RR_{\HH}$ on the right-hand side. Since such a formulation would require that the canonical foliation a priori exists, we prefered to state the smallness assumptions with respect to the geodesic foliation.
\end{enumerate}

\subsection{The spacelike-characteristic Cauchy problem of general relativity in low regularity} \label{SECintroCauchyProblem} Our motivation for the main Theorem~\ref{THMmainRoughIntro} in this paper is its application to the authors' spacelike-characteristic bounded $L^2$ theorem~\cite{CzimekGraf2}. First we define the volume radius of a Riemannian $3$-manifold.

\begin{definition}[Volume radius] Let $(\Si,g)$ be a Riemannian $3$-manifold with boundary, and let $r>0$ be a real number. The \emph{volume radius of $\Si$ at scale $r$} is defined by
\begin{align*}
r_{vol}(\Si,r) := \inf\limits_{p \in \Si} \inf\limits_{0<r'<r} \frac{\mathrm{vol}_g( B_g(p,r') )}{(r')^3},
\end{align*}
where $B_g(p,r')$ denotes the geodesic ball of radius $r'$ centred at $p \in \Si$.
\end{definition}

\begin{theorem}[The spacelike-characteristic bounded $L^2$-curvature theorem, \cite{CzimekGraf2}] \label{thm:MainResultIntro1} Let $(\MM,\g)$ be a smooth vacuum spacetime with past boundary consisting of a smooth maximal spacelike hypersurface $\Si \simeq \overline{B(0,1)} \subset \RRR^3$ and the outgoing null hypersurface $\HH$ emanating from $\pr \Si=S$. Assume that for some $\varep>0$,
\begin{align*}
\II_{S}+ \RR_{\HH}' & \leq \varep,& \Vert \RRRic \Vert_{L^2(\Si)} & \leq \varep,& \Vert k \Vert_{L^2(\Si)} + \Vert \nab k \Vert_{L^2(\Si)}& \leq \varep, \\
r_{vol}(\Si,1/2) & \geq 1/4,& \mathrm{vol}_g(\Si)& <\infty,& &
\end{align*}
where the initial foliation geometry $\II_{S}'$ and the $L^2$ curvature flux $\RR_{\HH}'$ are the same as in Theorem~\ref{THMmainRoughIntro}, and $\RRRic$ and $k$ denote the Ricci tensor and second fundamental form of $\Si \subset \MM$. Then:
\begin{enumerate}
\item {\bf $L^2$-regularity.} There is a small universal constant $\varep_0>0$ such that if $0<\varep< \varep_0$, then the maximal globally hyperbolic development of $(\MM,\g)$ contains a future region of $\Si\cup\HH$ which is foliated by maximal spacelike hypersurfaces $\Si_t$ given as level sets of a time function $t$ such that $\Si_1=\Si$ and
\begin{align*}
\pr \Si_t = S_t \text{ on } \HH,
\end{align*}
where $(S_t)_{1\leq t\leq 2}$ is the canonical foliation on $\HH$, and the following control holds for $1 \leq t \leq 2$,
\begin{align*} \begin{aligned}
\Vert \RRRic \Vert_{L^\infty_t L^2(\Si_t)} & \lesssim \varep,& \Vert k \Vert_{L^\infty_t L^2(\Si_t)} + \Vert \nab k \Vert_{L^\infty_t L^2(\Si_t)} & \lesssim \varep, \\
\inf\limits_{1\leq t \leq 2} r_{vol}(\Si_t,1/2) & \geq \frac{1}{8}, & \mathrm{vol}_g(\Si_t) < \infty.
\end{aligned} \end{align*}
\item {\bf Higher regularity.} The resulting spacetime region is smooth up to $t=2$.
\end{enumerate}
\end{theorem}

{\bf Remarks.}
\begin{enumerate}
\item In the proof of Theorem \ref{thm:MainResultIntro1}, the boundary regularity of the hypersurfaces $\Si_t$ is directly related to the regularity of the canonical foliation $(S_v)$ on $\HH$. More specifically, for the control of $k$ on $\pr \Si_t$, it is necessary to control both $\chi$ and $\chib$ in sufficiently high regularity, which is achieved by the main result of this paper, Theorem \ref{THMmainRoughIntro}.

\item In addition to the estimates of this paper, the proof of Theorem \ref{thm:MainResultIntro1} relies on the \emph{bounded $L^2$ curvature theorem} \cite{KRS}, the \emph{extension procedure for the constraint equations} \cite{Czimek1}, Cheeger-Gromov convergence theory on manifolds (with boundary) in low regularity \cite{Czimek21} \cite{Czimek22}, and global estimates for maximal spacelike hypersurfaces.

\end{enumerate}

\subsection{Acknowledgements} Both authors are very grateful to J\'er\'emie Szeftel for many interesting and stimulating discussions. The second author is supported by the ERC grant ERC-2016 CoG 725589 EPGR.

\section{Geometric setup and main results} \label{SECmainresults} In this section, we introduce the geometric setup of this paper and give a precise statement of our main result as well as an overview of its proof. In particular, in Section~\ref{sec:HNL}, we define the \emph{canonical foliation} and our main theorem is stated in Section~\ref{sec:StatementMainResultandNORMS}.

\subsection{Foliations on null hypersurfaces}
In this section, we set up foliations on null hypersurfaces following the notations (and normalisations) of \cite{KlRod1}. Let $(\MM,\g)$ be a Lorentzian manifold and let $S\subset \MM$ be a spacelike $2$-sphere. Let $\HH$ denote the outgoing null hypersurface emanating from $S$.

\begin{definition}[Geodesic foliation on $\HH$]
  Let $L$ be an $\HH$-tangential null vector field on $S$ orthogonal to $S$.
  Extend $L$ as null geodesic vector field onto $\HH$.
  Define the affine parameter $s$ of $L$ on $\HH$ by
\begin{align*}
Ls=1 \text{ on } \HH, \,\, s \vert_{S} =1.
\end{align*}
Denote the level sets of $s$ by $$S'_{s_0}= \{ s = s_0\}$$
and denote the geodesic foliation by $(S'_s)$.
\end{definition}

\begin{definition}[General foliations on $\HH$]\label{def:nulllapse} Let $v$ be a given scalar function on $\HH$. We denote the level sets of $v$ by
  $$\Gat_{v_0} = \{v= v_0\},$$
  and the foliation by $(S_v)$.\\
  We define the $\emph{null lapse}$ $\Om$ of $(S_v)$ on $\HH$ by 
\begin{align}\label{eq:subt} 
  \Omt := \Lt v.
\end{align}
\end{definition}


\begin{definition}[Orthonormal null frame] 
  Let $(\Gat_{\ubt})$ be a foliation on $\HH$. Let $\Lbt$ be the unique null vector field on $\HH$ orthogonal to the 2-spheres $\Gat_{\ubt} $ and such that $\g(\Lbt,\Lt) = -2$. The pair $(\Lt, \Lbt)$ is called a \emph{null pair for the foliation $(\Gat_{\ubt})$}. Let $(e_1,e_2)$ be an orthonormal frame tangential to the 2-spheres $\Gat_{\ubt}$.
  The frame $(\Lt,\Lbt,e_1,e_2)$ is called an \emph{orthonormal null frame for the foliation $(\Gat_{\ubt})$}.
\end{definition}

\begin{definition}\label{def:projnab} Let $(\Gat_{\ubt})$ be a foliation on $\HH$. We denote by $\gd$ and $\Nd$ the induced Riemannian metric and covariant derivative on the 2-spheres $S_v$ and define for any $\Gat_{\ubt}$-tangential $k$-tensor $T$ the derivative $\Nd_\Lt T$ by
\begin{align*}
\Nd_L T_{A_1\dots A_k} := \Pi_{A_1}^{\,\,\,\,\, \be_1} \cdots \Pi_{A_k}^{\,\,\,\,\, \be_k} \D_L T_{\be_1 \dots \be_k},
\end{align*}
where $\Pi$ denotes the projection operator onto the tangent space of $\Gat$, $\D$ is the covariant derivative on $(\MM,\g)$ and we tacitly use, as in the rest of this paper, the Einstein summation convention.
\end{definition}

Here and in the following, indices $A,B,C,D,E \in \{1,2\}$ denote evaluation of $S_v$-tangent tensors on the components $(e_1,e_2)$ of an orthonormal frame $(L,\Lb,e_1,e_2)$ adapted to the foliation $(S_v)$.

\begin{definition}[Null connection coefficients]We define the \emph{null connection coefficients}, to be the $S_v$-tangent tensors such that
  \begin{align} \begin{aligned} 
    \chit_{AB} & := \g(\D_A \Lt, \etB),& \chibt_{AB} & := \g(\D_A \Lbt, \etB), \\
    \zet_A & := \half \g(\D_A \Lt, \Lbt),& \etabt_A & := \half \g(\D_\Lt \Lbt, \etA),
  \end{aligned} \label{EQdefOfRicciCoefficients} \end{align}
\end{definition}
\begin{lemma}\label{lem:zetetabt}
  The connection coefficients $\etabt$ and $\zet$ and the null lapse $\Om$ verify
  \begin{align} \label{eq:zetetabt}
    \etabt = - \zet -\Nd(\log \Omt).
  \end{align}
\end{lemma}
\begin{proof}
  We have
  \begin{align*}
    \etabt_A & = -\half \g(\Lb,\D_L\etA) \\
             & = -\half \g(\Lb,\D_AL)-\half\g(\Lb,[L,e_A]) \\
             & = -\zet_A +\Om^{-1}[L,e_A](v) \\
             & = - \zet_A -\Nd_A(\log\Om),
  \end{align*}
as desired.
\end{proof}
We have the following relations between covariant derivatives and null connection coefficients (see \cite{ChrKl93}),
\begin{align}\begin{aligned}
    \D_\Lt \Lt & = 0,& \D_\Lt \Lbt & = 2 \etabt_A \etA, \\ 
    \D_A \Lt & = \chit_{AB} \etB -\zet_A \Lt,& \D_A \Lbt & = \chibt_{AB} \etB + \zet_A \Lbt \\
    \D_\Lt \etA & = \Nd_\Lt \etA + \etabt_A \Lt, & \D_A \etB & = \Nd_A\etB + \half \chit_{AB} \Lbt + \half \chibt_{AB} \Lt.
  \end{aligned} \label{eq:Null_Id} \end{align}

If the orthonormal null frame is such that $\Nd_L\etA=0$, we call it \emph{Fermi propagated}.\\

We have the following decomposition of $\chit$ and $\chibt$ into their trace and tracefree parts
\begin{align*}
\begin{aligned}
\trchit & := \gd^{AB}\chit_{AB},& \chiht_{AB} & := \chit_{AB} -\half \trchit \gd_{AB}, \\
\trchibt & := \gd^{AB}\chibt_{AB},& \chibht_{AB} & := \chibt_{AB} - \half \trchibt \gd_{AB}.
\end{aligned}
\end{align*}

\begin{definition}[Null curvature components] We define the null curvature components to be the $S_v$-tangent tensors such that
\begin{align*} \begin{aligned}
\alphat_{AB} & := \Rbf(\Lt,\etA,\Lt,\etB),& \betat_{A} & := \half \Rbf(\etA , \Lt, \Lbt, \Lt), \\
\rhot & := \frac{1}{4} \Rbf(\Lbt,\Lt,\Lbt,\Lt),& \sigmat & := \frac{1}{4} \dual \Rbf(\Lbt,\Lt,\Lbt,\Lt), \\
\betabt_A & := \half \Rbf(\etA, \Lbt,\Lbt,\Lt),&  \alphabt_{AB} & := \Rbf(\Lbt,\etA,\Lbt,\etB),
\end{aligned}\end{align*}
where $\dual \Rbf$ denotes the Hodge dual of $\Rbf$, given by $ \dual \Rbf_{\alpha\beta\gamma\delta} = \half \bf{\in}_{\alpha\beta\mu\nu}\Rbf^{\mu\nu}_{\,\,\,\,\,\, \gamma\delta}$, with $\bf{\in}$ the volume form associated to the metric $\g$. 
\end{definition}

\subsection{Tensor calculus on $2$-surfaces} 
We introduce the following notation.
\begin{definition}[Hodge duals]
  For a $S_v$-tangent $1$-tensor $\phi$, we define its left Hodge dual by
  \begin{align*}
    \dual\phi_A := \iin_{AB}\phi_B,
  \end{align*}
  where $\iin_{AB}:=\in_{AB\Lb L}$. Similarly, for a $S_v$-tangent symmetric $2$-tensor $\phi$, let
  \begin{align*}
    \dual\phi_{AB} := \iin_{AC}\phi_{CB}.
  \end{align*}
\end{definition}

\begin{definition}[$\Gat_\ubt$-tangent tensor calculus]
For $S_v$-tangent $r$-tensors $\phi$, $\phi^{(1)}$ and $\phi^{(2)}$, we define
\begin{align*} \begin{aligned}
\phi^{(1)}\cdot\phi^{(2)} & := \gd^{A_1 B_1}\cdots \gd^{A_r B_r} \phi^{(1)}_{A_1\cdots A_r} \phi^{(2)}_{B_1\cdots B_r},&  |\phi|^2 & = \phi\cdot\phi, 
\end{aligned} \end{align*}
and
\begin{align*}
\Divd \phi_{A_2\cdots A_r} := \gd^{AB}\Nd_A\phi_{B A_2\cdots A_r} \,\, \Curld \phi_{A_2\cdots A_r} = \iin^{AB}\Nd_A\phi_{B A_2\cdots A_r}.
\end{align*}
For a $1$-form $\phi$, we define
\begin{align*}
  (\Nd\otimesh\phi)_{AB} := \Nd_A\phi_B + \Nd_B\phi_A -\Divd \gd_{AB}.
\end{align*}
For $1$-forms $\phi^{(1)}$ and $\phi^{(2)}$ we define
\begin{align*}
(\phi^{(1)}\otimesh\phi^{(2)})_{AB} :=& \phi^{(1)}_A\phi^{(2)}_B+\phi^{(1)}_B\phi^{(2)}_A - \gd_{AB} \phi^{(1)}\cdot\phi^{(2)}, \\
\phi^{(1)} \wedge \phi^{(2)} :=& \iin^{AB} \phi^{(1)}_A \phi^{(2)}_B.
\end{align*}
For symmetric 2-tensors $\phi^{(1)}$ and $\phi^{(2)}$ we define the wedge product,
\begin{align*}
(\phi^{(1)}\wedge\phi^{(2)}) & := \iin^{AB}\gd^{CD}\phi^{(1)}_{AC}\phi^{(2)}_{BD}. 
\end{align*}
\end{definition}

\subsection{Null structure equations on $\HH$}\label{sec:nullstructureeqOLD}
The Einstein vacuum equations \eqref{EinsteinVacuumEquationsIntroJ} induce the following \emph{null structure equations} on $\HH$, see (\cite{ChrKl93}, pp. 168-170).
We have the \emph{first variation equation},
  \begin{subequations}
   \begin{align*}
    \LIE_\Lt \gd = 2 \chit,
  \end{align*}
 the \emph{null transport equations},
  \begin{align*}
    \Nd_\Lt \trchit +\half (\trchit)^2 & = - |\chiht|^2, \\
    \Nd_\Lt \chiht+\trchit\chiht & =  -\alphat, \\
    \Nd_\Lt \trchibt +\half\trchit\trchibt & = 2\Divd\etabt +2\rhotc+2|\etabt|^2,\\
    \Nd_\Lt \chibht +\half\trchit\chibht & = (\Nd\otimesh\etabt)-\half\trchibt\chiht+(\etabt\otimesh\etabt),\\
    \Nd_\Lt \zet +\half\trchit\zet & = \half\trchit\etabt -\chiht\cdot(\zet-\etabt)  -\betat,                      
  \end{align*}
the \emph{torsion equation},
  \begin{align*}
    \Curld \etabt = - \Curld \zet = -\sigmat +\half \chiht\wedge\chibht ,
  \end{align*}
and the \emph{Gauss-Codazzi equations},
\begin{align*}
\Kt & = -\frac{1}{4}\trchibt\trchit - \rhot +\half\chiht\cdot\chibht,\\
\Divd \chiht - \half\Nd\trchit & = -\zet\cdot\chiht+\half\zet\trchit-\betat,\\
\Divd \chibht - \half\Nd\trchibt & = \zet\cdot\chibht-\half\zet\trchibt +\betabt,
\end{align*}
\end{subequations}
where $\Kt$ denotes the Gauss curvature of $\Gat_v$.\\

\begin{remark}
  Only the trace and symmetrised traceless part of the transport equation for $\chib$ are stated in~\cite{ChrKl93}.
  By rederiving the equation, or simply using the null transport equation for the traced and symmetrised traceless together with the torsion equation, one can rededuce the following transport equation for the full tensor $\chib$
  \begin{align*}
    \Nd_\Lt\chib_{AB} +\chi_{AC}\chib_{CB} & = 2\Nd_A\etab_B + 2\etabt_A\etabt_B + \rhot \gd_{AB} + \sigma \iin_{AB}.
  \end{align*}
\end{remark}
\begin{remark}
  Similarly, only the divergence part of Codazzi equations are stated in~\cite{ChrKl93}.
  By rederiving the equation, or simply using the Gauss-Codazzi equation for $\Divd\chih$, one can rededuce the following equation
  \begin{align*}
    \Curl \chi & = -\zet\cdot\dual\chih + \half\trchit\dual\zet - \dual\beta. 
  \end{align*}
\end{remark}

\subsection{Null Bianchi identities on $\HH$}\label{sec:bianchi}

The Einstein vacuum equations \eqref{EinsteinVacuumEquationsIntroJ} further yield the following \emph{null Bianchi identities} on $\HH$, see (\cite{ChrKl93}, p. 161).
\begin{subequations} \label{EQSnullBianchi}
\begin{align}
\Nd_\Lt\alphabt +\half\trchit\alphabt  =& -(\Nd\otimesh\betabt)-3\chibht\rhot+3\dual\chibht\sigmat +((\zet-4\etabt)\otimesh\betabt), \lab{eq:DLalphab} \\
\Nd_\Lt\betabt +\trchit\betabt=& -\Nd\rhot+\dual\Nd\sigma+2\chibht\cdot\betat-3\etabt\rho+3\dual\etabt\sigma,\lab{eq:DLbetab} \\
\Nd_\Lt\rhot +\frac{3}{2}\trchit\rhot  =& \Divd\betat-\half\chibht\cdot\alphat+\zet\cdot\betat+2\etabt\cdot\betat, \lab{eq:DLrho}\\
\Nd_\Lt\sigmat +\frac{3}{2}\trchit\sigmat  =& -\Curld \betat+\half\chibht\wedge\alphat-\zet \wedge \betat-2\etabt\wedge \betat, \lab{eq:DLsigma}\\
\Nd_\Lt\betat +2\trchit\betat  =& \Divd \alphat+(2\zet+\etabt)\cdot\alphat. \lab{eq:DLbeta}
\end{align} 
\end{subequations}

\begin{definition}[Renormalised null curvature components]\label{def:rhotc} Let the \emph{renormalised} curvature component $\check{\rhot}$, $\sigmatc$, $\betabtc$ be 
\begin{align}\label{eq:rhotc}
\check{\rhot}  := \rhot -\half\chiht\cdot\chibht, \,\,\,\,\, \sigmatc  := \sigmat -\half \chiht\wedge\chibht, \,\,\,\,\, \betabtc  := \betabt + 2\chibht\cdot\zet.
\end{align}
\end{definition}

We have the following transport equations for $\rhoc$, $\sigmatc$ and $\betabc$.
\begin{lemma} \label{lemReducedCurvatureTransportEQS} The renormalised null curvature component $\rhotc$, $\sigmatc$ and $\betabtc$ satisfy
  \begin{subequations}
  \begin{align}
    \Nd_\Lt\rhotc + \frac{3}{2} \trchit\rhotc = & \Divd\betat+\zet\cdot\betat+2\etabt\cdot\betat-\half (\Nd\otimesh\etabt)\cdot\chih+\frac{1}{4}\trchibt|\chiht|^2-\half(\etabt\otimes\etabt)\cdot \chiht,    \label{eq:DLrhoc} \\
      \Nd_\Lt\sigmatc + \frac{3}{2}\trchit\sigmatc = & -\Curld\beta-\zet\wedge\betat-2\etabt\wedge\betat-\half\chiht\wedge(\Nd\otimesh\etabt)-\half\chiht\wedge(\etabt\otimes\etabt) \\                                                              
      \Nd_\Lt\betabtc + \trchit\betabtc = & -\Nd\rhot+\dual\Nd\sigmat+2(\Nd\otimesh\etabt)\cdot\zet-3\etabt\rhot+3\dual\etabt\sigmat\label{eq:DLbetabc} \\
      & -\trchibt\zet\cdot\chiht + \trchit\etab\cdot\chibht +2\zet\cdot(\etabt\otimesh\etabt) - 2\chiht\cdot\chibht\cdot(\zet-\etabt) .\nonumber
  \end{align}
  \end{subequations}
\end{lemma}
\begin{proof}
  From p. 14 in~\cite{KlRod1}, we have the first two equations.

  Using Bianchi equation~\eqref{eq:DLbetab} and the null structure equation for $\chibh$ and $\ze$ from Section~\ref{sec:nullstructureeqOLD}, we have
  \begin{align*}
    \Nd_L\betabc + \trchit\betabc = & \Nd_L\betab + \trchit\betab +2 \left(\Nd_L\chibh+\half\trchit\chibh\right)\cdot \ze + 2 \chibh\cdot\left(\Nd_L\ze + \half\trchit\ze\right) \\
    = & -\Nd\rhot+\dual\Nd\sigma+2\chibht\cdot\betat-3\etabt\rho+3\dual\etabt\sigma \\
                                    & + 2 \ze\cdot\left(\Nd\otimesh\etab - \half \trchib\chih + \etab\otimesh\etab \right) \\
                                    & + 2 \chibh\cdot\left(\half\trchi\etab-\chih\cdot(\ze-\etab) -\beta\right) \\
    = & -\Nd\rhot+\dual\Nd\sigmat+2(\Nd\otimesh\etabt)\cdot\zet-3\etabt\rhot+3\dual\etabt\sigmat \\
      & -\trchibt\zet\cdot\chiht + \trchit\etab\cdot\chibht +2\zet\cdot(\etabt\otimesh\etabt) - 2\chiht\cdot\chibht\cdot(\zet-\etabt),
  \end{align*}
as desired.
\end{proof}
\subsection{Commutation formulas on $\HH$} The next proposition follows from p. 159 in \cite{ChrKl93}.
\begin{proposition}[Commutation formulas] \label{prop:CommutatorIdentities}\label{prop:comm}
For a $S_v$-tangent $r$-tensor $\phi$, it holds that
\begin{align}
    \begin{aligned}\label{eq:DLDB}
    [\Nd_\Lt,\Nd_B] \phi_{A_1\cdots A_r} = & -\half\trchit\Nd_B\phi_{A_1\cdots A_r} -\chiht_{BC}\Nd_C\phi_{A_1\cdots A_r} + (\etabt_B+\zet_B)\Nd_\Lt\phi_{A_1\cdots A_r}\\ 
& +\sum_{i=1}^r(\chit_{A_i B}\etabt_C-\chit_{BC}\etabt_{A_i}+\iin_{A_iC}\dual\betat_B)\phi_{A_1\cdots C \cdots A_r}.
    \end{aligned}
  \end{align}
For an $S_v$-tangent $1$-form $\phi$, it holds that
\begin{align}
[\Nd_L ,\Divd ]\phi  =& -\half\trchit\Divd\phi - \chiht\cdot\Nd\phi +(\etabt+\zet)\cdot\Nd_\Lt\phi \label{eq:DLDivd} \\
    & + \trchit\etabt\cdot\phi-\etabt_A\phi_C\chit_{AC}+\betat\cdot\phi, \nonumber, \\
[\Nd_A,\Ld]\phi_B = & -3\Kt\Nd_A\phi_B + 2\gd_{AB}K\Divd\phi +2\iin_{AB}K\Curld\phi \label{eq:NdDelta1}\\
     & +\gd_{AB}\phi\cdot K - \phi_A\Nd_BK. \nonumber
\end{align}
where $K$ denotes the Gauss curvature of $S_v$. For a scalar function $\phi$, it holds that
\begin{align}
[\Nd_\Lt,\Nd_B]\phi  =& -\half\trchit \Nd_B\phi - \chiht_{BC}\Nd_C\phi + (\etabt_B+\zet_B)\Lt\phi,\label{eq:DLDBscal} \\
[{\Nd_L},\Ld] \phi = & -\trchit \Ld\phi -2\chit\cdot\Nd^2\phi+(\etabt+\zet)\cdot(\Nd\Nd_\Lt+\Nd_\Lt\Nd)\phi \label{eq:DLDelta}\\
& +(\trchit\etabt-\Divd\chit)\cdot\Nd\phi-\etabt_A\chit_{AB}\Nd_B\phi \nonumber\\
& +(\Divd\etabt+\Divd\zet)\Nd_\Lt\phi+\betat\cdot\Nd\phi, \nonumber \\      
[\Nd_A,\Ld]\phi =& -\Kt \Nd_A\phi. \label{eq:NdDeltascal}
\end{align}
\end{proposition}

\begin{proposition}\label{prop:Loverline}
  For a scalar function $f$ on $\HH$, it holds that
  \begin{align*}
    \Om^{-1}L\left(\overline{f}\right) & = \overline{\Om^{-1}Lf} +\overline{\Om^{-1}\trchit f} - \overline{\Om^{-1}\trchit}\cdot\overline{f},
  \end{align*}
  where $\overline{f}$ denotes the mean value of $f$ on $S_v$.
\end{proposition}
\begin{proof}
  Using the null structure equations from Section~\ref{sec:nullstructureeqOLD}, we have
  \begin{align*}
    \Om^{-1}L\left(\int_{S_v}f\right) = & \frac{\text{d}}{\text{d}v}\left(\int_{S_v}f\right) \\
    = & \int_{S_v}\Om^{-1}\left(Lf+\trchit f \right).
  \end{align*}
  We therefore deduce
  \begin{align*}
    \Om^{-1}L(\overline{f}) = & -\Om^{-1}L\left(\log\left(\left|S_v\right|\right)\right) \overline{f} + \overline{\Om^{-1}\left(Lf+\trchit f\right)} \\
    = & - \overline{\Om^{-1}\trchit}\cdot\overline{f} + \overline{\Om^{-1}\left(Lf+\trchit f\right)},
  \end{align*}
  as desired.
\end{proof}

\subsection{The mass aspect function on $\HH$}
\begin{definition}
Let the \emph{mass aspect function} $\mut$ on $\HH$ be defined by
  \begin{align} \begin{aligned}
    \mut := &  -\rhotc - \Divd\zet.
  \end{aligned}     \label{def:mu} \end{align}
\end{definition}

We have the following transport equation for $\mu$.
\begin{lemma} \label{eq:DLmu}
The mass aspect function $\mut$ verifies
\begin{align*}
  L(\mut) +\trchit\mut = & \half\trchit\rhotc -\half\trchit\Divd\etabt -2\zet\cdot\beta + (\zet-\etabt)\cdot\Nd\trchit + \chih\cdot\Nd\zet \\
                         & +\half\chih\cdot\Nd\etabt + \trchit\left(|\ze|^2-\zet\cdot\etabt - \half|\etabt|^2\right) - \frac{1}{4}\trchibt|\chih|^2 \\
                         & +2\chih\cdot\zet\cdot\etabt -\half\chih\cdot\etabt\cdot\etabt.
\end{align*}
\end{lemma}
\begin{proof}
  Using the transport equation for $\ze$ from the null structure equations from Section~\ref{sec:nullstructureeqOLD} and commutation formula~\eqref{eq:DLDivd}, we have
  \begin{align*}
    L(\Divd\zet) = & \Divd \Nd_L\ze + [\Nd_L,\Divd]\zet \\
    = & \Divd\bigg(-\half\trchi\ze + \half\trchi\etabt-\chih\cdot(\ze-\etabt) -\beta\bigg) \\
                   & -\half\trchi\Divd\ze - \chih\cdot\Nd\ze + (\etabt+\ze)\cdot\Nd_L\ze + \trchit\etabt\cdot\ze -\etabt\cdot\ze\cdot\chi + \beta\cdot\ze \\
    = & - \trchi \Divd\ze +\half\trchit\Divd\etabt - \Divd\beta + F_1,
  \end{align*}
  where
  \begin{align*}
    F_1 := & -\half \ze\cdot\Nd\trchi +\half\etabt\cdot\Nd\trchi - \Divd\chih\cdot(\ze-\etabt) - \chih\cdot(\Nd\ze-\Nd\etabt) \\
           & - \chih\cdot\Nd\ze + (\etabt+\ze)\cdot\bigg(-\half\trchit\ze + \half\trchi\etabt - \chih\cdot(\ze-\etabt) -\beta\bigg) \\
    & + \trchit\etabt\cdot\ze -\etabt\cdot\ze\cdot\chi +\beta\cdot\ze.
  \end{align*}
  From the above and transport equation~\eqref{eq:DLrhoc} for $\rhotc$, we have
  \begin{align*}
    L(\mu) + \trchi\mu  = & -L(\rhoc) -\trchi\rhoc - L(\Divd\ze) - \trchi\Divd\ze \\
                        = & \frac{3}{2}\trchi\rhoc - \Divd\beta + F_2 - \trchit\rhoc \\
                          & - \half\trchit\Divd\etabt + \Divd\beta - F_1 \\
    = & \half\trchit\rhoc - \half\trchit\Divd\etabt + F_2 - F_1,
  \end{align*}
  where the nonlinear term $F_2$ is given by
  \begin{align*}
    F_2:= - \ze\cdot\beta -2\etabt\cdot\beta +\half(\Nd\etabt)\cdot\chih -\frac{1}{4}\trchibt |\chih|^2 + \half \etabt\cdot\etabt\cdot\chih.
  \end{align*}
  Rearranging the total nonlinear term $F_2-F_1$ then gives the desired result.
\end{proof}

\subsection{The canonical foliation on $\HH$} \label{sec:HNL}
\begin{definition}[Canonical foliation]\label{defHNLfoliation} Let $\HH$ be an outgoing null hypersurface emanating from a spacelike $2$-sphere $S$. A foliation $(\Gat_\ubt)$ of $\HH$ is called a \emph{canonical foliation} if the null lapse $\Om$ satisfies the following elliptic equation on the leaves $S_v$
  \begin{align}
    \begin{aligned}\label{eq:CMAF}
      \Ld(\log\Omt) & = -\Divd\zet +\rhotc-\rhotco, \\
      \int_{S_v} \log\Omt & = 0,
  \end{aligned}
  \end{align}
  where $\rhotco$ is the average of $\rhotc$ on the $2$-sphere $S_v$.\\
  In the following, we will moreover consider canonical foliations that coincide with the geodesic foliation on $S$, \emph{i.e.} $v\vert_{S}=1$.
\end{definition}

\begin{remark}
  Using \eqref{def:mu}, the elliptic equation~\eqref{eq:CMAF} can also be rewritten
  \begin{align}\label{eq:CMAFbis}
    \begin{aligned}
      \Ld(\log\Om) & = -2\Divd\zet -\mu + \muo, \\
      & = 2(\rhotc-\rhotco) + \mu-\muo.
    \end{aligned}
  \end{align}
\end{remark}
\textbf{Notation.} From now on, primed quantities on $\HH$ will correspond to the geodesic foliation of $\HH$, while unprimed quantities correspond to the canonical foliation. Moreover, we call $S_1 = S = S'_1$.\\



As a first consequence of Definition~\ref{defHNLfoliation}, we note that in a canonical foliation, the quantities $\etabt$ and $\trchib$ satisfy the following equations.

\begin{lemma}\label{lem:Divdetab}In a canonical foliation, $\etabt$ satisfies the following equation
  \begin{align}
    \Divd\etabt = -\rhotc + \rhotco.
  \end{align}
\end{lemma}
\begin{proof}
  Using relation~\eqref{eq:zetetabt} and the elliptic equation~\eqref{eq:CMAF}, we have
  \begin{align*}
    \Divd\etabt & = -\Divd\zet - \Ld(\log\Om) \\
                & = -\rhotc +\rhotco,
  \end{align*}
  as desired.
\end{proof}

\begin{lemma} \label{prop:transportequationfortrchibt} \label{eq:DLtrchibtCTMAF} 
In a canonical foliation, $\trchib$ satisfies the next null transport equation
\begin{align} 
L(\trchibt) +\half \trchit\trchibt & = 2\rhotco+2|\etabt|^2.
\end{align}
\end{lemma}
\begin{proof} Using the transport equation for $\trchib$ from Section~\ref{sec:nullstructureeqOLD} and relation from Lemma~\ref{lem:Divdetab}, we have
\begin{align*}
\Lt(\trchibt) +\half\trchit\trchibt & = 2\Divd\etabt+2\rhotc+2|\etabt|^2\\
                                        & = 2\rhotco +2|\etabt|^2,
\end{align*}
as desired.
\end{proof}

\begin{lemma}\label{lem:DLmuc}
  In a canonical foliation, the transport equation for the mass aspect function can be written under the following form, where we notice that on the right-hand side there are no quadratic terms involving $\chibh$ and the only linear term involves $\rhoc$
  \begin{align}\label{eq:DLmuc}
    \begin{aligned}
    L(\mut) +\trchit\mut = & \trchit\rhotc -\half\trchit\rhotco -2\zet\cdot\beta + (\zet-\etabt)\cdot\Nd\trchit + \chih\cdot\Nd\zet \\
                           & +\half\chih\cdot\Nd\etabt + \trchit\left(|\ze|^2-\zet\cdot\etabt - \half|\etabt|^2\right) - \frac{1}{4}\trchibt|\chih|^2 \\
                           & +2\chih\cdot\zet\cdot\etabt -\half\chih\cdot\etabt\cdot\etabt.    
                         \end{aligned}
  \end{align}
\end{lemma}

For convenience, we summarise the full \emph{null structure equations in a canonical foliation}
  \begin{subequations}
    \begin{align}
    \LIE_\Lt \gd & = 2 \chit,\label{eq:firstvar} \\
    \Nd_\Lt \trchit +\half(\trchit)^2 & = - |\chiht|^2, \label{eq:DLtrchi} \\
    \Nd_\Lt \chiht +\trchit\chiht & = -\alphat, \label{eq:DLchih}\\
    \Nd_\Lt\chib +\chi\cdot\chib & = 2\Nd\etab + 2\etabt\,\etabt + \rhot \gd + \sigma \iin, \label{eq:DLchib}\\
    \Nd_\Lt \trchibt +\half\trchit\trchibt & = 2\rhotco+2|\etabt|^2, \label{eq:DLtrchibt} \\
    \Nd_\Lt \chibht +\half\trchit\chibht & = (\Nd\otimesh\etabt)-\half\trchibt\chiht+(\etabt\otimesh\etabt), \label{eq:DLchibh} \\
    \Nd_\Lt \zet + \half \trchit\zet & = \half\trchit\etabt +\chiht\cdot(\etabt -\zet)-\betat, \label{eq:DLze}\\                       
    \Curld \etabt &  = - \Curld \zet = -\sigmatc, \label{eq:Curleta} \\
    \Kt & = -\frac{1}{4}\trchibt\trchit - \rhotc, \label{eq:Gauss} \\
    \Divd \chiht - \half\Nd\trchit & = -\zet\cdot\chiht+\half\zet\trchit-\betat, \label{eq:Divdchih}\\
    \Divd \chibht - \half\Nd\trchibt & = \zet\cdot\chibht-\half\zet\trchibt +\betabt, \label{eq:divchibh} \\
      \Divd \etabt & = -\rhotc +\rhotco, \label{eq:Divdetab}\\
      \Curl \chi & = -\zet\cdot\dual\chih + \half\trchit\dual\zet - \dual\beta,\label{eq:Curlchi}\\
    \Ld(\log\Omt) & = -\Divd\zet +\rhotc-\rhotco.\label{eq:dampedHNL}
  \end{align}
\end{subequations}

\subsection{Comparison of foliations}
In this section, we derive equations that are used to compare a geodesic foliation and a canonical foliation starting from a common sphere $S$.\\
We first introduce the derivative of the geodesic parameter $s$ in the canonical foliation.
\begin{definition}
  We define the $S_v$-tangent 1-form $\Ups$ to be
  \begin{align*}
    \Ups := \Nd s.
  \end{align*}
\end{definition}

  

We have the following proposition.
\begin{proposition}[Null frame comparison]
  Let $(e'_A)_{A=1,2}$ and $(e_A)_{A=1,2}$ be Fermi propagated null frames respectively for the geodesic and the canonical foliation, such that $e'_A=e_A$ on $S$. For $A=1,2$, it holds that
  \begin{align}\label{eq:etAcomp}
    \etA' & = \etA - \Ups_A\Lt,
  \end{align}
  and 
  \begin{align}\label{eq:Lbcomp}
    \Lbt' & =  \Lbt - 2\Ups_A \etA +|\Ups|^2 \Lt.
  \end{align}
\end{proposition}
\begin{proof}
  It is straight-forward to verify that the vectors $\etA-\Ups_A\Lt$, $A=1,2$, are $S'_s$-tangent vector fields that coincide with $\etA'$ on $S$.
  Moreover, for any $X\in TS'_s$, we have
  \begin{align*}
    \g(X,\D_L(\etA-\Ups_A\Lt)) & = \g(X,\etab_AL - L(\Ups_A)L) \\
                                 & = 0.
  \end{align*}
  Thus, the vectors $\etA-\Ups_A\Lt$ are Fermi propagated with respect to the geodesic foliation and we deduce that they coincide with $\etA'$ on $\HH$.

  One then directly checks that the vector field $$Z:=\Lbt-2\Ups_A\etA + |\Ups|^2\Lt$$ satisfies $g(Z,\etA') = g(Z,Z) = 0$ and $g(Z,L) = -2$.
\end{proof}

We have the following definition of the projection of tensors from one foliation to another, see also Section 2.2 in~\cite{Alexakis}. 
\begin{definition}\label{def:compfol}
  Let $\phi'$ be a $S'_s$-tangent $r$-tensor. We define the projection $(\phi')^\dg$ to be the $S_v$-tangent $r$-tensor defined by
  \begin{align}
    (\phi')^\dg_{A_1\cdots A_r} = (\phi')^\dg\big(e_{A_1},\cdots,e_{A_r}\big) := \phi'\big(e'_{A_1},\cdots,e'_{A_r}\big) = \phi'_{A_1 \cdots A_r}.
  \end{align}
  Reciprocally, for $\phi$ a $S_v$-tangent $r$-tensor, we define the projection $(\phi)^\ddg$ to be the $S'_S$-tangent $r$-tensor defined by
  \begin{align}
    (\phi)^\ddg_{A_1 \cdots A_r} = (\phi)^\ddg\big(e'_{A_1},\cdots,e'_{A_r}\big) := \phi\big(e_{A_1},\cdots,e_{A_r}\big) = \phi_{A_1 \cdots A_r}.
  \end{align}
\end{definition}

We introduce the following projection of $\Ups$.
\begin{definition}\label{def:Ups'}
  We define the $S'_s$-tangent $1$-form $\Ups'$ to be
  \begin{align*}
    \Ups' := -(\Ups)^\ddg.
  \end{align*}
\end{definition}

We have the following relation between $\Ups'$ and the derivative of $v$ in the geodesic foliation.
\begin{lemma}\label{lem:UpsUps'}
  We have
  \begin{align*}
    \Ups' = \Om^{-1}\Nd'v.
  \end{align*}
\end{lemma}
\begin{proof}
  We have
  \begin{align*}
    \Nd'_Av = (\etA-\Ups_A\Lt)v = -\Om\Ups_A = \Om \Ups'_A, 
  \end{align*}
  as desired
\end{proof}

We have the following correspondences for $S_v$-tangential derivatives of projected tensors.
\begin{proposition}[Projection calculus] \label{prop:compNd} Let $r\geq1$ be an integer. Let $\phi'$ be an $S'_s$-tangent $r$-tensor. Then it holds that
  \begin{align}\label{eq:compNd}
    \begin{aligned}
    \Nd_L(\phi')^\dg_{A_1\cdots A_r} = & (\Nd'_L\phi')^\dg_{A_1\cdots A_r},\\
    \Nd_A(\phi')^\dg_{A_1\cdots A_r} = & (\Nd'\phi')^\dg_{A A_1 \cdots A_r} + \Ups_A(\Nd'_L\phi')^\dg_{A_1 \cdots A_r} \\
    & + \chi_{A A_i} \Ups_B (\phi')^\dg_{A_1 \cdots B \cdots A_r} - \chi_{AB}\Ups_{A_i}(\phi')^\dg_{A_1 \cdots B \cdots A_r}
  \end{aligned}
  \end{align}
 Similarly, for a given $S_v$-tangent $r$-tensor $\phi$, it holds that
 \begin{align}\label{eq:compNd2}
   \begin{aligned}
    \Nd'_L\phi^\ddg_{A_1\cdots A_r} = & (\Nd_L\phi)^\ddg_{A_1\cdots A_r},\\
    \Nd'_A\phi^\ddg_{A_1\cdots A_r} = & (\Nd\phi)^\ddg_{A A_1 \cdots A_r} + \Ups'_A(\Nd_L\phi)^\ddg_{A_1 \cdots A_r} \\
    &  + \chi'_{A A_i}\Ups'_B (\phi)^\ddg_{A_1 \cdots B \cdots A_r} - {\chi'}_{AB}\Ups'_{A_i}(\phi)^\ddg_{A_1 \cdots B \cdots A_r}.
  \end{aligned}
  \end{align}
\end{proposition}
\begin{proof}
  First, using that both frames $(e_1,e_2)$ and $(e'_1,e'_2)$ are Fermi propagated, we have
  \begin{align*}
    \Nd_L(\phi')^\dg_{A_1 \cdots A_r} & = L\big((\phi')^\dg(e_{A_1},\cdots,e_{A_r})\big) \\
                                      & = L\big(\phi'(e'_{A_1},\cdots,e'_{A_r})\big) \\
                                      & = \Nd'_L\phi'(e'_{A_1},\cdots,e'_{A_r}) \\
                                      & = (\Nd'_L\phi')^\dg(e_{A_1},\cdots,e_{A_r}) \\
                                      & = (\Nd'_L\phi')^\dg_{A_1 \cdots A_r}.
  \end{align*}

  Second,
  \begin{align*}
    \Nd_A(\phi')^\dg_{A_1\cdots A_r} = & e_A\big(\phi'_{A_1\cdots A_r}\big) - (\phi')^\dg\big(e_{A_1},\cdots, \Nd_Ae_{A_i},\cdots,e_{A_r}\big),
  \end{align*}
  and
  \begin{align*}
    e_A\big(\phi'_{A_1\cdots A_r}\big) = & e'_A\big(\phi'_{A_1\cdots A_r}\big) + \Ups_AL\big(\phi'_{A_1\cdots A_r}\big) \\
    = & \Nd'_A\phi'_{A_1 \cdots A_r} + \phi'\big(e'_{A_1},\cdots,\Nd'_Ae'_{A_i},\cdots,e'_{A_r}\big) + \Ups_A\Nd'_L\phi'_{A_1\cdots A_r}.
  \end{align*}
  Therefore, we deduce
  \begin{align} \begin{aligned}
    \Nd_A(\phi')^\dg_{A_1\cdots A_r} = & (\Nd'\phi')^\dg_{A A_1 \cdots A_r} +\Ups_A(\Nd'_L\phi')^\dg_{A_1\cdots A_r} \\
                                       & + \phi'\big(e'_{A_1},\cdots, \Nd'_Ae'_{A_i} - g(\Nd_Ae_{A_i},e'_B)e'_B,\cdots,e'_{A_r}\big).
  \end{aligned} \label{EQAppComparison22} \end{align}
  To compute the third term on the right-hand side of \eqref{EQAppComparison22}, write
  \begin{align*}
    g(\Nd_Ae_{A_i},e'_B) = & g\bigg(\D_Ae_{A_i}-\half\chi_{AA_i}\Lb-\half\chib_{AA_i}L,e'_B\bigg) \\
    = & g(\D_Ae_{A_i},e'_B) - \chi_{AA_i}\Ups_B \\
    = & g(\D_{e'_A-\Ups'_A L}(e'_{A_i}+\Ups_{A_i}L),e'_B) - \chi_{AA_i}\Ups_B \\
    = & g(\D_{e'_A}e'_{A_i},e'_B) +\Ups_{A_i}g(\D_{e'_A}L,e'_B) - \chi_{AA_i}\Ups_B \\
    = & g(\D_{e'_A}e'_{A_i},e'_B) +\Ups_{A_i}\chi'_{AB}- \chi_{AA_i}\Ups_B. 
  \end{align*}
  where we used \eqref{eq:Null_Id}, \eqref{eq:Lbcomp} and the fact that both frames are Fermi propagated.
  Moreover, it follows from \eqref{eq:etAcomp} that $\chi'_{AA_i} = \chi_{AA_i}$, and therefore
    \begin{align} \begin{aligned}
    & \phi'\big(e'_{A_1},\cdots,\Nd'_Ae'_{A_i}- g(\Nd_Ae_{A_i},e'_B)e'_B,\cdots,e'_{A_r}\big) \\
    & = -\Ups_{A_i}\chi_{AB}(\phi')^\dg_{A_1 \cdots B \cdots A_r}  + \chi_{AA_i}\Ups_B(\phi')^\dg_{A_1 \cdots A_r}.
  \end{aligned} \label{EQAppComparison222} \end{align}
Plugging \eqref{EQAppComparison222} into \eqref{EQAppComparison22} concludes the proof of \eqref{eq:compNd}.
  In view of \eqref{eq:etAcomp} and Lemma~\ref{lem:UpsUps'}, the proof of \eqref{eq:compNd2} follows by replacing $\Ups_A$ by $\Ups'_A$. This finishes the proof of Proposition \ref{prop:compNd}. \end{proof}

We have the following transport equation for $\Ups$.
\begin{lemma}\label{lem:DLUps}
  We have
  \begin{align}\label{eq:DLUps}
    \Nd_L\Ups = & -\Nd(\log\Om)-\chi\cdot\Ups.  
  \end{align}
\end{lemma}
\begin{proof}
  Using commutation formula~\eqref{eq:DLDB} and $Ls =1$, we have
  \begin{align*}
    \Nd_L\Ups_A = & \Nd_L\Nd s \\
    = & \Nd L(s) + [\Nd_L,\Nd]_A s \\
    = & -\Nd_A(\log\Om)-\chi_{AB}\Nd_Bs,
  \end{align*}
  as desired.
\end{proof}


We have the following comparison between null curvature components and connection coefficients and projected null curvature components and connections coefficients. The proofs are postponed to Appendix~\ref{app:RR'AA'}.
\begin{proposition}[Null curvature component comparison]\label{prop:RR'}
  The following relations hold.
  \begin{align}\label{eq:RR'}
    \begin{aligned}
      \alpha_{AB} =& (\alpha')^\dg_{AB}, \\
      \beta_A =& (\beta')^\dg_A + \Ups_B(\alpha')^\dg_{AB}, \\
      \rho =& \rho' +\Ups\cdot(\beta')^\dg + \Ups\cdot\Ups\cdot(\alpha')^\dg, \\
      \sigma =& \sigma' -\Ups\cdot(\dual\beta')^\dg -\Ups\cdot\Ups'\cdot(\dual\alpha')^\dg, \\
      \betab_A = & (\betab')^\dg_A -3\rho'\Ups_A + 3\sigma'\dual\Ups_A -2 \left(\Ups\cdot(\dual\beta')^\dg\right)\dual\Ups_A \\
      & + |\Ups|^2\beta'_A -2\left(\Ups\cdot\Ups\cdot(\alpha')^\dg\right)\Ups_A + |\Ups|^2\Ups\cdot(\alpha')^\dg_A.
    \end{aligned}
  \end{align}
\end{proposition}

\begin{proposition}[Null connection coefficients comparison]\label{prop:AA'}
  The following relations hold.
  \begin{align}\label{eq:AA'}
    \begin{aligned}
      \chi_{AB} = & (\chi')^\dg_{AB}, \\
      \zet_A = & (\zet')^\dg_A +(\chi')^\dg_{AB}\Ups_B,  \\
      \etabt_A = & (\etabt')^\dg_A + \Nd_L\Ups_A, \\
      \chib_{AB} = & (\chib')^\dg_{AB}+ 2\Ups_A(\etab')^\dg_B -2\Ups_B(\zet')^\dg_A + 2\Nd_A\Ups_{B} - |\Ups|^2\chi'_{AB}.
    \end{aligned}
  \end{align}
\end{proposition}
\subsection{Norms on $\HH$} 

In this section, we define norms on $\HH$. Throughout this section, we denote by $(\Stt_\vtt)_{1\leq\vtt\leq v^\ast}$ either the geodesic foliation $(S'_s)$ or the canonical foliation $(S_v)$.

\begin{definition}[$\Stt_\vtt$-mixed norms]
  Let $v^\ast\geq 1$. Let $F$ be an $\Stt_\vtt$-tangent tensor. We define the mixed norms on $\HH$ with respect to the foliation $(\Stt_\vtt)_{\vtt\in[1,v^\ast]}$,
  \begin{align*}
    \norm{F}_{L^p_\vtt ([1,v^\ast]) L^q} := & \bigg(\int_{1}^{v^\ast} \norm{F}_{L^q(\Stt_\vtt)}^p\,\text{d}\vt\bigg)^{\frac{1}{p}}, \\
    \norm{F}_{L^q L^p_{\vtt}([1,v^\ast])} := & \norm{\bigg(\int_{1}^{v^\ast}|F|^p\,\text{d}\vt\bigg)^{\frac{1}{p}}}_{L^q(S_1)}.
  \end{align*}
\end{definition}

\begin{definition}
  Let $v^\ast\geq 1$ and let $F$ be an $\Stt_\vtt$-tangent tensor. Define  
  \begin{align*}
    \NN_1^{\vtt,[1,v^\ast]}(F) & := \norm{F}_{H^{1/2}(S_1)} + \norm{F}_{L^2_{\vtt}([1,v^\ast]) L^2} + \norm{\widetilde{\Nd} F}_{L^2_{\vtt}([1,v^\ast]) L^2} + \norm{\widetilde{\Nd}_L F}_{L^2_{\vtt}([1,v^\ast]) L^2},
  \end{align*}
  where $\widetilde{\Nd}$ and $\widetilde{\Nd}_{\Lt}$ denote the induced covariant derivatives on $\Stt_\vtt$. Moreover, for $m\geq 1$, define 
  \begin{align*}
    \NN_m^{\vtt,[1,v^\ast]}(F) & := \sum_{k\leq m-1} \norm{\widetilde{\nab}^k F}_{H^{1/2}(S_1)} + \sum_{k\leq m}\norm{\widetilde{\nab}^k F}_{L^2_{\vtt}([1,v^\ast]) L^2},
  \end{align*}
  where $\widetilde{\nab} \in \{ \widetilde{\Nd}, \widetilde{\Nd}_L\}$.
  We refer to Section \ref{sec:LPtheory} for a precise definition of the space of tensors $H^{1/2}(S_1)$.
\end{definition}


\subsection{Weak regularity of $S_1$.}
In this section, we define the weak regularity assumption on $S_1$, see~\cite{ShaoBesov} Section 2.4.
\begin{definition}\label{def:weakreg}
  A Riemaniann $2$-sphere $(S,\gd)$ is \emph{weakly regular} with constants $N,c$ if it can be covered by $N$ coordinate patches $(x^1,x^2)$ with a partition of unity $\eta$ adapted to the coordinate patches and with functions $0\leq \etatt \leq 1$ that are compactly supported in the patches and equal to $1$ on the support of $\eta$, and if on each patch there exists an orthonormal frame $(e_1,e_2)$ such that for $a,b=1,2$ and $A=1,2$, 
  \begin{gather}
    c^{-1} \leq \sqrt{\det{\gd}} \leq c, \\
    c^{-1} \big((\xi^1)^2 +(\xi^2)^2\big) \leq \gd_{ab}\xi^a\xi^b \leq c \big((\xi^1)^2 +(\xi^2)^2\big), ~\forall(\xi^1,\xi^2)\in\RRR^2,\\
    |\pr_{x^a}\eta| + |\pr_{x^a}\pr_{x^b}\eta| + |\pr_{x^a} \etatt |\leq c, \\
    \norm{\Nd\pr_{x^a}}_{L^2(S)} + \norm{\Nd\etA}_{L^4(S)} \leq c.
  \end{gather}
\end{definition}

\subsection{Norms for the geodesic and canonical foliation geometry} \label{SECdefinitionOfNorms}
In this section, we introduce norms to measure the geometry of the geodesic foliation and the canonical foliation on $\HH$ at the level of bounded $L^2$ curvature. The definition of the Besov spaces $B^{0}(S)$ is postponed to Section \ref{sec:LPtheory}.\\

\paragraph{\bf {Norms for null connection coefficients of the geodesic foliation on $S_1$.}} 
\begin{align*} 
  \IIt_{S_1}' :=& \norm{\trchit'-2}_{L^\infty(S_1)} + \norm{\Nd \trchit'}_{B^{0}(S_1)} + \norm{\trchibt'+2}_{L^\infty(S_1)} + \norm{\Nd \trchibt'}_{L^2(S_1)}\\
  &+ \norm{\mut'}_{B^{0}(S_1)} +\norm{\zet'}_{H^{1/2}(S_1)} + \norm{\chih'}_{H^{1/2}(S_1)} + \norm{\chibh'}_{H^{1/2}(S_1)}.
\end{align*}

\paragraph{\bf{Norms for null connection coefficients of the canonical foliation on $S_1$.}}  
      \begin{align*}
  \IIt_{S_1} := & \norm{\trchit-2}_{L^\infty(S_1)}+\norm{\trchibt+2}_{L^\infty(S_1)}  + \norm{\Nd\trchit}_{B^{0}(S_1)} + \norm{\Nd\trchibt}_{L^2(S_1)} \\
 &+ \norm{\mut}_{B^{0}(S_1)} + \norm{\zet}_{H^{1/2}(S_1)}  + \norm{\chih}_{H^{1/2}(S_1)} + \norm{\chibh}_{H^{1/2}(S_1)} \\
                 & + \norm{\Nd\log\Om}_{H^{1/2}(S_1)} + \norm{\etabt}_{H^{1/2}(S_1)} + \norm{\log\Om}_{L^2(S_1)} + \norm{\Om-1}_{L^\infty(S_1)} + \norm{\muc}_{L^2(S_1)}.
\end{align*}

\paragraph{\bf{Norms for null connection coefficients of the geodesic foliation on $\HH$.}} 
\begin{align*}
  \OOt_{[1,s^\ast]}' := & \norm{\trchit'-\frac{2}{s}}_{L^\infty_{s}([1,s^\ast])L^\infty} + \norm{\chih'}_{L^\infty L^2_{s}([1,s^\ast])} + \norm{\zet'}_{L^\infty L^2_{s}([1,s^\ast])} \\
                          & + \NN^{s,[1,s^\ast]}_1\left(\trchit'-\frac{2}{s} \right) + \NN^{s,[1,s^\ast]}_1(\chih') + \NN^{s,[1,s^\ast]}_1(\zet').
\end{align*}

\paragraph{\bf{Norms for null connection coefficients of the canonical foliation on $\HH$.}} 
\begin{align*}
  \OOt_{[1,v^\ast]} := & \NN_1^{v,[1,v^\ast]}\left(\trchit-\frac{2}{v} \right) + \NN_1^{v,[1,v^\ast]}(\chih) + \NN_1^{v,[1,v^\ast]}(\zet) + \NN_1^{v,[1,v^\ast]}(\etabt) \\
                       & + \NN_1^{v,[1,v^\ast]}\left(\trchibt+\frac{2}{v}\right) + \NN_1^{v,[1,v^\ast]}(\chibh)\\
                       & + \norm{\Om-1}_{L^\infty_{v}([1,,v^\ast])L^\infty} + \norm{L(\log\Om)}_{L^2_{v}([1,,v^\ast])L^4} + \NN_1^{v,[1,v^\ast]}(\Nd\log\Om)\\
                       & + \norm{\trchit-\frac{2}{v}}_{L^\infty_{v}([1,v^\ast])L^\infty} + \norm{\chih}_{L^\infty L^2_{v}([1,v^\ast])} + \norm{\zet}_{L^\infty L^2_{v}([1,v^\ast])} \\
                       & + \norm{\etabt}_{L^\infty L^2_{v}([1,v^\ast])} + \norm{\trchibt+\frac{2}{v}}_{L^\infty_{v}([1,v^\ast])L^\infty}\\
                       & + \norm{\Nd\trchibt}_{L^2 L^\infty_{v}([1,v^\ast])} + \norm{\muc}_{L^2 L^\infty_{v}([1,v^\ast])} + \norm{\Nd\trchibt}_{L^2 L^\infty_{v}([1,v^\ast])}.
\end{align*}

\paragraph{\bf{Norms for null curvature components of the geodesic foliation on $\HH$.}}  
\begin{align*}
  \RRt_{[1,s^\ast]}'  := & \norm{\alphat'}_{L^2_{s}([1,s^\ast])L^2}+\norm{\betat'}_{L^2_{s}([1,s^\ast])L^2}+\norm{\rhot'}_{L^2_{s}([1,s^\ast])L^2}\\
                         & +\norm{\sigmat'}_{L^2_{s}([1,s^\ast])L^2}+\norm{\betabt'}_{L^2_{s}([1,s^\ast])L^2}.
\end{align*}

\paragraph{\bf{Norms for null curvature components of the canonical foliation on $\HH$.}}
 \begin{align*}
   \RR_{[1,v^\ast]} := & \norm{\alphat}_{L^2_{v}([1,v^\ast])L^2}+\norm{\betat}_{L^2_{v}([1,v^\ast])L^2}+\norm{\rhot}_{L^2_{v}([1,v^\ast])L^2}\\
                       & +\norm{\sigmat}_{L^2_{v}([1,v^\ast])L^2}+\norm{\betabt}_{L^2_{v}([1,v^\ast])L^2}.
\end{align*}


\subsection{Main results}  \label{sec:StatementMainResultandNORMS} The following is the main result of this paper.
     
\begin{theorem}[Existence and control of the canonical foliation, version 2] \label{TheoremMassLapseFoliationControl} Let $(\MM,\g)$ be a smooth spacetime and $\HH$ be a smooth null hypersurface emanating from a spacelike $2$-sphere $S$. Assume that the geodesic foliation $(S'_s)$ starting at $S$ with $s=1$ is well-defined and smooth up to $s=5/2$. Let $0<N,c<\infty$ and assume that $S$ is weakly regular with constants $N,c$. Assume moreover that for some $\varep>0$,
\begin{align}\label{est:boundedL2source}
 \IIt_{S_1}' + \RRt_{[1,5/2]}' + \OOt_{[1,5/2]}'  \leq \varep.
\end{align}
Then, there is a universal constant $\varep_0>0$ such that if $0<\varep<\varep_0$, the following holds.
\begin{enumerate}
\item {\bf Existence of the canonical foliation.} The canonical foliation (see Definition \ref{eq:CMAF}) is well-defined from $v=1$ to $v = 2$, and we have the following comparison estimate with respect to the geodesic foliation,
  \begin{align}\label{est:compvL2} 
    \Vert \Om - 1 \Vert_{L^\infty_{v}([1,2])L^\infty} \les \varep, \,\, \norm{\Ups}_{L^\infty_{v}([1,2])L^\infty} \les \varep. 
  \end{align}
\item {\bf $L^2$-regularity.} There is a constant $C=C(N,c)>0$ such that the canonical foliation is uniformly weakly regular with constants $N,C$ for $v=1$ to $v=2$, and moreover,
  \begin{align} \label{est:boundedL2conn} 
    \IIt_{S_1} +\RR_{[1,2]} + \OOt_{[1,2]}  \les \varep.
  \end{align}
\item {\bf Smoothness.} The canonical foliation is smooth up to $v=2$.
\end{enumerate}
\end{theorem}

{\bf Remarks.}
\begin{enumerate}
\item The comparison estimate~\eqref{est:compvL2} implies in particular that $|s-v| \les \varep$, so that both foliations remain close.

\item Using the conclusions of~\cite{KlRod1} (see also Theorem~\ref{THMKlRodrough}), we have that under the assumption $\II'_{S_1}+\RR'_{[1,5/2]}\leq \varep$, control of the geodesic connection coefficient norm $\OO'_{[1,5/2]} \les \varep$ can be obtained. The smallness hypothesis~\eqref{est:boundedL2source} can therefore be replaced by $\II'_{S_1}+\RR'_{[1,5/2]}\leq \varep$ and involves only $L^2$-norms of curvature components on $\HH$ and low regularity connection coefficients bounds on $S_1$.

\item Using the a priori estimates from the previous remark, together with an assumption on the injectivity radius of $\HH$ could lead to existence and non-degeneracy for the geodesic foliation from $s=1$ to $s=5/2$. For simplicity, we rather make the existence and smoothness of the geodesic foliation from $s=1$ to $s=5/2$ an assumption.

\item The null curvature components are essentially invariant by a change of foliation (see Proposition~\ref{prop:RR'}) and so is the smallness assumption $\RR \leq \varep$. Regarding the previous remarks, it is consistent however to assume it for the geodesic foliation, since we rely on its existence and on the control on the connection coefficient norm $\OO'_{[1,5/2]} \les \varep$ to obtain existence for the canonical foliation together with bounds for the corresponding canonical foliation connection coefficients.


  
\item The equations for the canonical foliation reduce to a system of coupled quasilinear elliptic and transport equations on $\HH$ (see equations~\eqref{eq:DLtrchi}-\eqref{eq:dampedHNL}), having curvature components as source terms, which are essentially invariant under a change of foliation. Thus, Theorem~\ref{TheoremMassLapseFoliationControl} can be seen as a small data time $1$ existence result for which the smallness is measured in terms of $L^2(\HH)$-norms of the null curvature components.

\item The desired bound for $\trchibt$ is obtained as part of the estimates~\eqref{est:boundedL2conn}, which are all needed to obtain existence and control of the canonical foliation on the interval $1\leq v \leq 2$ and eventually the aforementioned control of $\trchibt$.
  
\item Here and in the rest of the paper, \emph{smooth} means $C^\infty$ with respect to the $C^\infty$-topology of the manifolds $\MM$, $\HH$, etc.
\item The smoothness of the canonical foliation is a consequence of the smoothness of the geodesic foliation and is obtained by higher regularity comparison estimates (see Step 3 in Section~\ref{sec:overview} and Section~\ref{sec:apriorihigher}). Since we are only interested in smooth foliations, we did not seek any sharpness in these higher regularity estimates, as it can be seen in the proof of Proposition~\ref{prop:higherreg}, where we assume $C^k$-regularity with $k$ arbitrarily large for the geodesic foliation to prove $C^{k'}$-regularity with $k'\ll k$ for the canonical foliation.  
\end{enumerate}


\subsection{Proof of Theorem \ref{TheoremMassLapseFoliationControl}. The bootstrap argument.}\label{sec:overview}
The proof of Theorem \ref{TheoremMassLapseFoliationControl} relies on a bootstrap argument which we set up and prove in this section, assuming for the moment the local existence result and estimates that will be proved in Sections~\ref{sec:aprioriL2},~\ref{sec:apriorihigher} and~\ref{sec:thmloc}. \\
Let $D>0$ be a fixed (large) constant and let $v^\ast\in[1,2]$. We say that a foliation $(S_v)_{1\leq v \leq v^\ast}$ satisfies the \emph{bootstrap assumptions} $BA_{D\varep,[1,v^\ast]}$, if
  \begin{align*}
    \norm{\Om-1}_{L^\infty_{v}([1,v^\ast])L^\infty} + \norm{\Ups}_{L^\infty_{v}([1,v^\ast])L^\infty} + \OO_{[1,v^\ast]} & \leq D\varep.
  \end{align*}
                                                                                                                  
We say that a foliation $(S_v)_{1\leq v\leq v^\ast}$ is \emph{regular} if $v$ is a $C^1$-function and if
\begin{align*}
  \sum_{l \leq 5}\bigg(\norm{\Nd^l s}_{L^\infty_{v}([1,v^\ast])L^2} + \norm{\Nd^l(\Om-1)}_{L^\infty_{v}([1,v^\ast])L^2}\bigg) < \infty.
\end{align*}

We define $V\in[1,2]$ as
  \begin{align*}
    V := & \sup_{v^\ast\in[1,2]}\bigg\{\text{There exists a regular function $v$ on $\HH$ taking values} \\
         & \text{from $1$ to $v^\ast$ and such that the assumptions $BA_{D\varep,[1,v^\ast]}$ are satisfied} \bigg\},
  \end{align*}
and we show in the rest of this section that $V=2$.\\
  
\paragraph{\bf{Step 1}} \emph{It holds that $V>1$.}
  Indeed, this follows by the next local existence result.
\begin{theorem}[Local existence and continuation for the canonical foliation] \label{thm:loc} Let $(\MM,\g)$ be a smooth vacuum spacetime and $\HH \subset \MM$ a smooth null hypersurface foliated by a well-defined and smooth geodesic foliation $(S'_s)_{1\leq s\leq s^\ast}$. \\
Let $v^\ast \in [1,2]$. Assume there exists a $C^1$-function $v$ on $\HH$ taking values from $1$ to $v^\ast$ and defining a canonical foliation. Assume moreover that $s\in H^5(S_{v^\ast})$, $\norm{\Om-1}_{L^\infty(S_{v^\ast})}< \frac{1}{100}$ and that $S_{v^\ast}$ is close to the euclidian $2$-sphere in a weak sense (see Definition~\ref{def:unifweaksph}). There exists $\de>0$ (depending on $\sum_{l\leq 5}\norm{\Nd^l(s-v^\ast)}_{L^2(S_{v^\ast})}$ and $\norm{\gd'}_{C^{100}(\HH)}$) and a $C^1$-function $v$ taking values from $1$ to $v^\ast+\de$, coinciding with $v$ on $\{1\leq v\leq v^\ast\}$ and such that $(S_v)_{1\leq v\leq v^\ast+\de}$ is a canonical foliation.
\end{theorem}
The proof of Theorem \ref{thm:loc} is made by a fixed-point argument for a more general system of coupled quasilinear elliptic and transport equations and is detailed in Section~\ref{sec:thmloc}. The required assumption $|\Om-1|<1/100$ at the initial sphere $v=s=1$ and its weak sphericality are consequences of the low regularity bounds~\eqref{est:InitHNLL2} proved in Proposition~\ref{lem:InitHNLL2} (see also Section~\ref{SECboundsOnS1M0} and Remark~\ref{rem:weaksphS}).
\begin{remark}
  A local existence result for canonical foliations is proved in~\cite{Nicolo} but under the smallness assumption
  \begin{align*}
    \RR'_2 \leq \varep,
  \end{align*}
  where $\RR'_2$ contain $L^\infty$-norms of curvature components on $\HH$. Similarly, a local existence result was proved in~\cite{Sauter} for general foliations under $L^\infty(\HH)$-smallness assumptions on the curvature. As such smallness conditions can't be assumed in our low regularity setting, we have to prove another local existence result. 
\end{remark}
  
\paragraph{\bf{Step 2}} \emph{For $v^\ast\in[1,2]$ we can improve $BA_{D\varep,[1,v^\ast]}$ to $BA_{D'\varep,[1,v^\ast]}$ with $D'<D$.}
  We first show that the assumptions of Theorem \ref{TheoremMassLapseFoliationControl} imply that the canonical connection coefficients on $S_1$ are controlled. 
  \begin{proposition}[Connection coefficients bounds on $S_1$]\label{lem:InitHNLL2}
  Assume that for some real $\varep>0$,
  \begin{align}
    \IIt'_{S_1} \leq \varep.
  \end{align}
  There exists $\varep_0>0$ small such that if $0<\varep<\varep_0$, then we have
  \begin{align}\label{est:InitHNLL2}
    \IIt_{S_1} \les \varep.
  \end{align}
\end{proposition}
  The proof of Proposition~\ref{lem:InitHNLL2} is carried out in Section~\ref{SECboundsOnS1M0} and goes by direct comparison between the geodesic and the canonical connection coefficients. Namely most of the coefficients are identical since the first $2$-spheres of the two foliations coincide, $S'_{s=1} = S = S_{v=1}$.\\

Then the next proposition shows that we can improve the bootstrap assumptions.
  \begin{proposition}[Low regularity estimates]\label{prop:ImpBA}
    Assume that for some real $\varep>0$, 
    \begin{align*}
      \RR'_{[1,5/2]} + \OO'_{[1,5/2]} \leq \varep,
    \end{align*}
    and that
    \begin{align*}
      \II_{S_1} \leq \varep.
    \end{align*}
    Let $1<v^\ast<2$ and assume that the canonical foliation $(S_v)_{1\leq v \leq v^\ast}$ is regular and satisfies the bootstrap assumptions $BA_{D\varep,[1,v^\ast]}$, with $D>0$ a fixed constant.
    There exists $\varep_0>0$ such that if $0<\varep<\varep_0$, then the canonical foliation satisfies the bootstrap assumptions $BA_{D'\varep,[1,v^\ast]}$ with $D'<D$. 
  \end{proposition}
  Proposition~\ref{prop:ImpBA} is proved in Section~\ref{sec:aprioriL2}. The first step is to show that under the bootstrap bounds $BA_{D\varep,[1,v^\ast]}$, the null curvature components in the canonical foliation are comparable to the geodesic null curvature components
  \begin{align}\label{est:ImpL2sketch}
     \RR_{[1,v^\ast]} \les \RR'_{[1,5/2]}.
  \end{align}
  Then, under these bounds and weak regularity with constants $N,c$ of $S_1$, the foliation $(S_v)$ is uniformly weakly regular and spherical with constants depending only on $N,c,\varep$ (see Definitions~\ref{def:unifweakreg} and \ref{def:unifweaksph}). At this level of regularity, calculus inequalities can be derived on $\HH$ with constants depending only on $N,c,\varep$. Using these inequalities together with the null structure equations \eqref{eq:DLtrchi}-\eqref{eq:dampedHNL}, the bounds obtained on $S_1$ \eqref{est:InitHNLL2} and the obtained bounds for the null curvature components \eqref{est:ImpL2sketch}, there exists $\varep_0>0$ small enough such that if $0 <\varep < \varep_0$, the bootstrap assumptions $BA_{D\varep,[1,v^\ast]}$ can be improved to $BA_{D'\varep,[1,v^\ast]}$.\\

 
\paragraph{\bf{Step 3}}\emph{The canonical foliation is regular on $[1,V]$.}
Indeed, we show more generally the following proposition.
   \begin{proposition}[Higher regularity estimates]\label{prop:higherreg}
    Assume that the geodesic foliation $(S'_s)_{1\leq s \leq 5/2}$ is smooth and assume that for $1<v^\ast<2$ and for some real $\varep>0$ the canonical foliation $(S_v)_{1\leq v\leq v^\ast}$ is regular and satisfies the bootstrap assumptions $BA_{\varep,[1,v^\ast]}$.                                                                                         
    There exists $\varep_0>0$ such that if $0<\varep<\varep_0$ then, we have for all integers $m\geq 0$
  \begin{align*}
    \sum_{l+k\leq m}\norm{\Nd_L^{l}\Nd^{k}(s-v)}_{L^\infty_v([1,v^\ast]) L^2}\les C\left(\norm{\gd'}_{C^{m+100}(\HH)},m\right).
  \end{align*}
\end{proposition}
    The proof goes by standard Gr\"onwall arguments and is carried out in Section~\ref{sec:apriorihigher}.
In particular, for $m=6$, this gives the desired regularity result. 
                                                                                                                     
By continuity, we therefore deduce that the canonical foliation is regular on the full interval $[1,V]$.

Additionally, using these higher regularity estimates, one can deduce the smoothness of the canonical foliation.\\                                                                                          

\paragraph{\bf{Step 4}}\emph{The foliation can be continued past $V$.}
Using the estimates from Step 2 and 3, the assumptions of the local existence and continuation Theorem~\ref{thm:loc} are satisfied and therefore we deduce that the canonical foliation can be extended past $V$ which therefore implies that $V \geq 2$.                                                                                                                       

\subsection{Organisation of the paper}

The rest of the paper is organised as follows.
\begin{itemize}
\item In Section \ref{sec:calculus}, we state the functional calculus results that hold under weak regularity conditions for a foliation of $\HH$.  
\item Section~\ref{sec:aprioriL2} is dedicated to the proof of the low regularity bounds for the canonical connection coefficients on sphere $S_1$ and to the improvement of the bootstrap assumptions.
\item Section ~\ref{sec:apriorihigher} is dedicated to the proof of higher regularity bounds for the canonical foliation. 
\item Section~\ref{sec:thmloc} is dedicated to the proof of the local existence Theorem~\ref{thm:loc}.
\end{itemize}

\section{Calculus Prerequisites} \label{sec:calculus}

In this section, we state the necessary calculus prerequisites for Sections \ref{sec:aprioriL2}, \ref{sec:apriorihigher} and \ref{sec:thmloc}. The results are based on the pioneering works \cite{KlRod1}, \cite{KlRod2} and \cite{KlRod3} (see also \cite{Wang}), with further improvements and simplifications taken from \cite{ShaoBesov}, whose presentation we shall follow and whose results we shall use as a black box.

\subsection{Uniform weak regularity of foliations}
We now state the definition of uniform weak regularity for a foliation which allows to develop uniform calculus on the leaves of the foliation (see \cite{ShaoBesov} Sections 3.3 and 4.3).
\begin{definition}\label{def:unifweakreg}
  Let $N\geq 1$ be an integer and $C>0$ a real number. Let $v^\ast>1$ be a real number.
  We say that a foliation $(S_v)_{1\leq v \leq v^\ast}$ on $\HH$ is uniformly weakly regular with constants $N,C$, if the $2$-sphere $S_1$ is weakly regular with constants $N,C$ in the sense of Definition~\ref{def:weakreg}, the following bounds are satisfied
  \begin{align}\label{est:unifweakregnew1}\begin{aligned}
      \norm{\Om-1}_{L^\infty_vL^\infty} & \leq 1/10, \\
      \norm{\trchi}_{L^\infty_vL^\infty} & \leq C, \\
      \norm{\chih}_{L^\infty L^2_v} & \leq C, \\
      \NN_1(\chi) + \NN_1(\Om) & \leq C,
    \end{aligned}
  \end{align}
  and there exists a $S_v$-tangent 3-tensor $\Psi$ satisfying
  \begin{align}\label{eq:unifweakregnew2}
    \Nd_L\Psi_{ABC} = & \Om \Nd_A(\Om^{-1}\chi_{BC}) - \Om \Nd_C(\Om^{-1}\chi_{BA}),
  \end{align}
   such that
  \begin{align}\label{est:unifweakregnew2}
    \norm{\Psi}_{L^4 L^\infty_v} & \leq C.
  \end{align}
\end{definition}
{\bf Remarks}
\begin{enumerate}
\item For simplicity these assumptions are stronger than and imply in particular assumptions {\bf (F2)} with constants $N,C$ and $B\equiv C$ of Section 4.5 in~\cite{ShaoBesov}, since the tensor $k$ in~\cite{ShaoBesov} reads in the present paper $k \equiv \Om^{-1}\chi$.
\item Under these assumptions, one can deduce that each $2$-sphere $S_v$ is weakly regular in the sense of Definition~\ref{def:weakreg} with uniform constants $N,C'$, where $C'(N,C)>0$ (see Proposition 4.13 in~\cite{ShaoBesov}).  
\item The assumption~\eqref{est:unifweakregnew2} is designed so that using the frames $(e_A)_{A=1,2}$ defined on the first $2$-sphere $S$, the regularity of their associated Fermi propagated frames on $\HH$ is transported, \emph{i.e.} $\norm{\Nd e_A}_{L^4 L^\infty_v} \leq C$. This regularity is then sufficient in~\cite{ShaoBesov} for running a scalarisation procedure for tensorial estimates and then comparing geometric Besov norms for scalars to coordinate-based Besov norms (see Sections 4,5 and Appendix A in that paper).
\item From the weak regularity of each $2$-sphere $S_v$, we deduce in particular that
  \begin{align*}
    C^{-1} \les \sqrt{\det(\gd)} \les C,
  \end{align*}
  uniformly on $\HH$.
  As a consequence, for all $S_v$-tangent tensor $F$ and for all $1 \leq p \leq q \leq \infty$, we have
\begin{align*}
\norm{F}_{L^q L^p_v} \les \norm{F}_{L^p_v L^q},
\end{align*}
 where the constant depends only on $N,C$.
\end{enumerate}

{\bf Notations.} Here and in the rest of this section, we take out any reference to $v^\ast$ in the $L^pL^q$-norms for ease of notation and we moreover denote $\NN_m := \NN_m^v$.

\subsection{Littlewood-Paley theory and Besov spaces}\label{sec:LPtheory} In this section, we define Littlewood-Paley projections and Besov spaces on Riemannian $2$-spheres $(S,\gd)$.\\

Let $\Ld$ denote the Laplacian on $(S,\gd)$. Interpreting $-\Ld$ as a positive self-adjoint unbounded operator acting on tensors in $L^2(S)$, we have the spectral decomposition (see \cite{ShaoBesov} for details)
\begin{align*}
-\Ld = \int\limits_0^\infty \la \mathrm{d}E_\la.
\end{align*}

We define the corresponding Littlewood-Paley operators as follows.
\begin{itemize}

\item Let $\phi \in C^\infty(\RRR)$ be a function such that $\mathrm{supp} \, \phi \subset \{ 1/2 \leq \vert \xi \vert \leq 2\}$ and 
\begin{align*}
\sum\limits_{k \in \mathbb{Z}} \phi(2^{-2k}\xi) =1 \text{ for all } \xi \in \RRR \setminus \{ 0\}.
\end{align*}

\item For each $k \in \mathbb{Z}$, define the Littlewood-Paley operator acting on tensors in $L^2(S)$ by
\begin{align*}
P_k = \phi(-2^{-2k} \Ld), \,\,\,\,\,\, P_-= \de_{\{0\}}(-\Ld),
\end{align*}
where $\de_{\{0\}}(-\Ld)$ denotes the $L^2$-projection onto the kernel of $-\Ld$.

\item For $k\in \ZZZ$, define the aggregated operators
\begin{align*}
P_{<k} = P_- + \sum_{l<k} P_l,
\end{align*}
where the summation is in the strong operator topology. In particular,
\begin{align}\label{eq:LPId}
P_{<0} + \sum\limits_{k\geq0} P_k = 1.
\end{align}

\end{itemize}
  
Using the Littlewood-Paley operators, we next define Besov spaces.
\begin{definition}[Geometric tensorial Besov space] For a $S$-tangent tensor $F$ we define the norms
\begin{align*}
\Vert F \Vert_{B^{0}(S)} :=& \sum\limits_{k\geq0} \Vert P_k F \Vert_{L^2(S)} + \Vert P_{<0} F \Vert_{L^2(S)}, 
\end{align*}
\end{definition}
                             
We have moreover the following $v$-\emph{integrated} Besov spaces.
\begin{definition}[Geometric tensorial $v$-integrated Besov spaces]\label{def:vBesovspaces}
Define for a $S_v$-tangent tensor $F$ on $\HH$,
\begin{align*}
  \norm{F}_{\PPo} & := \sum\limits_{k\geq 0} \norm{P_kF}_{L^2_v L^2} + \norm{P_{<0}F}_{L^2_vL^2}, \\
  \norm{F}_{\QQo} & := \left(\sum\limits_{k\geq 0} 2^{k}\norm{P_kF}^2_{L^\infty_vL^2} + \norm{P_{<0}F}^2_{L^\infty_vL^2}\right)^{1/2}.                               
\end{align*}
\end{definition}
                    
\begin{remark}
  The space $B^0(S)$ corresponds to the $L^2(S)$-based Besov space on $S$, with, with respect to the notations in~\cite{ShaoBesov}, parameters $s=0$ and $a=1$.\\
  The $v$-integrated spaces $\PPo$ and $\QQo$ correspond to the $L^2(S)$-based $v$-integrated Besov space, with, with respect to the notations in~\cite{ShaoBesov}, parameters $s=0$, $a=1$ and $p=2$ and $s=1/2$, $a=2$ and $p=\infty$ respectively. 
\end{remark}

Finally, set for reals $s\in \RRR$ and $S$-tangent tensors $F$,
\begin{align*}
\Vert F \Vert_{H^{s}(S)} := \Vert (I-\Ld)^{s/2} F \Vert_{L^2(S)},
\end{align*}
where the fractional Laplace operator is defined as in \cite{ShaoBesov}.
With this definition, we have
\begin{align*}
  \norm{F}_{H^1(S)} \simeq \norm{\Nd F}_{L^2(S)} + \norm{F}_{L^2(S)}.
\end{align*}


The next lemma is proved in Appendix~\ref{app:proofSec3}.
\begin{lemma}\label{lem:Hs}
  For an $S_v$-tangent tensor $F$, we have
  \begin{align*}
    \norm{\Nd F}_{H^{-1/2}(S)} \les \norm{F}_{H^{1/2}(S)}.
  \end{align*}
\end{lemma}

\subsection{Sobolev inequalities on 2-spheres}
The next lemma is proved in Section 2.5 in~\cite{ShaoBesov}.
\begin{lemma}[Classical Sobolev inequalities on $S$] \label{lemma:CalculusOnGat1} \label{lem:sobS}
  Let $(S, \gd)$ be weakly regular $2$-sphere with constants $N,C$.
  Then for a scalar function $f$ and for an $S$-tangent tensor $F$, we have
\begin{align*}
\Vert F \Vert_{L^4(S)} \lesssim & \Vert \Nd F \Vert_{L^2(S)}^{1/2} \Vert F \Vert_{L^2(S)}^{1/2} + \Vert F \Vert_{L^2(S)},  \\ 
\Vert F \Vert_{L^\infty(S)} \lesssim & \Vert \Nd^2 F \Vert_{L^2(S)}^{1/2} \Vert F \Vert_{L^2(S)}^{1/2} + \Vert F \Vert_{L^2(S)}, \\
\Vert F \Vert_{L^\infty(S)} \lesssim & \Vert \Nd F \Vert_{L^4(S)} + \Vert F\Vert_{L^4(S)},
\end{align*}
where the constants depend only on $N,C$.
\end{lemma}

The next lemma follows from Proposition 3.3 in \cite{ShaoBesov}.
\begin{lemma}[Besov Sobolev inequalities on $S$]\label{lem:sobBesS}
  Let $(S,\gd)$ be a weakly regular $2$-sphere with constants $N,C$.
  Then for an $S$-tangent tensor $F$, we have
  \begin{align*}
    \norm{F}_{L^4(S)} & \les \norm{F}_{H^{1/2}(S)},
  \end{align*}
  where the constant depends only on $N,C$.
\end{lemma}

\subsection{Sobolev inequalities on $\HH$}

\begin{lemma}[Classical Sobolev inequalities on $\HH$] \label{lemma:calculusOnH1} \label{lem:sob}
  Let $(S_v)$ be a uniformly weakly regular foliation on $\HH$ with constants $N,C$.
  Then for an $S_v$-tangent tensor $F$ on $\HH$,
  \begin{align*}
    \Vert F \Vert_{L^4 L^\infty_v} \les & \NN_1(F),\\
    \Vert F \Vert_{L^6_v L^6} \les & \NN_1(F), \\
    \Vert F \Vert_{L^\infty_{v}L^\infty} \les & \NN_1(\Nd F) + \NN_1(F), \\
    \norm{F}_{L^2_v L^4} \les & \norm{\Nd F}_{L^2_v L^2}^\half\norm{F}_{L^2_v L^2}^\half + \norm{F}_{L^2_v L^2}.
  \end{align*}
  All constants in these estimates depend only on $N,C$.
\end{lemma}


We have the following Sobolev Besov estimate, see Proposition 5.3 in~\cite{ShaoBesov}.
\begin{lemma}[Besov Sobolev inequalities on $\HH$]\label{lem:SobBes}
  Let $(S_v)$ be a uniformly weakly regular foliation on $\HH$ with constants $N,C$.
  Let $F$ be a $S_v$-tangent tensor. We have
  \begin{align*}
  \norm{F}_{\QQo} \les & \NN_1(F),
\end{align*}
where the constant depends only on $N,C$.
\end{lemma}

We have the following product estimate in Besov spaces, see Theorem 3.6 in~\cite{ShaoBesov}.
\begin{lemma}[Besov product estimates]\label{lem:prodBesov}
  Let $(S_v)$ be a uniformly weakly regular foliation on $\HH$ with constants $N,C$.
  Let $F$ and $G$ be two $S_v$-tangent tensors. We have
  \begin{align*}
    \norm{FG}_{\PPo} \les & \NN_1(F)(\norm{G}_{L^2_v L^2}+\norm{\Nd G}_{L^2_v L^2}),
  \end{align*}
  where the constant depends only on $N,C$.
\end{lemma}

\subsection{Null transport equations on $\HH$}

We have the following $L^p L^\infty_v$-estimates for solutions of null transport equations. 
                            .
\begin{lemma}[$L^p L^\infty_v$-estimates for transport equations] \label{lemmaTRANSPORT} \label{lem:transport}
  Let $(S_v)$ be a uniformly weakly regular foliation of $\HH$ with constants $N,C$. \\
  Let $\kappa$ be a real number.
  For $F$ an $S_v$-tangent tensor satisfying on $\HH$
  \begin{align*}
    \Nd_L F +\kappa \trchit F = W,
  \end{align*}
  and for all $1 \leq p \leq \infty$, we have
  \begin{align*}
    \norm{F}_{L^p L^\infty_v} \les \norm{F}_{L^p(S_1)} + \norm{W}_{L^p L^1_v},
  \end{align*}
  where the constant depends only on $N,C,p,\kappa$.
\end{lemma}
\begin{proof}
  The proof follows by applying Proposition 4.6 in~\cite{ShaoBesov} to the transport equation obtained for the renormalised quantity $\exp\left(\int_{1}^v\Om^{-1} \trchit\,\text{d}v'\right) F$ and using the uniform weak regularity bounds~\eqref{est:unifweakregnew1} for $\Om$ and $\trchit$. Details are left to the reader.
\end{proof}
\begin{remark}\label{rem:noassum}
  Proposition 4.6 in~\cite{ShaoBesov} does not require that assumption~\eqref{est:unifweakregnew2} is satisfied. This will be used when proving that assumption~\eqref{est:unifweakregnew2} holds.
\end{remark}

\subsection{$L^\infty L^2_v$ Trace estimate}
By Theorem 5.7 in~\cite{ShaoBesov}, we have the next trace estimate.
\begin{lemma}[Trace estimate]\label{lem:traceBesov}
  Let $(S_v)$ be a uniformly weakly regular foliation with constants $N,C$ and let $F$ be a $S_v$-tangent tensor such that
  \begin{align*}
    \Nd F = \Nd_LP + E.
  \end{align*}
  Then
  \begin{align*}
    \norm{F}_{L^\infty L^2_v} \les & \NN_1(P) + \norm{E}_{\PPo} + \NN_1(F),
  \end{align*}
  where the constant only depends on $N,C$.
\end{lemma}
\begin{proof}
  This is Theorem 5.7 in~\cite{ShaoBesov} together with the bounds~\eqref{est:unifweakregnew1} of Definition~\ref{def:unifweakreg} and Sobolev Lemma~\ref{lem:sob}.
\end{proof}                            

\subsection{Uniform weak sphericality}
In order to have uniform estimates for Hodge systems on 2-spheres $S_v$, we introduce the following definition of uniform weak sphericality (see Section 6.1 in~\cite{ShaoBesov}).
                            \begin{definition}\label{def:unifweaksph} Let $N\geq 1$ be an integer and $C,D_{sph}>0$ be reals.
                              We say that a $2$-sphere $S$ is \emph{weakly spherical} with constants $N,C,D_{sph}$ and radius $R$ if it is weakly regular in the sense of Definition~\ref{def:weakreg} with constants $N,C$ and if the Gauss curvature $K$ of $S$ can be written
                              \begin{align*}
                                K-\frac{1}{R^2} = \Divd\Psi +\Th,
                              \end{align*}
                              where
                              \begin{align*}
                                \norm{\Psi}_{H^{1/2}(S)} +\norm{\Th}_{L^2(S)} \leq D_{sph}.
                              \end{align*}

  We say that a foliation $(S_v)$ of $\HH$ is \emph{uniformly weakly spherical} with constants $N,C,D_{sph}$ if it is uniformly weakly regular in the sense of~\ref{def:unifweakreg} with constants $N,C$ and such that the Gauss curvature $K$ of $S_v$ can be written 
  \begin{align*}
    K-\frac{1}{v^2} = \Divd \Psi + \Th,
  \end{align*}
  where
  \begin{align*}
    \norm{\Psi}_{\QQo} + \norm{\Th}_{L^\infty_v L^2} \leq & D_{sph}.
  \end{align*}
\end{definition}

\begin{remark}
  From the Definition~\ref{def:vBesovspaces} of the $v$-integrated Besov space $\QQo$, it is clear that every $2$-sphere $S_v$ of a uniformly weakly spherical foliation is weakly spherical with radius $v$ and uniform constants.
\end{remark}                                                                                  
\subsection{Bochner identities on 2-spheres and consequences}
We first recall the Bochner identity on spheres (see~\cite{KlRod1}, p. 483 and p. 488).
\begin{lemma}[Bochner identities]\label{lem:Bochnerid} Let $(S,\gd)$ be a Riemannian $2$-sphere.
  For scalar functions $f$ on $S$, we have
  \begin{align*}
    \int_{S}|\Nd^2 f|^2 = & \int_S |\Ld f|^2 - \int_SK|\Nd f|^2,
  \end{align*}
  where $K$ denotes the Gauss curvature of $S$.
  For an $S$-tangent $1$-form $F$, we have
  \begin{align*}
    \int_S |\Nd^2F|^2 = &  \int_S |\Ld F|^2 -2\int_SK|\Nd F|^2 + \int_SK |\Divd F|^2  + \int_SK^2 |F|^2.
  \end{align*}
\end{lemma}

The Bochner identities of Lemma \ref{lem:Bochnerid} imply the next estimates, see Section 6 in~\cite{ShaoBesov}.
\begin{lemma}[Bochner estimates]\label{lem:Bochnerest}
  For a weakly spherical $2$-sphere $S$ of radius $1$ with constants $N,C,D_{sph}$, there exists a universal constant $D_0>0$ such that if $D_{sph} < D_0$, then the following holds.
\begin{enumerate}
\item  For scalar function $f$ on $S$, we have
  \begin{align*}
    \norm{\Nd^2f}_{L^2(S)} + \norm{\Nd f}_{L^2(S)} \les \norm{\Ld f}_{L^2(S)}.
  \end{align*}

\item  For an $S$-tangent $1$-form $F$, we have
  \begin{align*}
    \norm{\Nd^2F}_{L^2(S)} \les & \norm{\Ld F}_{L^2(S)} + \norm{\Nd F}_{L^2(S)} + \norm{F}_{L^2(S)}.
  \end{align*}
\end{enumerate} 
  Moreover, for a uniformly weakly spherical foliation $(S_v)$ of $\HH$ with constants $N,C,D$, with $D_{sph} < D_0$, the following holds.
\begin{enumerate}
\item For scalar functions $f$ on $\HH$, we have
  \begin{align*}
    \norm{\Nd^2f}_{L^2_v L^2} + \norm{\Nd f}_{L^2_vL^2} \les & \norm{\Ld f}_{L^2_vL^2}.
  \end{align*}
\item For a $S_v$-tangent $1$-form $F$, we have
  \begin{align*}
    \norm{\Nd^2F}_{L^2_vL^2} \les & \norm{\Ld F}_{L^2_vL^2} + \norm{\Nd F}_{L^2_v L^2} + \norm{F}_{L^2_vL^2}.
  \end{align*}
\end{enumerate}
\end{lemma}
                                    


\subsection{Hodge systems on 2-spheres} \label{sec:HodgeSystemsOnGa} 
In this section, we recall standard Hodge theory on Riemannian $2$-spheres, see for example~\cite{ChrKl93}.
\begin{definition}\label{def:Dd} Let $(S,\gd)$ be a Riemannian $2$-sphere.
  We define the Hodge operators $\Dd_1$ and $\Dd_2$ that act respectively on $S$-tangent $1$-forms $\phi$ and on $S$-tangent traceless symmetric $2$-tensors $\psi$ by
  \begin{align*}
    \Dd_1\phi & := (\Divd\phi, \Curld\phi), \\
    (\Dd_2\psi)_A & := \Divd \psi_A.
  \end{align*}
  We denote by $\dual\Dd_1$ and $\dual\Dd_2$ their $L^2$-adjoint.
  For $f,h$ scalar functions on $S$ and for $\phi$ a $S$-tangent $1$-form, we have
  \begin{align*}
    \Dds_1(f,h) & = -\Nd_A f + \dual \Nd_A h, \\
    \Dds_2\phi & = - \half \Nd\otimesh\phi.
  \end{align*}
\end{definition}
\textbf{Remarks.}
\begin{enumerate}
\item The following identities hold (see \cite{ChrKl93}),
\begin{align}\label{eq:DdDds}
    \begin{aligned}
    & \Dd_1\Dds_1 =  -\Ld, \,\, & \Dd_1\Dds_1 & =  -\Ld +K, \\
    & \Dd_2\Dds_2 =  -\half\Ld -\half K, \,\, & \Dds_2\Dd_2 & = -\half\Ld +K.
    \end{aligned}
  \end{align}
\item The operator $\Dd_1$ is a bijection between the space of vector fields and the space of pairs of functions with vanishing mean.
\item The operator $\Dd_2$ is a bijection between the space of symmetric tracefree $2$-tensors and the orthogonal complement of the space of conformal Killing vector fields.
\item We denote by $\Ddi_1$ and $\Ddi_2$ the inverses of $\Dd_1$ and $\Dd_2$ composed with the projections onto their respective domain.
\end{enumerate}

We have the following $L^2$-estimates for Hodge systems, see Proposition 6.5 in~\cite{ShaoBesov}.
\begin{lemma}[Estimate for Hodge systems] \label{lemmaHODGEGAMMA} \label{lem:Hodge}
  For a weakly spherical $2$-sphere of radius $1$ $S$ with constants $N,C,D_{sph}$, there exists $D_0>0$ such that if $D_{sph}<D_0$, the following holds. 
  For an $S$-tangent tensor of appropriate type $F$, we have 
  \begin{align*}
    \norm{\Nd\Ddi F}_{L^2(S)} + \norm{\Ddi F}_{L^2(S)} \les & \norm{F}_{L^2(S)}, \\
    \norm{\Nd (\Dds_1)^{-1} F}_{L^2(S)} + \norm{(\Dds_1)^{-1}F}_{L^2(S)} \les & \norm{F}_{L^2(S)}, \\
  \end{align*}
  where $\Ddi \in \left\{ \Dd_1^{-1},\Dd_2^{-1}\right\}$ and the constants depend only on $N,C$.\\
  For a uniformly weakly spherical foliation $(S_v)$ of $\HH$ with constants $N,C,D_{sph}$, there exists $D_0>0$ such that if $D_{sph} < D_0$, the following holds.
  For an $S_v$-tangent tensor of appropriate type $F$, we have
  \begin{align*}
    \norm{\Nd\Ddi F}_{L^2_vL^2} + \norm{\Ddi F}_{L^2_vL^2} \les & \norm{F}_{L^2_vL^2}, \\
    \norm{\Nd (\Dds_1)^{-1} F}_{L^2_vL^2} + \norm{(\Dds_1)^{-1}F}_{L^2_vL^2} \les & \norm{F}_{L^2_vL^2}, \\
  \end{align*}
  where $\Ddi \in \left\{ \Dd_1^{-1},\Dd_2^{-1}\right\}$ and the constants depend only on $N,C$. 
\end{lemma}

We have the following elliptic estimates. The proof is postponed to Appendix~\ref{app:proofSec3}.
\begin{lemma}\label{lem:heat}
  For a weakly spherical $2$-sphere $S$ of radius $1$ with constants $N,C,D_{sph}$, there exists a universal constant $D_0>0$ such that if $D_{sph}<D_0$, then the following holds.\\
  Assume that $f$ satisfies the equation
  \begin{align*}
    \Ld f = \Divd P + h,
  \end{align*}
  then
  \begin{align*}
    \norm{\Nd f}_{L^2(S)} + \norm{f-\overline{f}}_{L^2(S)} \les & \norm{P}_{L^2(S)} + \norm{h}_{L^{4/3}(S)}. 
  \end{align*}
  Moreover, for a uniformly weakly spherical foliation $(S_v)$ of $\HH$ with constants $N,C,D$, with $D_{sph} < D_0$, the following holds.\\
  Assume that $f$ satisfies the equation
  \begin{align*}
    \Ld f = \Divd P + h,
  \end{align*}
  then,
  \begin{align*}
    \norm{\Nd f}_{L^2_vL^2} +\norm{f-\overline{f}}_{L^2_vL^2} \les \norm{P}_{L^2_vL^2} + \norm{h}_{L^2_vL^{4/3}}.
  \end{align*}
  The constants only depend on $N,C$.
\end{lemma}

The next lemma follows from Proposition 6.10 and Theorem 6.8 in~\cite{ShaoBesov}.
\begin{lemma}[Elliptic estimates in fractional Sobolev spaces]\label{lem:BochnerBesovest}
  For a weakly spherical $2$-sphere $S$ with constants $N,C,D_{sph}$, there exists a constant $D_0>0$ such that if $D_{sph}<D_0$, then the following holds.
  \begin{enumerate}
  \item
  Let $f$ be a scalar function on $S$ and $X$ a $S$-tangent $1$-form satisfying
  \begin{align*}
    \Ld f = \Divd X, \, \, \int_S f = 0.
  \end{align*}
  Then,
  \begin{align*}
    \norm{\Nd f}_{H^{1/2}(S)} + \norm{f}_{L^2(S)} \les \norm{X}_{H^{1/2}(S)}. 
  \end{align*}
\item For an $S$-tangent $1$-form $F$, we have
  \begin{align*}
    \norm{\Dds_1^{-1}F}_{H^{1/2}(S)} \les \norm{F}_{H^{-1/2}(S)}.
  \end{align*}
\end{enumerate}
\end{lemma}



\section{Low regularity estimates}\label{sec:aprioriL2}
This section is dedicated to the proofs of Propositions~\ref{lem:InitHNLL2} and~\ref{prop:ImpBA} (see Step 2 in Section~\ref{sec:overview}).\\  
Let $(S'_s)_{1\leq s \leq 5/2}$ denote the geodesic foliation on $\HH$, and assume that for some $\varep>0$,
\begin{align}\label{eq:CTMAFthmAssum1}
\II'_{S_1} + \RR'_{[1,5/2]} + \OO'_{[1,5/2]} \leq \varep.
\end{align}
In Section~\ref{SECboundsOnS1M0}, we show that, for $\varep$ small enough, we have $\II_{S_1}$. This proves Proposition~\ref{lem:InitHNLL2}.\\  
From Section~\ref{sec:unifweakImp} on, we assume that $1<v^\ast <2$ is a real number and that the canonical foliation $(S_v)_{1\leq v \leq v^\ast}$ is regular. We suppose further that for a fixed large constant $D$, for $1 \leq v \leq v^\ast$,
\begin{align} \label{eq:CTMAFthmBA}
\norm{\Om -1}_{L^\infty_{v}([1,v^\ast])L^\infty} + \norm{\Ups}_{L^\infty_{v}([1,v^\ast])L^\infty} + \OO_{[1,v^\ast]} \leq D \varep.
\end{align}

We prove in Sections~\ref{sec:unifweakImp}-~\ref{SECcomparisonEstimates} that for $\varep>0$ sufficiently small, we can improve \eqref{eq:CTMAFthmBA}, \emph{i.e.} we show that,
\begin{align*} 
\Vert \Om -1 \Vert_{L^\infty_{v}([1,v^\ast])L^\infty} + \norm{\Ups}_{L^\infty_{v}([1,v^\ast])L^\infty} + \OO_{[1,v^\ast]} + \RR_{[1,v^\ast]} \leq D'\varep
\end{align*}
for a constant $0<D'<D$. This proves Proposition~\ref{prop:ImpBA}. 

{\bf Notation.} To ease the notations, we suppress all references to $v^\ast$ in the following and we denote by $\NN_m := \NN_m^v$ and by $\NN_m' := \NN_m^{s,5/2}$.
    

\subsection{Bounds on $S_1$ for the connection coefficients.}\label{SECboundsOnS1M0} This section is dedicated to the proof of Proposition~\ref{lem:InitHNLL2}. \\  
    From the relations from Proposition~\ref{prop:AA'} and the fact that $\Ups=0$ on $S_1$ since $s=1$, we have
  \begin{align*}
    \chi = (\chi')^\dg, \,\, \chib = (\chib')^\dg, \,\, \zet =~(\zet')^\dg,
  \end{align*}
  which by using~\eqref{def:mu} also gives $\mut = \mut'$.


  Using the elliptic equation for $\log\Om$~\eqref{eq:CMAFbis}, using the hypotheses estimates~\eqref{eq:CTMAFthmAssum1} and applying Lemma~\ref{lem:BochnerBesovest}, we have
  \begin{align*}
    \norm{\Nd\log\Om}_{H^{1/2}(S_1)} + \norm{\log\Om}_{L^2(S_1)} & \les \norm{\zet}_{H^{1/2}(S_1)} + \norm{\mu-\muo}_{L^2(S_1)}\\
                                                                & \les \varep.
  \end{align*}

  From this and relation~\eqref{eq:zetetabt}, we deduce
  \begin{align*}
    \norm{\etabt}_{H^{1/2}(S_1)} \les \norm{\zet}_{H^{1/2}(S_1)} + \norm{\Nd\log\Om}_{H^{1/2}(S_1)} \les \varep.
  \end{align*}
  Moreover, using Sobolev Lemmas~\ref{lem:sobS} and~\ref{lem:sobBesS}, together with the previous estimates, we have
   \begin{align*}
     \norm{\log\Om}_{L^\infty(S_1)} & \les \norm{\log\Om}_{H^{1/2}(S_1)} + \norm{\log\Om}_{L^2(S_1)} \\
     & \les \varep,
   \end{align*}
   which implies
   \begin{align*}
         \norm{\Om-1}_{L^\infty(S_1)} \les \varep.
   \end{align*}

  This finishes the proof of the bound~(\ref{est:InitHNLL2}) and of Proposition~\ref{lem:InitHNLL2}.
  \begin{remark}\label{rem:weaksphS}
    From the Gauss equation~(\ref{eq:Gauss}) and the definition of the mass aspect function~(\ref{def:mu}), we have on $S_1$
    \begin{align*}
      K-1 = \Th + \Divd \ze,
    \end{align*}
    where
    \begin{align*}
      \Th := \half \left(\trchit-2\right)-\half\left(\trchibt+2\right) - \frac{1}{4}\left(\trchit-2\right)\left(\trchibt+2\right) + \mu.
    \end{align*}
    Using estimate~\eqref{est:InitHNLL2}, this implies that the $2$-sphere $S_1$ is weakly spherical with radius $1$ and constants $N,C,C'\varep$ in the sense of Definition~\ref{def:unifweaksph}, where $C'>0$ is a universal constant. This assumption is required to apply Theorem~\ref{thm:loc} in Step 1 of Section~\ref{sec:overview}.
  \end{remark}

\subsection{Uniform weak regularity and sphericality of the canonical foliation}\label{sec:unifweakImp}
In this section, we show that under the bootstrap assumptions \eqref{eq:CTMAFthmBA}, the regularity of $S_1$ is propagated to the canonical foliation. More specifically, we show that the canonical foliation is uniformly weakly regular in Lemma \ref{LemmaFoliationRegpropagation1} and uniformly weakly spherical in Lemma \ref{LemmaFoliationRegpropagation2}. 
  
\begin{lemma}[Uniform weak regularity]\label{LemmaFoliationRegpropagation1}
  Let $S_1$ be a weakly regular $2$-sphere with constants $N,c$. There exists $\varep_0>0$ and $C(N,c)>0$ such that for $\varep < \varep_0$ the canonical foliation is uniformly weakly regular with constants $N,C$.
\end{lemma}
\begin{proof}
  The bounds~\eqref{est:unifweakregnew1} directly follow from the bootstrap assumptions~\eqref{eq:CTMAFthmBA}.\\
  Using Gauss-Codazzi equation~\eqref{eq:Curlchi} for $\Curl\chi$, we deduce that for $\Psi$ a $3$-tensor verifying equation~\eqref{eq:unifweakregnew2}, we have
  \begin{align*}
    \Nd_L\Psi_{ABC} = & \Om \Nd_A(\Om^{-1}\chi_{BC}) - \Om \Nd_C(\Om^{-1}\chi_{BA}) \\
    = & -\Om^{-1}\left(\Nd_A\Om \chi_{BC}-\Nd_C\Om\chi_{BA}\right) \\
                      & + \iin_{AC}\Curl\chi_{B}\\
    = & -\Om^{-1}\left(\Nd_A\Om \chi_{BC}-\Nd_C\Om\chi_{BA}\right) \\
                      & + \iin_{AC} \left(-\ze_D\dual\chih_{DB} + \half\trchit \dual\ze_B - \dual\beta_B\right).    
  \end{align*}
  Using the transport equation~(\ref{eq:DLze}) for $\ze$, we have
  \begin{align*}
    \dual\beta_B = & \half\trchi(-\dual\ze_B+\dual\etab_B) + \dual\chih_{BD}(\etabt_D-\ze_D) -\Nd_L\dual\ze_B.
  \end{align*}
  Therefore, we have
  \begin{align*}
    \Nd_L\left(\Psi_{ABC}-\iin_{AC}\dual\ze_B\right) = & E_{ABC},
  \end{align*}
  with
  \begin{align*}
    E_{ABC} := & -\Om^{-1}\left(\Nd_A\Om \chi_{BC}-\Nd_C\Om\chi_{BA}\right) \\
               & + \iin_{AC} \left(\trchit\dual\ze_B-\half\trchit\dual\etabt_B -\dual\chih_{BD}\etabt_D\right).
  \end{align*}
  Using the transport Lemma~\ref{lem:transport} (see also Remark~\ref{rem:noassum}), the bootstrap assumptions~\eqref{eq:CTMAFthmBA} and choosing $\Psi_{ABC} := \iin_{AC}\dual\ze_B$ on $S_1$, we have
  \begin{align*}
    \norm{\Psi -\iin \dual\ze}_{L^4 L^\infty_v} \les & \norm{E}_{L^4 L^1_v} \\
    \les & \norm{\Om^{-1}}_{L^\infty_vL^\infty}\norm{\chi}_{L^\infty L^2_v}\norm{\Nd\Om}_{L^4 L^2_v} \\
                                                     & + \norm{\trchi}_{L^\infty_vL^\infty}\left(\norm{\ze}_{L^4L^1_v}+\norm{\etabt}_{L^4L^1_v}\right) \\
                                                     & + \norm{\chih}_{l^\infty L^2_v}\norm{\etabt}_{L^4L^2_v} \\
    \les & \, D\varep + (D\varep)^2.
  \end{align*}
  Using the bootstrap assumptions~\eqref{eq:CTMAFthmBA}, we deduce
  \begin{align*}
    \norm{\Psi}_{L^4L^\infty_v} \les & \norm{\ze}_{L^4 L^\infty_v} + (D\varep) + (D\varep)^2 \\
    \les & \,D\varep + (D\varep)^2 \\
    \leq & C.
  \end{align*}
  This finishes the proof of Lemma~\ref{LemmaFoliationRegpropagation1}.
\end{proof}

\begin{lemma}\label{LemmaFoliationRegpropagation2}[Uniform weak sphericality]
  There exists $\varep_0>0$, such that for $\varep < \varep_0$, the canonical foliation is uniformly weakly spherical with constants $N,C,D_{sph}$ and we have $D_{sph}<D_0$, where $D_0$ is the constant from Lemmas~\ref{lem:Bochnerest},~\ref{lem:Hodge},~\ref{lem:heat}, and~\ref{lem:BochnerBesovest}.
\end{lemma}
\begin{proof}
We have by the Gauss equation~\eqref{eq:Gauss} and the definition of the mass aspect function~\eqref{def:mu}
\begin{align*}
  K -\frac{1}{v^2} = \Th + \Divd\zet,
\end{align*}
where
\begin{align*}
  \Th := & \half v^{-1}\bigg(\trchit-\frac{2}{v}\bigg) -\half v^{-1}\bigg(\trchibt+\frac{2}{v}\bigg) - \frac{1}{4}\bigg(\trchit-\frac{2}{v}\bigg)\bigg(\trchibt+\frac{2}{v}\bigg) + \mut.
\end{align*}
Using bootstrap assumptions~\eqref{eq:CTMAFthmBA}, we have
\begin{align*}
  \norm{\Th}_{L^\infty_v L^2} \les & \norm{\trchit-\frac{2}{v}}_{L^\infty_vL^\infty} + \norm{\trchibt+\frac{2}{v}}_{L^\infty_vL^\infty}+ \norm{\trchit-\frac{2}{v}}_{L^\infty_vL^\infty}\norm{\trchibt+\frac{2}{v}}_{L^\infty_vL^\infty} \\
  & +\norm{\mut}_{L^\infty_v L^2} \\
  \les & \, D\varep + (D\varep)^2.
\end{align*}
And moreover, using bootstrap assumptions~\eqref{eq:CTMAFthmBA} and Lemma~\ref{lem:SobBes}, we have
\begin{align*}
  \norm{\zet}_{\QQo} & \les \NN_1(\zet) \les D\varep.
\end{align*}
Therefore
\begin{align*}
  \norm{\Th}_{L^\infty_v L^2} + \norm{\zet}_{\QQo} \les (D\varep) + (D\varep)^2 \leq D_0,
\end{align*}
for $\varep>0$ small enough. This finishes the proof of Lemma \ref{LemmaFoliationRegpropagation2}.
\end{proof}

The next lemma will allow us to compare norms for the geodesic and the canonical foliation components.                                              
\begin{lemma}[Integral comparison]
  For all $1\leq p \leq \infty$, and for all $S_v$-tangent tensor $F$, we have
  \begin{align*}
     \norm{F}_{L^p_v([1,v^\ast]) L^p} \les \norm{F^\ddg}_{L^p_s([1,5/2]) L^p}.
  \end{align*}
  To stress this foliation independence, we shall replace all $L^p_vL^p$ and $L^p_sL^p$ norms by $L^p(\HH)$ in the rest of the paper.
\end{lemma}
\begin{proof}
    Let $(x^1,x^2)$ be local coordinates on $S_1$. Extend $(x^1,x^2)$ on $\HH$ by
    \begin{align*}
    L (x^a) = 0 \text{ on } \HH \text{ for } a=1,2.
    \end{align*}
    The triplets $(s,x^1,x^2)$ and $(v,x^1,x^2)$ are local coordinates on $\HH$. Let $(\pr_s,\pr'_{x^1},\pr'_{x^2})$ and $(\pr_v, \pr_{x^1}, \pr_{x^2})$ be the respective corresponding coordinate vector fields.
    Then the following relations hold,
    \begin{align*}
    \begin{aligned}
      \pr_v & = \Om^{-1}\pr_s, \\
      \pr_{x^a} & = \pr'_{x^a} - \Om^{-1}(\pr'_{x^a}v) \pr_s \text{ for } a=1,2.
    \end{aligned}
    \end{align*}
    In particular,
    \begin{align*}
      \gd'_{ab} = \g(\pr'_{x^a},\pr'_{x^b}) = \g(\pr_{x^a},\pr_{x^b}) = \gd_{ab} \text{ for } a,b=1,2,
    \end{align*}
    and
    \begin{align*}
      \ga := \sqrt{\det(\gd'_{ab})} = \sqrt{\det(\gd_{ab})},
    \end{align*}
    where the indices $a,b$ with $a,b=1,2$ correspond to evaluation with respect to the coordinate vector fields $\pr_{x^a},\pr_{x^b}$.
    Performing a change of variable in the integrals, using the previous relations, and the bootstrap assumption~\eqref{eq:CTMAFthmBA} on $\Om$, we therefore deduce
    \begin{align*}
      \norm{F}^p_{L^p_v([1,v^\ast])L^p} = &  \int_v\int_{S_v}|F|^p \,\text{d}v \\
      = & \int_s\int_{S'_s}|F^\ddg|^p \Om \,\text{d}s \\
      \simeq & \int_s\int_{S'_s}|F^\ddg|^p \,\text{d}s \\
      \les & \norm{F^\ddg}_{L^p_s([1,5/2]) L^p}.
    \end{align*}
\end{proof}
\subsection{Bounds for the null curvature components of the canonical foliation on $\HH$}\label{sec:L2curvcomp}
           In this section, we estimate the canonical curvature components $(\alpha,\beta,\rho,\sigma,\betab)$ and $(\rhoc,\sigmatc,\betabc)$ on $\HH$ by comparing them to the geodesic components.
           \begin{lemma}\label{LEMHNLcurvature1}
             For $\varep>0$, sufficiently small, it holds that
             \begin{align}\label{est:ImpL2curv}
               \norm{\alpha}_{L^2(\HH)} + \norm{\beta}_{L^2(\HH)} +\norm{\rho}_{L^2(\HH)} + \norm{\sigma}_{L^2(\HH)} + \norm{\betab}_{L^2(\HH)} \les \varep.
             \end{align}
           \end{lemma}
             \begin{proof} By Proposition~\ref{prop:RR'}, it holds schematically that 
\begin{align*}
  R = & R' + (\Ups') R' + (\Ups')^2 R'  + (\Ups')^3 R',
\end{align*}
where 
\begin{align*}
R \in \{ \alpha,\beta,\rho,\sigma,\betab\} \text{ and }
R' \in \{\alpha',\beta',\rho',\sigma',\betab' \}.
\end{align*}

Using the above and the bootstrap assumptions \eqref{eq:CTMAFthmAssum1}, we have 
\begin{align*}
  \norm{R}_{L^2(\HH)} \les & \norm{R'}_{L^2(\HH)} + \norm{\Ups'}_{L^\infty(\HH)} \norm{R'}_{L^2(\HH)} + \norm{\Ups'}^2_{L^\infty(\HH)} \norm{R'}_{L^2(\HH)} \\
  &+ \norm{\Ups'}_{L^\infty(\HH)}^3 \norm{R'}_{L^2(\HH)} \\
  \les &\, \varep + (D\varep)\varep +(D\varep)^2\varep + (D\varep)^3\varep \\
  \les &\, \varep.
\end{align*}This finishes the proof of Lemma~\ref{LEMHNLcurvature1}.
           \end{proof}

Moreover, we have the following estimates for the renormalised canonical curvature components.
\begin{lemma}\label{LEMcurvatureCheckControl}
  We have,
  \begin{align}\label{est:checkcurv}
    \Vert \check{\rho} \Vert_{L^2(\HH)}+ \Vert \check{\si} \Vert_{L^2(\HH)}+\Vert\check{\betab}\Vert_{L^2(\HH)} \lesssim \varep.
  \end{align}
\end{lemma}
\begin{proof}
  Using equation~\eqref{eq:rhotc}, the bootstrap assumptions~\eqref{eq:CTMAFthmBA} and the previous estimate~\eqref{est:ImpL2curv}, we have
  \begin{align*}
    \norm{\rhotc}_{L^2(\HH)} \les & \norm{\rho}_{L^2(\HH)} + \norm{\chih}_{L^\infty_v L^4}\norm{\chibh}_{L^\infty_v L^4} \\
    \les & \norm{\rho}_{L^2(\HH)} + \NN_1(\chih)\NN_1(\chibh) \\
    \les &\, \varep + (D\varep)^2\\
    \les &\, \varep.
  \end{align*}
  The estimates for $\sigmatc$ and $\betabc$ follow analogously and are left to the reader.
\end{proof}
\subsection{Schematic notation for null connection coefficients and null curvature components}\label{SECnotationImprov}
For ease of presentation, in the following sections, we employ the next schematic notation for null connection coefficients and null curvature components.\\
\textbf{Notation.} Let 
\begin{align*}
A \in& \left\{ \trchit-\frac{2}{v},~\chih,~\zet,~\etabt,~\Nd\log\Om,~\trchib+\frac{2}{v} \right\} \cup \{  \Om-1,~\log\Om,~\Om^{-1}-1\}, \\
\Ab \in& A \cup \{ \chibh\}, \\
R \in& \left\{ \alpha,~\beta,~\rho,~\sigma,~\betab,~\rhotc,~\sigmatc,~\betabc \right\}. 
\end{align*}
Let moreover
\begin{align*}
\nab \in \{ \Nd, \Nd_L\}.
\end{align*}

Using the above notation, the bootstrap assumptions \eqref{eq:CTMAFthmBA}, the bounds on $S_1$ \eqref{est:InitHNLL2}, the improved curvature bounds~\eqref{est:ImpL2curv} and~\eqref{est:checkcurv}, can be written as follows.
\begin{lemma} It holds that
\begin{align*}
    \norm{\Ab}_{H^{1/2}(S_1)} & \les \varep, \\
        \norm{R}_{L^2(\HH)} & \les \varep, \\
    \norm{A}_{L^\infty L^2_v}+\NN_1(\Ab)+\norm{\Ab}_{L^4 L^\infty_v} & \les D\varep.
  \end{align*}
\end{lemma}
\begin{remark}
  To improve the bootstrap assumption for $\NN_1(\Ab)$, it is enough to improve only $\norm{\Nd_L\Ab}_{L^2(\HH)} + \norm{\Nd\Ab}_{L^2(\HH)} + \norm{\Ab}_{L^2(\HH)}$, since the $H^{1/2}(S_1)$ norms are already controlled by Lemma~\ref{lem:InitHNLL2}. 
\end{remark}

\subsection{Improvement of the estimates for $\Om$ and $\protect\Ab$}\label{sec:NullconnImpL2}
In this section, we improve the bounds for $\Om$ and $\Ab$ using the previously improved bounds for the null curvature components. 
                                                                        

\begin{lemma} 
  For $\varep>0$ sufficiently small, it holds that 
  \begin{align}\label{est:murenorm}
    \norm{\muc}_{L^2 L^\infty_v} \les \varep.
  \end{align}
\end{lemma}
\begin{proof}
  We rewrite schematically the transport equation~\eqref{def:mu} under the form
  \begin{align*}
    L(\muc) + \trchit \muc = \rhotc-\half\rhotco + A R + A\Nd A + A^2 + A^3. 
  \end{align*}
  Using Lemma~\ref{lem:transport}, the bootstrap assumptions~\eqref{eq:CTMAFthmBA}, the bounds on $S_1$~\eqref{est:InitHNLL2}, and the renormalised curvature bounds \eqref{est:ImpL2curv} and \eqref{est:checkcurv}, we have
  \begin{align*}
    \norm{\muc}_{L^2 L^\infty_v} \les & \norm{\muc}_{L^2(S_1)} + \norm{\rhotc}_{L^2 L^1_v} \\
                                      & + \norm{AR}_{L^2 L^1_v} + \norm{A\Nd A}_{L^2 L^1_v} + \norm{A^2}_{L^2 L^1_v} + \norm{A^3}_{L^2 L^1_v} \\
    \les & \norm{\mu}_{L^2(S_1)} + \norm{\rhotc}_{L^2(\HH)} \\
                                      & + \norm{A}_{L^\infty L^2_v}(\norm{R}_{L^2(\HH)} +\norm{\Nd A}_{L^2(\HH)} +\norm{A}_{L^2(\HH)} + \norm{A}^2_{L^4(\HH)}) \\
    \les & \,\varep + (D\varep)^2 + (D\varep)^3 \\
    \les & \,\varep.
  \end{align*}
\end{proof}
{\bf Notation.} From now on, we include $\mu$ in the schematic notation $R$.           

\begin{lemma} \label{lemmaHNLimprov2} For $\varep>0$ sufficiently small, it holds that
\begin{align}
\norm{\Nd^2\log\Om}_{L^2(\HH)} + \norm{\Nd\log\Om}_{L^2(\HH)} +\norm{\log\Om}_{L^2(\HH)} \les \varep. \label{est:ImpOm1}
\end{align}
\end{lemma}

\begin{proof} 
Using Lemma~\ref{lem:heat} on the elliptic equation~\eqref{eq:CMAFbis} and the improved estimates~\eqref{est:murenorm} and \eqref{est:checkcurv}, we get that
\begin{align*}
  \norm{\Nd\log\Om}_{L^2(\HH)} + \norm{\log\Om}_{L^2(\HH)} & \les \norm{\muc}_{L^2_v L^{4/3}} + \norm{\rhotc}_{L^2_v L^{4/3}} \\
                                                                                                    & \les \varep + \norm{\muc}_{L^2 L^\infty_v} + \norm{\rhotc}_{L^2(\HH)} \\
                                                                                                    & \les \varep.
\end{align*}
Using Lemma~\ref{lem:Bochnerest} on the same elliptic equation~\eqref{eq:CMAFbis} gives
\begin{align*}
  \norm{\Nd^2\log\Om}_{L^2(\HH)} + \norm{\Nd\log\Om}_{L^2(\HH)} & \les \norm{\rhotc}_{L^2(\HH)} + \norm{\mu}_{L^2(\HH)} \\
                                                                & \les \varep.
\end{align*}
\end{proof}





\begin{lemma}  \label{lemmaHNLimprov4} For $\varep>0$ sufficiently small, it holds that
  \begin{align}
    \norm{\Nd L(\log\Om)}_{L^2(\HH)} + \norm{L(\log\Om)}_{L^2(\HH)} \les &\, \varep, \label{eqlemmaHNLimprov31} \\
    \norm{\Nd_L\Nd \log\Om}_{L^2(\HH)} \les & \,\varep. \label{eqlemmaHNLimprov32}
  \end{align}
\end{lemma}
\begin{proof} Consider first \eqref{eqlemmaHNLimprov31}. We want to derive an equation for $L(\log\Om)$ and apply Lemma~\ref{lem:heat}.
  Commuting equation~\eqref{eq:CMAFbis} with $L$ gives the following elliptic equation for $L(\log\Om)$
  \begin{align*}
    \Ld\big(L(\log\Om)\big) = & F
  \end{align*}
  with the source term $F$
  \begin{align*}
    F := L(\muc) + 2L(\rhoc) + [\Ld,L]\log\Om -2 L(\rhotco) - L(\muco) .
  \end{align*}
  Using commutation formula~\eqref{eq:DLDelta}, we have
  \begin{align}
    [\Ld,{L}]\log\Om = & -\trchit \Ld(\log \Om) -2\chih\cdot\Nd^2\log\Om \nonumber \\
    & +(\zet+\etabt) \left(\Nd (L\log\Om) + \Nd_L\Nd\log\Om\right) \nonumber \\
    & + (\trchit\etabt-\Divd\chibh)\cdot\Nd\log\Om - \etabt \cdot \chi \cdot \Nd \log\Om \nonumber \\
    &  + (\Divd\zet + \Divd\etabt)L(\log\Om) + \beta\cdot\Nd\log\Om \nonumber \\
    = & -2v^{-1}\Ld\log\Om  \nonumber \\
    & - \left(\trchit-\frac{2}{v}\right)\Ld\log\Om -2\chih\cdot\Nd^2\log\Om \nonumber \\
    & +(\zet+\etabt) \left(\Nd (L\log\Om) + \Nd_L\Nd\log\Om\right) \label{eq:LLlogOm12} \\
    & + (\trchit\etabt-\Divd\chibh)\cdot\Nd\log\Om - \etabt \cdot \chi \cdot \Nd \log\Om \nonumber \\
    &+ (\Divd\zet + \Divd\etabt)L(\log\Om) + \beta\cdot\Nd\log\Om.\nonumber 
  \end{align}
  Using Bianchi equation~\eqref{eq:DLrhoc} for $\rhoc$, and the transport equation~\eqref{eq:DLmuc} for $\muc$, we obtain the formula
  \begin{align}\label{eq:LLlogOm13}
    \begin{aligned}
    L(\muc) + 2L(\rhoc) = & -2v^{-1}\muc -4v^{-1}\rhoc -v^{-1}\rhotco + 2\Divd\beta \\
                          & -\left(\trchit-\frac{2}{v}\right)\muc-2\bigg(\trchit-\frac{2}{v}\bigg)\rhoc  \\
                          & +4\etabt\cdot\beta + (\zet-\etabt)\cdot\Nd\trchit + \chih\cdot\Nd\zet \\
                          & -\chih\cdot\Nd\etabt + \trchit\left(|\ze|^2-\zet\cdot\etabt - \half|\etabt|^2\right) \\
                          &+ \frac{1}{4}\trchibt|\chih|^2+2\chih\cdot\zet\cdot\etabt -\frac{3}{2}\chih\cdot\etabt\cdot\etabt.
    \end{aligned}
  \end{align}
  Using Proposition~\ref{prop:Loverline}, we can moreover schematically write
  \begin{align}\label{eq:LLlogOm11}
    \begin{aligned}
      L(\muco) + 2L(\rhoco) = & \overline{L(\muc)+2L(\rhoc)} + AR.
    \end{aligned}
  \end{align}



  Using the three equations \eqref{eq:LLlogOm12}, \eqref{eq:LLlogOm13} and \eqref{eq:LLlogOm11} we deduce that $F$ can be rewritten in the following schematic form
  \begin{align*}
    F = F^{L} + F^{NL},
  \end{align*}
  with the linear source terms $F^L$ 
  \begin{align*}
    F^L = -2v^{-1}\Ld\log\Om -2v^{-1}(\muc-\muco) -4v^{-1}(\rhoc-\rhoco) + 2\Divd\beta ,
  \end{align*}
  and the non-linear source terms $F^{NL}$ being of the form
  \begin{align*}
    F^{NL} = & \Nd A L\log\Om + (A+A^2)L(\log\Om) + A \Nd L\log\Om + A R + A\Nd A + A^2 + A^3.
  \end{align*}

  On the one hand, it holds that
  \begin{align*}
    F^{L} = \Divd P^L + W^L,
  \end{align*}
  with
  \begin{align*}
    P^L & = 2\betat -3v^{-1}\Nd\log\Om,\\
    W^L & = - 4v^{-1}(\rhotc-\rhotco) - 2v^{-1}(\muc-\muco),
  \end{align*}
  and using the already improved bounds~\eqref{est:ImpL2curv},~\eqref{est:checkcurv},~\eqref{est:murenorm} and ~\eqref{est:ImpOm1}, we get
  \begin{align}\label{est:PLWL}
    \begin{aligned}
    \norm{P^L}_{L^2(\HH)} + \norm{W^L}_{L^2_v L^{4/3}} \les & \norm{\beta}_{L^2(\HH)} + \norm{\Nd\log\Om}_{L^2(\HH)}  \\
                                                                 & + \norm{\rhotc}_{L^2(\HH)}+\norm{\muc}_{L^2 L^\infty_v} \\
                                                                 \les & \, \varep.
                                                                 \end{aligned}
  \end{align}

  On the other hand, we have
  \begin{align*}
    F^{NL} = \Nd P^{NL} + W^{NL},
  \end{align*}
  with
  \begin{align*}
    P^{NL} & = A L(\log\Om) \\
    W^{NL} & = (A+A^2)L(\log\Om) + A \Nd L\log\Om + A R + A\Nd A + A^2 + A^3,
  \end{align*}
  and using the bootstrap assumptions~\eqref{eq:CTMAFthmBA} and the estimates~\eqref{est:ImpL2curv}, we get
  \begin{align}\label{est:PNLWNL}
    \begin{aligned}
      \norm{P^{NL}}_{L^2(\HH)} + \norm{W^{NL}}_{L^2_v L^{4/3}} \les & \norm{A}_{L^\infty_v L^4}\norm{L(\log\Om)}_{L^2_v L^4} \\
      & + \left(\norm{A}_{L^\infty_v L^2}+ \norm{A}^2_{L^\infty_v L^4}\right) \norm{L(\log\Om)}_{L^2_v L^{4}} \\
    & + \norm{A}_{L^\infty_v L^4}\bigg(\norm{\Nd L(\log\Om)}_{L^2(\HH)} + \norm{R}_{L^2(\HH)} \\
    &+ \norm{\Nd A}_{L^2(\HH)} + \norm{A}_{L^2(\HH)} + \norm{A}_{L^4(\HH)}^2 \bigg)\\
    \les & (D\varep)^2 + (D\varep)^3 \\
    \les & \,\varep.
    \end{aligned}
  \end{align}
  Applying Lemma~\ref{lem:heat}, using the bounds ~\eqref{est:PLWL} and \eqref{est:PNLWNL}, we have
  \begin{align}\label{est:NdLlogOmpf}
    \begin{aligned}
    \norm{\Nd L(\log\Om)}_{L^2(\HH)} +  \norm{L(\log\Om)-\overline{L(\log\Om)}}_{L^2(\HH)}  \les &  \norm{P^L}_{L^2(\HH)} + \norm{W^{L}}_{L^2_v L^{4/3}} \\
                                                                        & + \norm{P^{NL}}_{L^2(\HH)} + \norm{W^{NL}}_{L^2_v L^{4/3}} \\
                                                                        \les & \, \varep.
                                                                      \end{aligned}
  \end{align}

  Using~\ref{prop:Loverline} and the equation~(\ref{eq:CMAF}), we have
  \begin{align*}
    \overline{L(\log\Om)} = & \, \overline{\Om^{-1}L(\log\Om)} + \overline{(1-\Om^{-1})L(\log\Om)} \\
    = & \, \Om^{-1}L(\overline{\log\Om}) - \overline{\Om^{-1}\trchit \log\Om} + \overline{\Om^{-1}\trchit}\cdot\overline{\log\Om} + \overline{(1-\Om^{-1})L(\log\Om)} \\
    = & -\overline{\Om^{-1}\trchit\log\Om} + \overline{(1-\Om^{-1})L(\log\Om)}.
  \end{align*}
  Using the improved bound (\ref{est:ImpOm1}) and the bootstrap assumptions~\eqref{eq:CTMAFthmBA} we therefore deduce
  \begin{align*}
    \norm{\overline{L(\log\Om)}}_{L^2(\HH)} \les & \norm{\log\Om}_{L^2(\HH)}\norm{\Om^{-1}}_{L^\infty(\HH)}\norm{\trchi}_{L^\infty(\HH)} \\
                                                 & + \norm{\Om-1}_{L^\infty(\HH)}\norm{\Om^{-1}}_{L^\infty(\HH)}\norm{L(\log\Om)}_{L^2(\HH)}\\
    \les & \, \varep + (D\varep)^2\\
    \les & \, \varep.
  \end{align*}
  Using~\eqref{est:NdLlogOmpf}, we finally get
  \begin{align*}
    \norm{\Nd L(\log\Om)}_{L^2(\HH)} + \norm{L(\log\Om)}_{L^2(\HH)} \les \varep.
  \end{align*}

  By the above bounds, commutation formula~\eqref{eq:DLDB} and the bootstrap assumptions~\eqref{eq:CTMAFthmBA} we also get that
  \begin{align*}
    \norm{\Nd_L\Nd \log\Om}_{L^2(\HH)} \les & \norm{\Nd L(\log\Om)}_{L^2(\HH)} + \norm{\tr \chi}_{L^\infty(\HH)}\norm{\Nd \log\Om}_{L^2(\HH)} \\
                                            & + \norm{\chiht}_{L^\infty_v L^4} \norm{\Nd\log\Om}_{L^\infty_v L^4} \\
                                            &+ \left(\norm{\zet}_{L^\infty_v L^4}+\norm{\etabt}_{L^\infty_v L^4} \right) \norm{L(\log\Om)}_{L^2_v L^4} \\
    \les & \, \varep + (D\varep)^2 \\
    \les & \, \varep.
  \end{align*}

  This finishes the proof of Lemma \ref{lemmaHNLimprov4}. \end{proof}

\begin{lemma} \label{lemmaHNLimprov6} For $\varep>0$ sufficiently small, we have
\begin{align}
\norm{\log\Om}_{L^\infty(\HH)} + \norm{\Om-1}_{L^\infty(\HH)} + \norm{L(\log\Om)}_{L^2_vL^4} & \les \varep. \label{eqlemmaHNLimprov52}
\end{align}
\end{lemma}

\begin{proof} 
This is a consequence of estimates \eqref{est:ImpOm1}, \eqref{eqlemmaHNLimprov31}, \eqref{eqlemmaHNLimprov32} and Sobolev embeddings from Lemma~\ref{lem:sob}. 
\end{proof}


\begin{lemma} \label{lemmaHNLimprov5} For $\varep>0$ sufficiently small, it holds that
  \begin{align}
    \NN_1(\zet) + \NN_1(\etabt) & \les \varep. \label{est:Impetabt}
  \end{align}
\end{lemma}

\begin{proof}
By equations \eqref{def:mu} and \eqref{eq:Curleta}, $\zet$ satisfies the Hodge system
  \begin{align*}
    \begin{aligned}
      \Dd_1\big(\zet\big) = & \big(- \mu - \check{\rho},\sigmatc\big).
    \end{aligned}
  \end{align*}
  Using Lemma~\ref{lemmaHODGEGAMMA}, the improved bounds \eqref{est:checkcurv} and (\ref{est:murenorm}), we get that
  \begin{align*}
    \norm{\zet}_{L^2(\HH)} + \norm{\Nd\zet}_{L^2(\HH)}  \les& \norm{\mut}_{L^2(\HH)} + \norm{\rhotc}_{L^2(\HH)} + \norm{\sigmatc}_{L^2(\HH)} \\
    \les& \,\varep.
  \end{align*}
  By relation~\eqref{eq:zetetabt}, the improved bounds~\eqref{est:ImpOm1} and the above improvement, we directly deduce
  \begin{align*}
    \norm{\etabt}_{L^2(\HH)} + \norm{\Nd\etabt}_{L^2(\HH)} \les & \, \varep.
  \end{align*}
  Using equation~\eqref{eq:DLze}, the bootstrap assumptions~\eqref{eq:CTMAFthmBA}, the curvature bounds~\eqref{est:ImpL2curv} and the just obtained improved estimate for $\zet$ and $\etabt$, we have
  \begin{align*}
    \norm{\Nd_L\zet}_{L^2(\HH)} \les & \norm{\trchi}_{L^\infty}(\norm{\zet}_{L^2(\HH)}+\norm{\etabt}_{L^2(\HH)}) \\
                                     & + \norm{\chih}_{L^\infty_v L^4}(\norm{\zet}_{L^\infty_v L^4}+\norm{\etabt}_{L^\infty_v L^4})+\norm{\betat}_{L^2(\HH)} \\
    \les& \, \varep + (D\varep)^2\\
    \les & \,  \varep.
  \end{align*}
  By relation~\eqref{eq:zetetabt} and the improved bounds~\eqref{eqlemmaHNLimprov32}, we therefore also deduce
  \begin{align*}
    \norm{\Nd_L\etabt}_{L^2(\HH)} \les \varep,
  \end{align*}
  and this finishes the proof of Lemma \ref{lemmaHNLimprov5}.
\end{proof}


\begin{lemma}\label{lem:trchitL2}
  For $\varep>0$ sufficiently small, it holds that
  \begin{align}
    \norm{\Nd\trchit}_{L^2 L^\infty_v} &\les \varep,\label{eqlemmaNDestimatesLinfL21}\\
    \norm{\trchit-\frac{2}{v}}_{L^\infty(\HH)} &\les \varep, \label{est:Linftrchit}\\
    \NN_1\big(\trchit-\frac{2}{v}\big) + \NN_1(\chih) &\les \varep.\label{est:NN1chih}
  \end{align}
\end{lemma}
\begin{proof}
Consider first \eqref{eqlemmaNDestimatesLinfL21}. Commuting equation~\eqref{eq:DLtrchi} with $\Nd$, we get
\begin{align*}
  \Nd_L(\Nd\trchit) +\frac{3}{2}\trchit\Nd\trchit = G,
\end{align*}
with 
\begin{align*}
  G = -2\chiht\cdot\Nd\chiht -\chiht\cdot\Nd\trchit +(\zet+\etabt)\bigg(\half(\trchit)^2+|\chiht|^2\bigg).
\end{align*}

By the improved bounds~\eqref{est:Impetabt} for $\zet$ and $\etabt$ and the bootstrap assumptions~\eqref{eq:CTMAFthmAssum1}, we have
\begin{align*}
\norm{G}_{L^2 L^1_v} \les & \norm{\chiht}_{L^\infty L^2_v}\norm{\Nd\chiht}_{L^2(\HH)} + \norm{\chiht}_{L^\infty L^2_v}\norm{\Nd\trchit}_{L^2(\HH)} \\
                          & + (\norm{\zet}_{L^2(\HH)} + \norm{\etabt}_{L^2(\HH)})\norm{\trchit}_{L^\infty(\HH)}^2\\
                          & +(\norm{\zet}_{L^\infty L^2_v} + \norm{\etabt}_{L^\infty L^2_v})\norm{\chiht}^2_{L^4(\HH)} \\
\les & \, \varep + (D\varep)^2 \\
\les & \, \varep.
\end{align*}
Therefore, we deduce from Lemma~\ref{lemmaTRANSPORT} with \eqref{eq:CTMAFthmAssum1} that \eqref{eqlemmaNDestimatesLinfL21} holds.

Next, we consider~\eqref{est:Linftrchit}.
The transport equation for $\trchit$~\eqref{eq:DLtrchi} can be rewritten
\begin{align}\label{eq:Ltrchit2v}
  L\bigg(\trchit-\frac{2}{v}\bigg) +\trchit\bigg(\trchit-\frac{2}{v}\bigg) = 2v^{-2}(\Om-1) -|\chih|^2 + \half\bigg(\trchit-\frac{2}{v}\bigg)^2.
  \end{align}
  Using Lemma~\ref{lem:transport}, the bootstrap assumptions \eqref{eq:CTMAFthmBA}, the bounds on $S_1$ \eqref{est:InitHNLL2}, and the improved bound \eqref{eqlemmaHNLimprov52}, we have
  \begin{align*}
    \norm{\trchit-\frac{2}{v}}_{L^\infty(\HH)} \les & \norm{\trchit-2}_{L^\infty(S_1)} + \norm{\Om-1}_{L^\infty(\HH)} + \norm{\chih}^2_{L^\infty L^2_v} + \norm{\trchit-\frac{2}{v}}^2_{L^\infty(\HH)} \\
    \les & \, \varep +(D\varep)^2 \\
    \les & \, \varep,
  \end{align*}
  which proves~\eqref{est:Linftrchit}.
  
  It remains to prove~\eqref{est:NN1chih}. Using transport equation \eqref{eq:Ltrchit2v} for $\trchit$, we deduce that
  \begin{align*}
    \norm{L\bigg(\trchit-\frac{2}{v}\bigg)}_{L^2(\HH)} \les \varep,
  \end{align*}
  and therefore
  \begin{align*}
    \NN_1\bigg(\trchit-\frac{2}{v}\bigg) \les \varep.
  \end{align*}
  Applying Hodge Lemma~\ref{lem:Hodge} to the Codazzi equation on $\chih$~\eqref{eq:Divdchih}, with the curvature bounds~\eqref{est:ImpL2curv}, the improved bound ~\eqref{est:Impetabt}, and the bound just proven for $\Nd\trchit$ gives
  \begin{align*}
    \norm{\Nd\chih}_{L^2(\HH)} + \norm{\chih}_{L^2(\HH)} \les & \norm{\Nd\trchit}_{L^2L^\infty_v} + \norm{\zet}_{L^4L^\infty_v}\norm{\chih}_{L^4L^\infty_v} \\
                                                              &+ \norm{\trchit}_{L^\infty(\HH)}\norm{\zet}_{L^2(\HH)} +\norm{\beta}_{L^2(\HH)} \\
    \les & \, \varep + (D\varep)^2\\
    \les & \, \varep.
  \end{align*}
  Taking directly the $L^2(\HH)$-norm in the transport equation for $\chih$ ~\eqref{eq:DLchih}, we finally obtain $\norm{\Nd_L\chih}_{L^2(\HH)}\les \varep$ and this finishes the proof of Lemma \ref{lem:trchitL2}.
\end{proof}


\begin{lemma}\label{lem:ImpchibL2}
  For $\varep>0$ sufficiently small, it holds that
  \begin{align}
    \norm{\Nd\trchibt}_{L^2 L^\infty_v} & \les \varep,\label{eqlemmaNDestimatesLinfL22} \\
    \norm{\trchibt+\frac{2}{v}}_{L^\infty(\HH)} & \les \varep, \label{eqlemmaRemainingchib1}\\
    \NN_1\bigg(\trchibt+\frac{2}{v}\bigg) + \NN_1(\chibh) & \les \varep. \label{eqlemmaRemainingchib3}
  \end{align}
\end{lemma}
\begin{proof}
  Consider \eqref{eqlemmaNDestimatesLinfL22}. Commuting the transport equation for $\trchibt$~\eqref{eq:DLtrchibt} with $\Nd$, we get
  \begin{align*}
    \Nd_L\Nd\trchibt + \trchit\Nd\trchibt = & G,
  \end{align*}
  with
  \begin{align*}
    G =  4\etabt\cdot\Nd\etabt -\chiht\cdot\Nd\trchibt+(\zet+\etabt)\bigg(-\half\trchit\trchibt +4\rhotco +4|\etabt|^2\bigg),
  \end{align*}
  which can be rewritten in the schematic form
  \begin{align*}
    G =  2v^{-2}(\zet+\etabt) + A(\nab A + A + A^2+R).
  \end{align*}
  
  Using Lemma~\ref{lem:transport}, the bootstrap assumptions~\eqref{eq:CTMAFthmBA}, the initial bounds~\eqref{est:InitHNLL2} and the improved bounds ~\eqref{eqlemmaHNLimprov31} \eqref{est:Impetabt}, we have
  \begin{align*}
    \norm{\Nd\trchibt}_{L^2 L^\infty_v} \les & \norm{\Nd\trchibt}_{L^2(S_1)} + \norm{G}_{L^2 L^1_v} \\
    \les & \,\varep + \norm{\zet}_{L^2(\HH)} + \norm{\etabt}_{L^2(\HH)} \\
                                             & + \norm{A}_{L^\infty L^2_v}(\norm{\nab A}_{L^2(\HH)} + \norm{A}_{L^2(\HH)} + \norm{A}^2_{L^4(\HH)} + \norm{R}_{L^2(\HH)}) \\
    \les & \,\varep + (D\varep)^2 + (D\varep)^3\\
    \les & \,\varep.
  \end{align*}
  
  We turn to estimate~\eqref{eqlemmaRemainingchib1}. 
  The transport equation for $\trchibt$~\eqref{eq:DLtrchibt} can be rewritten in the following form
  \begin{align*}
    &  L\bigg(\trchibt+\frac{2}{v}\bigg) +\half\trchit\bigg(\trchibt+\frac{2}{v}\bigg) \\
     = & - 2v^{-2}(\Om-1) +v^{-1}\bigg(\trchit-\frac{2}{v}\bigg) +2\rhotco +2|\etabt|^2.
  \end{align*}
  Using Lemma~\ref{lem:transport}, the bootstrap assumptions \eqref{eq:CTMAFthmBA}, the bounds on $S_1$ \eqref{est:InitHNLL2} and the improved bounds ~\eqref{est:Linftrchit}~\eqref{eqlemmaHNLimprov52}, we have
  \begin{align*}
    \norm{\trchibt+\frac{2}{v}}_{L^\infty(\HH)} \les & \norm{\trchibt+2}_{L^\infty(S_1)}  + \norm{\Om-1}_{L^\infty(\HH)} \\
                                                                   & + \norm{\trchit-\frac{2}{v}}_{L^\infty(\HH)} + \norm{\rhotco}_{L^\infty L^1_v} + \norm{\etabt}^2_{L^\infty L^2_v} \\
    \les & \,\varep + (D\varep)^2 \\
    \les & \,\varep.
  \end{align*}
  
  To prove estimate \eqref{eqlemmaRemainingchib3}, we apply Hodge Lemma~\ref{lem:Hodge} to the Codazzi equation for $\chibh$ and since $\Nd\trchibt$ and $\trchibt$ have already been estimated the $\Nd$-control of $\trchib+\frac{2}{v}$ and $\chibh$ follows. The estimates for $L(\trchibt+\frac{2}{v})$ and $\Nd_L\chibh$ are obtained by taking directly the $L^2(\HH)$ norm in the transport equations for $\trchibt$ and $\chibh$ ~\eqref{eq:DLtrchibt} and ~\eqref{eq:DLchibh} since all linear source terms have already been estimated. This concludes the proof of Lemma~\ref{lem:ImpchibL2}. 
\end{proof}

\subsection{Improvement of $\Ups$}\label{sec:ImpNdv}
In this section, we improve the estimate for $\norm{\Ups}_{L^\infty(\HH)}$ which is the key quantity to compare the geodesic and canonical foliations. 
  Using the estimates proved in the previous sections, we can first improve the $L^\infty L^2_v$ estimate for $\etabt$.
\begin{lemma} \label{lemmaHNLimprov7}
For $\varep>0$ sufficiently small, it holds that
  \begin{align}\label{est:etabtLinfL2}
    \norm{\etabt}_{L^\infty L^2_v} \les \varep.
  \end{align}
\end{lemma}
\begin{proof}
  Our goal is to apply the trace estimate of Lemma~\ref{lem:traceBesov}.

  

  By the improved estimates \eqref{est:Impetabt} for $\etabt$, it suffices to prove that there exist $P$ and $E$ such that
  \begin{align*}
    \Nd\etabt = \Nd_LP + E,
  \end{align*}
  with $\NN_1(P) \les \varep$ and $\norm{E}_{\PPo} \les \varep$.
  From the transport equation for $\chib$ ~\eqref{eq:DLchib}, we have
  \begin{align*}
    \Nd_A\etabt_B = & \half \Nd_L\chib_{AB} - \half\rho \gd_{AB} -\half\sigma\iin_{AB} +\half\chi_{AC}\chib_{CB} \\
    = & \half \Nd_L\bigg(\chib_{AB}+\frac{2}{v}\gd_{AB}\bigg) - \half\rho\gd_{AB} - \half\sigma\iin_{AB} + E_{AB},
  \end{align*}
  where
  \begin{align*}
    E_{AB} := & \bigg(\half\trchibt\bigg(\trchit-\frac{2}{v}\bigg) +2v^{-2}(\Om-1)+ \half\trchit\bigg(\trchibt+\frac{2}{v}\bigg) + \half\bigg(\trchit-\frac{2}{v}\bigg)\bigg(\trchibt+\frac{2}{v}\bigg)\bigg)\gd_{AB}\\
                    & +\frac{1}{4}\trchibt\chih_{AB}+\frac{1}{4}\trchit\chibh_{AB}+\half\chih_{AC}\chibh_{CB}.
  \end{align*}

  First, using the results of Lemma~\ref{lem:ImpchibL2}, we have
  \begin{align*}
    \NN_1\bigg(\chib+\frac{2}{v}\gd\bigg) \les \varep.
  \end{align*}
  Second, using Lemma~\ref{lem:prodBesov} and the improved bounds for $\Om$, $\chi$ and $\chib$ ~\eqref{eqlemmaHNLimprov52}, ~\eqref{est:NN1chih} ~\eqref{eqlemmaRemainingchib3}, we have
  \begin{align*}
    \norm{E}_{\PPo} \les & \bigg(\NN_1(\Om-1) + \NN_1\bigg(\chi-\frac{2}{v}\gd\bigg) + \NN_1\bigg(\chib+\frac{2}{v}\gd\bigg)\bigg) \\
    & \times\bigg(1 + \NN_1(\Om-1) + \NN_1\bigg(\chi-\frac{2}{v}\gd\bigg) + \NN_1\bigg(\chib+\frac{2}{v}\gd\bigg)\bigg) \\
                    \les & \, \varep.
  \end{align*}

  Third, we define $(\phi,\psi)$ to be the solution of the transport equation
  \begin{align}\label{eq:DLphipsi}
    \begin{aligned}
      L\phi & = \rho, \\
      L\psi & = \sigma,\\
      (\phi,\psi)|_{S_1} & = \Dds_1^{-1}\betabc.
    \end{aligned}
  \end{align}

  Using the curvature bounds ~\eqref{est:ImpL2curv}, we have directly
  \begin{align}\label{est:Lphipsi}
    \norm{L\phi}_{L^2(\HH)} + \norm{L\psi}_{L^2(\HH)} \les \varep
  \end{align}

  Using the definition of $\betabc$ \eqref{eq:rhotc} and the Codazzi equation for $\chibh$ \eqref{eq:divchibh}, we have schematically
  \begin{align*}
    \betabc = \Divd\chibh -\half\Nd\trchibt + A + A\Ab.
  \end{align*}
  Therefore, using Lemma~\ref{lem:Hs} and the bounds on $S_1$ ~\eqref{est:InitHNLL2}, we have
  \begin{align*}
    \norm{\betabc}_{H^{-1/2}(S_1)} & \les \norm{\chibh}_{H^{1/2}(S_1)} + \norm{\trchibt+\frac{2}{v}}_{H^{1/2}(S_1)} + \norm{A\Ab}_{L^2(S_1)} \\
    & \les \norm{\chih}_{H^{1/2}(S_1)} + \norm{\trchibt+\frac{2}{v}}_{H^{1/2}(S_1)} +\norm{A}_{L^4(S_1)}\norm{\Ab}_{L^4(S_1)} \\
    & \les \norm{\chibh}_{H^{1/2}(S_1)} + \norm{\trchibt+\frac{2}{v}}_{H^{1/2}(S_1)} + \norm{A}_{H^{1/2}(S_1)}\norm{\Ab}_{H^{1/2}(S_1)} \\
                                      & \les \varep.
  \end{align*}

  Thus, using Lemma ~\ref{lem:BochnerBesovest}, we have
  \begin{align}\label{est:S1phipsi}
    \norm{(\phi,\psi)}_{H^{1/2}(S_1)} \les \norm{\Dds_1^{-1}\betabc}_{H^{1/2}(S_1)} \les \norm{\betabc}_{H^{-1/2}(S_1)} \les \varep. 
  \end{align}
  Using the transport Lemma~\ref{lem:transport} with these bounds, we deduce
  \begin{align}\label{est:phipsi}
    \norm{(\phi,\psi)}_{L^2 L^\infty_v} \les \norm{L(\phi,\psi)}_{L^2(\HH)} + \norm{(\phi,\psi)}_{H^{1/2}(S_1)} \les \varep.
  \end{align}

  Commuting the transport equation~\eqref{eq:DLphipsi} by $\Dds_1$, using Bianchi equation~\eqref{eq:DLbetabc} for $\betabc$ and commutation formula~\eqref{eq:DLDBscal} gives
  \begin{align*}
    \Nd_L\Dds_1(\phi,\psi) = & \Dds_1(\rho,\sigma) + [\Nd_L,\Dds_1](\phi,\psi) \\
    = & \Nd_L\betabc + \trchit\betabc + A\big(R + \nab\Ab + \Ab + \Ab^2\big) + A\nab(\phi,\psi). 
  \end{align*}
  Using Lemma~\ref{lem:transport}, the bootstrap assumptions~\eqref{eq:CTMAFthmBA}, the curvature bounds~\eqref{est:checkcurv} and the condition ~\eqref{eq:DLphipsi} on $S_1$ for $(\phi,\psi)$, we have
  \begin{align*}
    \norm{\Dds_1(\phi,\psi)-\betabc}_{L^\infty_v L^2} \les & \norm{\Dds_1(\phi,\psi)-\betabc}_{L^2(S_1)} + \norm{\trchit}_{L^\infty(\HH)} \norm{\betabc}_{L^2(\HH)} \\
                                                           & + \norm{A}_{L^\infty L^2_v}(\norm{R}_{L^2(\HH)} + \norm{\nab \Ab}_{L^2(\HH)} + \norm{\Ab}_{L^2(\HH)} + \norm{\Ab}^2_{L^4(\HH)}) \\
                                                           & + \norm{A}_{L^\infty L^2_v}\norm{\nab(\phi,\psi)}_{L^2(\HH)} \\
    \les & \, \varep + (D\varep)^2 + (D\varep)\NN_1(\phi,\psi).    
  \end{align*}

  Using Hodge Lemma~\ref{lem:Hodge} and the curvature bound~\eqref{est:checkcurv}, we deduce from the above that
  \begin{align}\label{est:Ndphipsi}
    \begin{aligned}
    \norm{\Nd(\phi,\psi)}_{L^2(\HH)} &\les \norm{\Dds_1(\phi,\psi)}_{L^2(\HH)} \\
                                     & \les \norm{\betabc}_{L^2(\HH)} + \norm{\Dds_1(\phi,\psi)-\betabc}_{L^\infty L^2_v} \\
                                     & \les \varep + D\varep\NN_1(\phi,\psi).
                                   \end{aligned}
  \end{align}
  For $\varep>0$ sufficiently small, we therefore have by \eqref{est:Lphipsi}, \eqref{est:S1phipsi}, \eqref{est:phipsi} and \eqref{est:Ndphipsi}, that
  \begin{align*}
    \NN_1(\phi,\psi) \les \varep.
  \end{align*}
  This finishes the proof of the lemma.
\end{proof}

\begin{lemma}\label{lem:ImpUps}
  We have the improved bound
  \begin{align}\label{est:ImpUps}
    \norm{\Ups'}_{L^\infty(\HH)} + \norm{\Ups}_{L^\infty(\HH)} \les \varep.
  \end{align}
\end{lemma}
\begin{proof} 
  From Proposition~\ref{prop:AA'}, we have
  \begin{align*}
    \Nd'_L\Ups'_A & = \etabt'_A-(\etabt)^\ddg_A \\
                  & = \etabt'_A - \etabt(\etA).
  \end{align*}
  Integrating in $s$, we deduce
  \begin{align*}
    \Ups'_A = & \int_{1}^s L(\Ups'_A) \,\text{d}s' \\
    = & \int_{1}^s (\etabt'_A-\etabt(\etA)) \,\text{d}s' \\
    = & \int_{1}^s\etabt'_A \,\text{d}s'- \int_{v'=1}^{v(s)}\etabt_A\Om^{-1}\,\text{d}v'.
  \end{align*}
  Therefore, using the assumed bound~\eqref{eq:CTMAFthmAssum1} on the geodesic connection coefficient $\etab' = -\zet'$ and the improved bound for $\etabt$ ~\eqref{est:etabtLinfL2}, we obtain
  \begin{align*}
    \norm{\Ups'}_{L^\infty(\HH)} \les & \norm{\etabt'}_{L^\infty L^2_s} +\norm{\Om^{-1}}_{L^\infty(\HH)} \norm{\etabt}_{L^\infty L^2_v}) \\
    \les & \, \varep,
  \end{align*}
  which, together with Lemma~\ref{lem:UpsUps'} proves~\eqref{est:ImpUps}.
\end{proof}
\subsection{Improvement of $L^\infty L^2_v$ estimates for $A$} \label{SECcomparisonEstimates} In section~\ref{sec:ImpNdv}, we proved $L^\infty L^2_v$ estimate for $\etabt$. In this section, we prove the remaining $L^\infty L^2_v$ estimates for 
$\chih, \zeta, \Nd \log\Om$ by comparing the canonical foliation to the geodesic foliation on $\HH$. This concludes the improvement of the bootstrap assumptions~\eqref{eq:CTMAFthmBA}. 

\begin{lemma} \label{lemmaConclusionComparisonHNL} For $\varep>0$ sufficiently small, we have
  \begin{align}
    \norm{\chiht}_{L^\infty L^2_v} & \lesssim \varep, \label{eqlemmaConclusionComparisonHNL2} \\
    \norm{\zet}_{L^\infty L^2_v} & \lesssim \varep, \label{eqlemmaConclusionComparisonHNL3} \\
    \norm{\Nd \log\Om}_{L^\infty L^2_v} &\les \varep.\label{eqlemmaConclusionComparisonHNL4}
  \end{align}
\end{lemma}

\begin{proof}
  From Proposition~\ref{prop:AA'}, we have
  \begin{align*}
    \chi_{AB} = & \chi'_{AB} \\
    \zet_A = & \zet'_A -2\Ups'_B\chi'_{AB},
  \end{align*}
  Thus, the estimate \eqref{eqlemmaConclusionComparisonHNL2} is direct.\\

  Further, by the improved bound on $\Ups'$ ~\eqref{est:ImpUps} and the assumed bound \eqref{est:boundedL2source} for the geodesic connection coefficients, we get 
\begin{align*}
\norm{\zet}_{L^\infty L^2_v} \les \norm{\zet'}_{L^\infty L^2_s}+\norm{\Ups'}_{L^\infty(\HH)}\norm{\chi'}_{L^\infty L^2_s} \les \varep.
\end{align*}
Estimate \eqref{eqlemmaConclusionComparisonHNL4} then follows directly using relation \eqref{eq:zetetabt}. This finishes the proof of Lemma \ref{lemmaConclusionComparisonHNL}. \end{proof}

\subsection{Additional bounds for $\Ups$}
In Section~\ref{sec:apriorihigher}, we will use the following additional estimates.
\begin{lemma}\label{lem:AddUps}
  For $\varep>0$ sufficiently small, we have
  \begin{align}\label{est:AddUps}
    \norm{\Nd_L\Ups}_{L^2(\HH)} + \norm{\Nd\Ups}_{L^2 L^\infty_v} \les \varep. 
  \end{align}
\end{lemma}
\begin{proof}
  Taking the $L^2(\HH)$-norm in the transport equation~\eqref{eq:DLUps} for $\Ups$, using the improved bound \eqref{est:ImpOm1} for $\Om$ and the improved bound on $\Ups$ \eqref{est:ImpUps}, we have
  \begin{align*}
    \norm{\Nd_L\Ups}_{L^2(\HH)} \les & \norm{\trchit}_{L^\infty(\HH)}\norm{\Ups}_{L^\infty(\HH)} + \norm{\chih}_{L^2 L^\infty_v} \norm{\Ups}_{L^\infty(\HH)} + \norm{\Nd\log\Om}_{L^2(\HH)} \\
    \les & \, \varep + \varep^2\\
    \les & \, \varep.
  \end{align*}

  To obtain the other bound, we make the additional bootstrap assumption $\norm{\Nd\Ups}_{L^2 L^\infty_v} \leq D\varep$.
  We commute the transport equation~\eqref{eq:DLUps} by $\Nd$ 
  \begin{align}\label{eq:DLNdUps}\begin{aligned}
    \Nd_L\Nd_A\Ups_B = & -\half\trchit\Nd_A\Ups_B - \half\Nd_A\trchit\Ups_B + \Nd_A\Nd_B\log\Om \\
    & -\Nd_A\chih_{BC}\Ups_C -\chih_{BC}\Nd_A\Ups_C + [\Nd_L,\Nd]_A\Ups_B,
    \end{aligned}
  \end{align}
  where by using formula~\eqref{eq:DLDB} and~\eqref{eq:DLUps} we have
  \begin{align*}
    [\Nd_L,\Nd]_A\Ups_B = & -\half\trchit\Nd_A\Ups_B - \chih_{AC}\Nd_C\Ups_B -|\Nd(\log\Om)|^2 +\chi\cdot\Nd\log\Om\cdot\Ups - \dual\beta_A\dual\Ups_B.
  \end{align*}
  Therefore, applying Lemma~\ref{lem:transport}, using that $\Ups =0$ on $S_1$, the improved bounds and the additional bootstrap assumption, we obtain
  \begin{align*}
    \norm{\Nd\Ups}_{L^2 L^\infty_v} \les & \norm{\Nd^2\log\Om}_{L^2(\HH)} + \norm{\Nd\log\Om}^2_{L^2_vL^4} \\
                                         & + \norm{\Ups}_{L^\infty(\HH)}(\norm{\Nd\chi}_{L^2(\HH)} +\norm{\chi}_{L^\infty L^2_v}\norm{\Nd\log\Om}_{L^2(\HH)} + \norm{\beta}_{L^2(\HH)}) \\
                                         & + \norm{\Nd\Ups}_{L^2 L^\infty_v}\norm{\chih}_{L^\infty L^2_v} \\
    \les & \, \varep + (D\varep)^2 \\
    \les & \, \varep,
  \end{align*}
  which improves the additional bootstrap assumption and hence finishes the proof of Lemma \ref{lem:AddUps}.
\end{proof}

\section{Higher regularity estimates}\label{sec:apriorihigher}
This section is dedicated to the proof of Proposition~\ref{prop:higherreg} and completes Step 3 in Section~\ref{sec:overview}.\\
We assume that $(\MM,\g)$ is a smooth spacetime and $\HH$ a smooth null hypersurface foliated by a smooth geodesic foliation.
We assume moreover that the following bounds hold on $[1,v^\ast]$
  \begin{align}\label{est:LowUpsOm}
     \begin{aligned}
    \norm{\Om-1}_{L^\infty_v([1,v^\ast])L^\infty} + \NN_1^{v,v^\ast}(\Nd\log\Om) & \les \varep, \\
    \norm{\Ups}_{L^\infty_v([1,v^\ast])L^\infty} + \norm{\Nd\Ups}_{L^2 L^\infty_v([1,v^\ast])} & \les \varep.
    \end{aligned}
  \end{align}
  For all $m\geq 0$, we will prove the following estimates
  \begin{align}\label{est:higherregm}
    \sum_{l\leq m}\left(\norm{\Nd^l(\Om-1)}_{L^\infty_v([1,v^\ast])L^2} + \norm{\Nd^l\Ups}_{L^\infty_v([1,v^\ast])L^2} \right)\leq C(\norm{\gd'}_{C^{m+100}(\HH)},m).
  \end{align}
  Moreover, we will also have the following estimates on the $L$-derivatives
  \begin{align}\label{est:Ckhigherreg}
    \sum_{l\leq m}\norm{\Nd^lL^k(\Om)}_{L^\infty_v([1,v^\ast])L^2} \leq C(\norm{\gd'}_{C^{m+k+100}(\HH)},m+k),
  \end{align}
 for all $m\geq 0$ and all $k\geq 0$. This will complete the proof of Proposition~\ref{prop:higherreg}.\\
  
Before turning to the proof of~\eqref{est:higherregm} and~\eqref{est:Ckhigherreg}, we prove the following lemma that is a rewriting of equations~\eqref{eq:DLUps} and~\eqref{eq:CMAF}. This will also be used in the proof of the local existence Theorem~\ref{thm:loc}.                                                                                                
\begin{lemma}\label{lem:smootheq}
  We have
  \begin{align}
    \Nd_L\Ups & = -(\chi')^\dg\cdot\Ups - \Nd\log\Om, \label{eq:smoothUps}\\
    \Ld(\log\Om) & = F_1' + (F_2')^\dg\cdot\Ups + (F_3')^\dg\cdot\Ups\cdot\Ups + (F'_4)^\dg\cdot\Nd^2\Ups,\label{eq:smoothOm}
  \end{align}
  where $F'_1,F'_2,F'_3,F'_4$ are (contractions of) geodesic quantities.
\end{lemma}
\begin{remark}
  In this lemma, the specific structure of the terms $(F'_i)^\dg$ is hidden.
  Since from this point on we are not interested in proving sharp estimates anymore, this loss of information is not an issue and rather simplifies the analysis.                   
\end{remark}
\begin{proof}
  Equation~\eqref{eq:smoothUps} is a rewriting of~\eqref{eq:DLUps}.\\
  For equation~(\ref{eq:smoothOm}), we have using the relations from Proposition~\ref{prop:AA'} and the derivatives relation from Proposition~\ref{prop:compNd}
  \begin{align*}
    \Divd\ze = & \Divd\left((\ze')^\dg + (\chi')^\dg\cdot\Ups\right) \\
    = & \Divd'\ze' + (\Nd'_L\ze')^\dg\cdot\Ups + (\trchi')^\dg(\ze')^\dg\cdot\Ups \\
               & - (\chi')^\dg\cdot(\ze')^\dg\cdot\Ups + (\chi')^\dg \cdot \Nd\Ups + (\Divd'\chi')^\dg \cdot \Ups \\
    & + \Nd'_L\chi'\cdot\Ups\cdot\Ups + (\trchi')^\dg(\chi')^\dg\cdot\Ups\cdot\Ups - 2(\chi')^\dg\cdot(\chi')^\dg\cdot\Ups\cdot\Ups.
  \end{align*}
  Using the relations from Proposition~\ref{prop:RR'}, we have
  \begin{align*}
    \rho = & \rho' + (\beta')^\dg\cdot\Ups + (\alpha')^\dg\cdot\Ups\cdot\Ups.
  \end{align*}
  Using the relations from Proposition~\ref{prop:AA'}, we have
  \begin{align*}
    -\half\chih\cdot\chibh = & -\half(\chih')^\dg \cdot \left((\chibh')^\dg -4 (\ze')^\dg\cdot\Ups + 2 \Nd\Ups - |\Ups|^2 (\chi')^\dg\right).  
  \end{align*}
  Therefore, defining
  \begin{align*}
    F'_1 :=& -\Divd'\ze'+ \rho' - \half\chih'\cdot\chibh',\\
    F'_2 := & - \Nd'_L\ze'- \trchi'\ze'+ \chi'\cdot\ze'- (\Divd'\chi')^\dg+ \beta'+ 2 \chih'\cdot\ze',\\
    F'_3 := & + \alpha'+ \half \chih'\cdot\chi'-\half |\chih'|^2 - (\trchi')^\dg(\chi')^\dg +2 (\chi')^\dg\cdot(\chi')^\dg,\\
    F'_4 := & -\trchi'\gd'-2\chih',
  \end{align*}
  gives
  \begin{align*}
    \Ld(\log\Om) & = F_1' + (F_2')^\dg\cdot\Ups + (F_3')^\dg\cdot\Ups\cdot\Ups + (F'_4)^\dg\cdot\Nd\Ups. \\
  \end{align*}
  This finishes the proof of Lemma~\ref{lem:smootheq}.
\end{proof}
                
\begin{proof}[Proof of (\ref{est:higherregm}) and (\ref{est:Ckhigherreg})]
  The proof of~\eqref{est:higherregm} goes by induction on $m$. The cases $m=0$ and $m=1$ were already obtained in Section~\ref{sec:aprioriL2}. We prove the case $m=2$ and the cases $m\geq 3$ are proved similarly and are left to the reader. In what follows, we use that the quantities $F'_1,F'_2,F'_3,F'_4$ appearing in Lemma~\ref{lem:smootheq} are smooth in the geodesic foliation. More precisely, we are going to obtain bounds in terms of
  \begin{align*}
    \sum_{k\leq 100} \bigg(\norm{(\nab')^kF'_1}_{L^\infty(\HH)} + \norm{(\nab')^kF'_2}_{L^\infty(\HH)}+ \norm{(\nab')^kF'_3}_{L^\infty(\HH)} + \norm{(\nab')^kF'_4}_{L^\infty(\HH)}\bigg).
  \end{align*}
  For simplicity, we don't write the exact bound and this quantity shall be always implicitly included in the constants $C$ appearing in the following.\\
  
  First, applying Lemma~\ref{lem:Bochnerest} to elliptic equation~\eqref{eq:CMAF}, we have
  \begin{align*}
    \norm{\Nd^2\log\Om}_{L^\infty_vL^2} \les & \norm{F'_1}_{L^\infty(\HH)} + \norm{F'_2}_{L^\infty(\HH)}\norm{\Ups}_{L^\infty(\HH)} \\
                                             & + \norm{F'_3}_{L^\infty(\HH)}\norm{\Ups}^2_{L^\infty(\HH)} + \norm{F'_4}_{L^\infty(\HH)}\norm{\Nd\Ups}_{L^\infty_vL^2} \\
    \les & C\left(\varep\right).
  \end{align*}
  Second, commuting equation~\eqref{eq:smoothUps} by $\Nd^2$, we have schematically
  \begin{align*}
    \Nd_L\Nd^2\Ups = & -(\chi')^\dg\cdot\Nd^2\Ups + G'(\Ups) \cdot \left(\Nd\Ups+\Nd\log\Om\right)+ \Nd^3(\log\Om) + [\Nd_L,\Nd^2]\Ups,
  \end{align*}
  where $G'(\Ups)$ denotes an arbitrary number of contractions of geodesic quantities with $\Ups$.
  Using commutation formula~\eqref{eq:DLDB}, the commutator can be schematically rewritten
  \begin{align*}
    [\Nd_L,\Nd^2]\Ups = & \Nd [\Nd_L,\Nd] \Ups + [\Nd_L,\Nd] \Nd\Ups \\
    = & G'(\Ups) \cdot \left(\Nd\Ups + \Nd\log\Om\right) \\
     & + (\chi')^\dg \cdot\Nd^2\Ups + \Nd\log\Om \cdot\Nd^2\log\Om.
  \end{align*}
  We therefore obtain the following schematic formula
  \begin{align*}
    \Nd_L\Nd^2\Ups = & (\chi')^\dg\cdot \Nd^2\Ups + \Nd^3(\log\Om) + \Nd\log\Om \cdot \Nd^2\log\Om \\
                     & + G'(\Ups)\cdot\left(\Nd\Ups + \Nd\log\Om\right).
  \end{align*}

  Using Lemma~\ref{lem:transport}, the assumptions~\eqref{est:LowUpsOm}, and the above formulas, we therefore get
  \begin{align}\label{est:Nd2Upssmooth}\begin{aligned}
    \norm{\Nd^2\Ups}_{L^\infty_v L^2} \les & \norm{\chi'}_{L^\infty(\HH)}\norm{\Nd^2\Ups}_{L^1_vL^2} + \norm{\Nd^3\log\Om}_{L^1_vL^2} \\
    &  + \norm{\Nd\log\Om}_{L^\infty L^2_v}\norm{\Nd^2\log\Om}_{L^2(\HH)} \\
    & + \norm{G'(\Ups)}_{L^\infty(\HH)} \left(\norm{\Nd\Ups}_{L^\infty_vL^2}+\norm{\Nd\log\Om}_{L^\infty_vL^2}\right) \\
    \les & C\left(\varep\right)+ C \int_{v}\left(\norm{\Nd^2\Ups}_{L^2(S_v)} + \norm{\Nd^3\log\Om}_{L^2(S_v)}\right)\,\text{d}v.
    \end{aligned}
  \end{align}
  On the other hand, commuting elliptic equation~\eqref{eq:smoothOm}, with $\Nd$, we obtain schematically
  \begin{align*}
    \Ld \Nd \log\Om = & G'(\Ups) \cdot \Nd\Ups + (F'_4)^\dg\cdot\Nd^2\Ups + [\Ld,\Nd]\log\Om,   
  \end{align*}
  where using formula~(\ref{eq:NdDeltascal}) and Propositions~\ref{prop:RR'} and~\ref{prop:AA'}, the commutator can be rewritten
  \begin{align*}
    [\Ld,\Nd]\log\Om = & -K\Nd\log\Om \\
    = & -\left(-\frac{1}{4}\trchi\trchib +\half \chih\cdot\chibh - \rho\right)\Nd\log\Om \\
    = & \left(G'(\Ups) + G'(\Ups)\cdot\Nd\Ups\right)\Nd\log\Om.
  \end{align*}
  We therefore obtain the following formula
  \begin{align*}
    \Ld \Nd \log\Om = & G'(\Ups) \cdot \left(\Nd\Ups+\Nd\log\Om\right) +  G'(\Ups)\cdot\Nd\Ups\cdot\Nd\log\Om + (F'_4)^\dg\cdot\Nd^2\Ups.
  \end{align*}

  Therefore, using Lemma~\ref{lem:Bochnerest}, assumptions~\eqref{est:LowUpsOm} and Sobolev Lemma~\ref{lem:sob}, we have
  \begin{align*}
    \norm{\Nd^3\log\Om}_{L^2(S_v)} \les & \norm{G'(\Ups)}_{L^\infty(\HH)}\left(\norm{\Nd\Ups}_{L^\infty_vL^2}+\norm{\Nd\log\Om}_{L^\infty_vL^2}\right)\\
                                        & + \norm{G'(\Ups)}_{L^\infty(\HH)}\norm{\Nd\log\Om}_{L^\infty_vL^4}\norm{\Nd\Ups}_{L^4(S_v)}\\
                                        & + \norm{F'_4}_{L^\infty(\HH)}\norm{\Nd^2\Ups}_{L^2(S_v)} \\
    \les & C(\varep)\left(1 + \norm{\Nd^2\Ups}_{L^2(S_v)}\right).
  \end{align*}
  Plugging this estimate into~\eqref{est:Nd2Upssmooth}, we obtain
  \begin{align*}
    \norm{\Nd^2\Ups}_{L^\infty_vL^2} \les C + C\int_v \norm{\Nd^2\Ups}_{L^2(S_v)}\,\text{d}v.
  \end{align*}
  By a Gr\"onwall argument, we deduce
  \begin{align*}
    \norm{\Nd^2\Ups}_{L^\infty_v L^2} \les C.
  \end{align*}
  This finishes the proof of the Lemma in the case $m=2$.\\

  To prove estimates~\eqref{est:Ckhigherreg}, we commute elliptic equation~\eqref{eq:CMAF} for $\log\Om$ with $L$ and using the formula~\eqref{eq:smoothUps} the right-hand side can be expressed in terms of lower order derivatives of $\log\Om$. Details are left to the reader.
\end{proof}

\section{Local existence theorem}\label{sec:thmloc}
  In this section, we prove Theorem~\ref{thm:loc} by showing a more general local existence theorem for equations of the type (\ref{eq:smoothUps})-(\ref{eq:smoothOm}), where the unknown is the function $(v,\om) \in[1,2]\times\SSS^2 \mapsto s(v,\om)$. This strategy is similar to writing foliation by geometric flows as family of graphs and was already used in~\cite{Nicolo} and~\cite{Sauter}.
\subsection{Geometric set up and theorem}\label{sec:thmlocgeomsetup}
  Let $v_0\in[1,2)$, and define
  \begin{align*}
    \CC := \left\{ (v,\om) \in [v_0,2] \times \SSS^2\right\},  \,\,\,\,\,\,\,\,\,\,\, S_v := \{v\}\times\SSS^2 \subset  \CC.
  \end{align*}
  Similarly, let
  \begin{align*}
    \CC' := \left\{ (s,\om) \in [1,5/2]\times \SSS^2\right\},  \,\,\,\,\,\,\,\,\,\,\, S'_s := \{s\}\times\SSS^2\subset \CC'.
  \end{align*}
  Let $g$ be a smooth degenerate metric on $\CC'$ such that the induced metric on $S'_s$ is Riemannian. 
    Let $F'_1,F'_2,F'_3,F'_4$ be respectively a fixed scalar field, a fixed $1$, $2$ and $2$ $S'_s$-tangent tensor.\\
    For a function $s~:~\CC \to [1,5/2]$, we define
    \begin{align*}
      \Phi(s)~:~\,\,\,  \CC \,\, \,\,& \to  \,\,\,\,\,\,\,\,\,\CC' \\
       (v,\om) & \mapsto  (s(v,\om),\om).
    \end{align*}
  For a function $s~:~\CC\to[1,5/2]$, we define $\gd(s)$ to be the induced Riemannian metric on $S_v$ by $\Phi(s)^\ast g$, and $F_i(s) := \left(\Phi(s)^\ast F'_i\right)^\dg$, where $^\dg$ denotes the projection of $\CC$-tangent tensors to $S_v$-tangent tensors defined in Definition~\ref{def:compfol}.
Our goal in what follows is to prove local existence for the system of quasilinear elliptic transport equations in $\CC$
  \begin{align}\label{eq:sysfix}\begin{aligned}
    \log\Om & = \Ld^{-1}\left(F_1(s) + F_2(s)\cdot\Nd s + F_3(s)\cdot\Nd s\cdot\Nd s + F_4(s)\cdot\Nd^2 s\right),\\
    \pr_vs & = \Om^{-1},
    \end{aligned}
  \end{align}
   where $\Nd$ and $\Ld$ are respectively the covariant derivative and the Laplacian associated to $\gd(s)$ and where for a Riemannian 2-sphere $(S,\gd)$, $u := \Ld^{-1}f$ is the solution of
                \begin{align*}
                  \Ld u & = f-\int_{S}f, \\
                  \int_S u & = 0,
                \end{align*}
                             with integrals taken with respect to the metric $\gd$. \\
  
  For ease of notation, we shall define $F(s,\Nd s)$ and $G(s)$ to denote schematically
  \begin{align*}
    F(s,\Nd s,\Nd^2s) & := F_1(s) + F_2(s)\cdot\Nd s + F_3(s) \cdot \Nd s \cdot \Nd s + F_4(s)\cdot\Nd^2s.
  \end{align*}
  Let $s_0$ be a function $S_{v_0} \to [1,5/2]$ and extend it to $\CC$ by requiring that $\pr_vs_0 = 0$.
  Define $\Nda$ and $\Lda$ to be respectively the covariant derivative and Laplacian on all spheres $S_v$ associated to $\gd(s_0)$.
  Define $\log\Om_0$ on all spheres $S_v$ by
  \begin{align*}
    \log\Om_0 & := \Lda^{-1}\left(F(s_0,\Nda s_0,\Nda^2s_0)\right).
  \end{align*} 
  We have the following result.
  \begin{theorem}\label{thm:loc2}
    Assume that $s_0\in H^5(S_{v_0})$ and that $|\log\Om_0| \leq 1/100$.\\
    Assume moreover that $\left(S_{v_0},\gd(s_0)\right)$ is a weakly spherical $2$-sphere of radius $v_0$ (see Definition~\ref{def:unifweaksph}) with constants such that the Bochner and Hodge estimates from Lemmas~\ref{lem:Bochnerest}, \ref{lemmaHODGEGAMMA}, \ref{lem:heat} hold true. \\
    There exists
  \begin{align*}
    \de\left(\norm{(s_0-5/2)^{-1}}_{L^\infty(S_{v_0})},\norm{s_0}_{H^5(S_{v_0})},\norm{F'_i}_{C^{100}}\right) >0,
  \end{align*}
  and
  \begin{align*}
    s\in C^0([v_0,v_0+\de],H^5) \cap C^1([v_0,v_0+\de],H^4),
  \end{align*}
  such that on $\CC$ the system of equations~\eqref{eq:sysfix} is satisfied for $s$, with the initial condition $s|_{S_{v_0}} = s_0$.\\
  Moreover, we have \begin{align}\label{est:10logOm}|\log\Om| < 1/10.\end{align}
\end{theorem}             
\subsection{Proof of Theorem~\ref{thm:loc2}}             
\begin{proof}
  The proof goes by a classical Banach-Picard fixed-point theorem.\\
  
  \paragraph{\bf{Definition of the iteration}}
  As defined previously, we have $s_0(v,\om) := s_0(\om)$ and $\Om_0(v,\om) := \Om_0(\om)$.\\
  For all $n\geq 0$, we define $s_{n+1}$ and $\log\Om_{n+1}$ on $\CC$ by 
  \begin{align}
    s_{n+1}(v,\om) & := s_0(\om) + \int_{v^\ast}^v \Om^{-1}_n(v',\om) \,\text{d}v', \label{eq:intsn+1}\\
    \log(\Om_{n+1}) & := \Ldc^{-1}\left(F\left(s_{n+1},\Ndc s_{n+1},\Ndc^2s_{n+1}\right)\right).\label{eq:LdlogOmn+1}
  \end{align}
  We define
  \begin{align*}
    M_0 := \sum_{l\leq 5}\norm{\Nda^ls_0}_{L^2(S_{v_0})},
  \end{align*}
  and
  \begin{align*}
    M^\de_{n} := \sup_{v_0 \leq v \leq v_0+\de}\left(\sum_{l\leq 5} \norm{\Nda^l(s_{n}(v)-s_0)}_{L^2(S_{v_0})} + \sum_{l\leq 5} \norm{\Nda^l\log\Om_{n}(v)}_{L^2(S_{v_0})}\right).
  \end{align*}
  
  \paragraph{\bf{Boundedness of the iteration}}
  In this section, we show that if
  \begin{align*}
    \de < \de\left(\norm{(s_0-5/2)^{-1}}_{L^\infty(S_{v_0})}, M_0, \norm{F}_{C^{100}}, \norm{G}_{C^{100}}\right),
  \end{align*}
  then, defining $M:=2M_0$, we have for all $n\geq 0$
  \begin{align}
    M^\de_n & \leq M, \label{est:MnM}\\
    \sup_{v_0\leq v \leq v_0+\de} s_{n}(v) & < 5/2,\label{est:sn52} \\
    \sup_{v_0\leq v \leq v_0+\de}\norm{\log\Om_{n}}_{L^\infty(S_v)} & \leq 1/10. \label{est:logOmn}
  \end{align}

  We argue by induction and assume that these assumptions hold for an arbitrary $n\in \NNN$.
  First, using the transport equation~\eqref{eq:intsn+1} and estimate~\eqref{est:logOmn}, if $\de$ is small enough depending on $\norm{(s_0-5/2)^{-1}}_{L^\infty(S_{v_0})}$, we have
  \begin{align*}
    \sup_{v_0 \leq v \leq v_0 +\de} s_{n+1}(v) < 5/2, 
  \end{align*}
  and therefore~\eqref{est:sn52} is proved for $n+1$.\\

  Second, using estimate \eqref{est:MnM} and~\eqref{est:logOmn} at step $n$, we obtain that
  \begin{align*}
    \sup_{v_0\leq v \leq v_0+\de}\sum_{l\leq 5}\norm{\Nda^l(\Om^{-1}_n)}_{L^2(S_{v_0})} \leq C(M).
  \end{align*}
  Therefore, deriving and estimating equation~\eqref{eq:intsn+1}, we obtain
  \begin{align}\label{est:desn+1}
    \sup_{v_0\leq v \leq v_0+\de}\sum_{l\leq 5} \norm{\Nda^l(s_{n+1}(v)-s_0)}_{L^2(S_{v_0})} \leq \de C(M).
  \end{align}
  
  Third, we can rewrite equation~\eqref{eq:LdlogOmn+1}
  \begin{align}\label{eq:logOmn+1comp}
    \begin{aligned}
      \Lda\left(\log\Om_{n+1}-\log\Om_0\right) = & \left(\Lda-\Ldc\right)\left(\log\Om_{n+1}-\log\Om_0\right) + E_{n+1},
    \end{aligned}
  \end{align}
  with
    \begin{align*}
      E_{n+1} := & + \left(\Lda-\Ldc\right)\left(\log\Om_{0}\right) \\
                & + F(s_{n+1},\Ndc s_{n+1},\Ndc^2s_{n+1}) - F(s_0,\Nda s_0,\Nda^2s_0).
  \end{align*}
  From the weak sphericality assumption from Theorem~\ref{thm:loc2}, we can apply the elliptic estimate from Lemma~\ref{lem:heat} and we therefore deduce that
  \begin{align}\label{eq:NdlogOmn+10}
    \begin{aligned}
    & \norm{\Nda\left(\log\Om_{n+1}-\log\Om_0\right)}_{L^2(S_v)} + \norm{\log\Om_{n+1}-\log\Om_0 - \overline{\log\Om_{n+1}}}_{L^2(S_v)} \\
    \les & \norm{\left(\Lda-\Ldc\right)\left(\log\Om_{n+1}-\log\Om_0\right)}_{L^{4/3}(S_v)}  + \norm{E_{n+1}}_{L^{4/3}(S_v)}.
    \end{aligned}
  \end{align}
  We have
  \begin{align*}
    \left|\int_{\left(S_v,\gd(s_0)\right)}\log\Om_{n+1}\right| & \les \left|\int_{\left(S_v,\gd(s_0)\right)}\log\Om_{n+1}-\int_{\left(S_v,\gd(s_{n+1})\right)}\log\Om_{n+1} \right|\\
                                                  & \les  C \norm{s_{n+1}-s_0}_{L^\infty(S_v)} \norm{\log\Om_{n+1}}_{L^2(S_v)} \\
                                                  & \les \de C(M) \norm{\log\Om_{n+1}}_{L^2(S_v)} \\
                                                  & \les \de C(M)\norm{\log\Om_{n+1}-\log\Om_0}_{L^2(S_v)} +\de C(M).
  \end{align*}
  Moreover, we have
  \begin{align*}
    & \norm{\left(\Lda-\Ldc\right)\left(\log\Om_{n+1}-\log\Om_0\right)}_{L^{4/3}(S_v)} \\
    \les & \,C\left(\sum_{l\leq 5}\norm{\Nda s_0}_{L^2(S_v)},\sum_{l\leq 5} \norm{\Nda s_{n+1}}_{L^2(S_v)}\right) \\
    & \times \left(\sum_{l\leq 5}\norm{\Nda^l\left(s_{n+1}-s_0\right)}_{L^2(S_v)}\right)\left(\sum_{l\leq 2}\norm{\Nda^l\left(\log\Om_{n+1}-\log\Om_0\right)}_{L^2(S_v)}\right) \\
    \les & \,\de C(M) \left(\sum_{l\leq 2}\norm{\Nda^l\left(\log\Om_{n+1}-\log\Om_0\right)}_{L^2(S_v)}\right),
  \end{align*}
  and similarly
  \begin{align*}
    \norm{E_{n+1}}_{L^{4/3}(S_v)} \les & \,\de C(M).
  \end{align*}
  Therefore, for $\de C(M)$ small enough we can perform a standard absorption argument in the elliptic estimate~\eqref{eq:NdlogOmn+10} and we finally deduce
  \begin{align*}
    \norm{\Nda\left(\log\Om_{n+1}-\log\Om_0\right)}_{L^2(S_v)} + \norm{\log\Om_{n+1}-\log\Om_0}_{L^2(S_v)} \les \de C(M).
  \end{align*}
  We now prove that we have the following bounds for the remaining higher order derivatives
  \begin{align*}
    \sum_{l = 2}^5 \norm{\Nda^l\left(\log\Om_{n+1}-\log\Om_0\right)}_{L^2(S_v)} \les \de C(M).
  \end{align*}
  The proof goes by induction on $l$ for $l=2$ to $l=5$. We only do the case $l=5$ assuming that the bounds for $l \leq 4$ have been obtained, since it will be clear that the proof for the other cases is almost identical.
  Commuting equation~\eqref{eq:logOmn+1comp} with $\Nda^{3}$ gives
  \begin{align}\label{eq:LdNd3logOmn+1}\begin{aligned}
      \Lda \Nda^3\left(\log\Om_{n+1}-\log\Om_0\right) = & [\Lda,\Nda^3](\log\Om_{n+1}-\log\Om_0) \\
      & + \Nda^3\left(\left(\Lda-\Ldc\right)\left(\log\Om_{n+1}-\log\Om_0\right)\right) \\
    & + \Nda^3\left(E_{n+1}\right).
    \end{aligned}
  \end{align}

  We have the following three estimates, which proofs are left to the reader
  \begin{align*}
    \norm{[\Lda,\Nda^3](\log\Om_{n+1}-\log\Om_0)}_{L^2(S_v)} \les & \,\de C(M), \\
    \norm{\Nda^3E_{n+1}}_{L^2(S_v)} \les & \,\de C(M),
  \end{align*}
  and
  \begin{align*}
    & \norm{\Nda^3\left(\left(\Lda-\Ldc\right)\left(\log\Om_{n+1}-\log\Om_0\right)\right)}_{L^2(S_v)} \\
    \les & \,\de C(M)\left(\sum_{l\leq 5}\norm{\Nd^l\left(\log\Om_{n+1}-\log\Om_0\right)}_{L^2(S_v)}\right).
  \end{align*}

  Using these together with a Bochner estimate similar to Lemma~\ref{lem:Bochnerest} for tensors of arbitrary type applied to equation~\eqref{eq:LdNd3logOmn+1}, we deduce that
  \begin{align*}
    \norm{\Nda^5(\log\Om_{n+1}-\log\Om_0)}_{L^2(S_v)} \les & \norm{[\Lda,\Nda^3](\log\Om_{n+1}-\log\Om_0)}_{L^2(S_v)} \\
                                                           & + \norm{\Nda^3\left(\left(\Lda-\Ldc\right)\left(\log\Om_{n+1}-\log\Om_0\right)\right)}_{L^2(S_v)} \\
                                                           & + \norm{\Nda^3E_{n+1}}_{L^2(S_v)} \\
    \les & ~ \de C(M) + \de C(M) \norm{\Nd^5\left(\log\Om_{n+1}-\log\Om_0\right)}_{L^2(S_v)},
  \end{align*}
  which, performing a standard absorption argument, gives the desired bound.\\

  We therefore have proved that 
  \begin{align*}
    M^\de_{n+1} \leq M_0 + \de C(M) \leq M,
  \end{align*}
  provided that $\de$ has been chosen small enough.
  
  Moreover, by Sobolev embedding, we deduce that
  \begin{align*}
    \norm{\log\Om_{n+1}-\log\Om_0}_{L^\infty(S_v)} & \les \sum_{l\leq 5} \norm{\Nda^l(\log\Om_{n+1}-\log\Om_0)}_{L^2(S_v)} \\
                                                   & \leq \de C(M). 
  \end{align*}
  Therefore, for $\de$ such that $\de C(M) < 1/100$, and using the assumption $|\log\Om_0| \leq 1/100$, this proves the bound~(\ref{est:logOmn}) for $n+1$. This finishes the proof of the boundedness of the sequence $s_{n+1}$.\\

  \paragraph{\bf{Contraction of the iteration}}
  We define
  \begin{align*}
    \Delta_{n+1}^\de := \sup_{v_0\leq v \leq v_0+\de}\bigg(\sum_{l\leq 5} \norm{\Nda^l(s_{n+1}-s_n)}_{L^2(S_v)} + \sum_{l\leq 5} \norm{\Nda^l(\log\Om_{n+1}-\log\Om_n)}\bigg),
  \end{align*}
  and we show, provided that $\de$ has been chose small enough, that we have
  \begin{align*}
    \Delta_{n+1}^\de \leq \kappa \Delta^\de_n,
  \end{align*}
  with $\kappa <1$.
  The proof follows the lines of the proof of the boundedness.
  First, we have
  \begin{align*}
    \sum_{l\leq 5} \norm{\Nda^l\left(\Om_{n}^{-1}-\Om_{n-1}^{-1}\right)}_{L^2(S_v)} \leq C(M) \left(\sum_{l\leq 5}\norm{\Nda^l\left(\log\Om_n-\log\Om_{n-1}\right)}_{L^2(S_v)}\right).
  \end{align*}
  Therefore we deduce using equation~\eqref{eq:intsn+1} that
  \begin{align*}
    \sum_{l\leq 5} \norm{\Nda^l(s_{n+1}-s_n)}_{L^2(S_v)} \leq \de C(M) \De^\de_n.
  \end{align*}
  Performing a similar elliptic estimate as in the proof of the boundedness of $s_n$, we therefore deduce that
  \begin{align*}
    \sum_{l\leq 5}\norm{\Nda^l(\log\Om_{n+1}-\log\Om_n)}_{L^2(S_v)} \leq \de C(M) \De^\de_n.
  \end{align*}
  Thus, for $\de$ such that $\de C(M) <1$, we deduce the result and this finishes the proof of the contraction and of Theorem~\ref{thm:loc2}.
\end{proof}
                      
\subsection{Proof of Theorem~\ref{thm:loc}}
In this section, we show how Theorem~\ref{thm:loc} follows from Theorem~\ref{thm:loc2}.\\
  We define $v_0 := v^\ast$ and $s_0 := s^\ast = s|_{v^\ast}$, and $F'_1,F'_2,F'_3,F'_4$ to be the tensors defined in Lemma~\ref{lem:smootheq}.
  By assumptions and since the $F'_i$ have been defined as in Lemma~\ref{lem:smootheq}, the quantity $\log\Om_0$ defined in Section~\ref{sec:thmlocgeomsetup} coincides with $\log\Om|_{S_{v_0}}$ and we have
  \begin{align*}
    \Lda\log\Om_0 & = F_1(s_0) + F_2(s_0)\cdot\Nda s_0 + F_3(s_0)\cdot\Nda s_0\cdot\Nda s_0 + F_4(s_0)\cdot\Nda^2 s_0, \\
    \int_{(S_{v_0},\gd(s_0))}\log\Om_0  & = 0.
  \end{align*}
  By assumption of Theorem~\ref{thm:loc}, we have $s_0\in H^5(S_{v_0})$, $|\log\Om_0| \leq 1/100$ and that $(S_{v_0},\gd(s_0))$ is a weakly spherical $2$-sphere of radius $v_0$. \\
  Applying Theorem~\ref{thm:loc2}, there exists $\de>0$ and a function $s\in C^1([v_0,v_0+\de],\SSS^2)$, satisfying the system of equations \eqref{eq:sysfix}.
  Since by estimate~\eqref{est:10logOm} we have $|\pr_vs - 1| < 1/5$, the map $\Phi(s): [v_0,v_0+\de]\times \SSS^2 \to [1,5/2]\times \SSS^2$ admits a $C^1$-inverse by the global inverse theorem.
  This defines a $C^1$-function $v$ in geodesic coordinates, and therefore on $\HH$, taking values from $v^\ast = v_0$ to $v^\ast + \de$, which, using the conclusion of Theorem~\ref{thm:loc2}, is regular.
  Since equations~\eqref{eq:sysfix} are satisfied and since the $F'_i$ have been defined as in Lemma~\ref{lem:smootheq}, we deduce that $(S_v)_{v^\ast \leq v \leq v^\ast+\de}$ is a canonical foliation. In case $(S_v)_{1\leq v\leq v^\ast}$ was a regular canonical foliation, using equations~\eqref{eq:sysfix}, one deduces that $(S_v)_{v^\ast \leq v \leq v^\ast +\de}$ is a regular extension thereof. This finishes the proof of Theorem~\ref{thm:loc}. 

\appendix
\section{Proof of Propositions~\ref{prop:RR'} and~\ref{prop:AA'}}\label{app:RR'AA'}  
\begin{proof}
  In this section we prove the formulas from Proposition~\ref{prop:RR'} and~\ref{prop:AA'}.  
  The following computations are standard and can be found in various forms in~\cite{Alexakis},~\cite{Nicolo},~\cite{Sauter} for instance. In what follows, we use the formulas from~\cite{ChrKl93} pp. 149-150. We have
  \begin{align*}
    \alpha_{AB} = \R(L,e_A,L,e_B) = \R(L,e'_A+\Ups_AL,L,e'_B+\Ups_BL) = \R(L,e'_A,L,e'_B),
  \end{align*}
  and
  \begin{align*}
    \beta_A = & \half \R(e_A,L,\Lb,L) \\
    = & \half \R(e'_A+ \Ups_AL,L,\Lb'+2\Ups_Ae'_A+|\Ups|^2L,L) \\
    = & \half \R(e'_A,L,\Lb',L) + \Ups_A \R(e'_A,L,e'_B,L) \\
    = & \beta'_A + \Ups_B\alpha_{AB},
  \end{align*}
  and
  \begin{align*}
    \rho = & \frac{1}{4}\R(\Lb,L,\Lb,L) \\
            = & \frac{1}{4}\R(\Lb'+2\Ups_Ae_A'+|\Ups|L,L,\Lb'+2\Ups_Be_B' + |\Ups|^2L,L) \\
    = & \frac{1}{4}\R(\Lb',L,\Lb',L) + \half\Ups_A\R(e'_A,L,\Lb',L) + \half \Ups_B\R(\Lb',L,e'_B,L) + \Ups_A\Ups_B\R(e_A',L,e'_B,L) \\
    = & \rho' + (\beta')^\dg\cdot\Ups + (\alpha')^\dg\cdot\Ups\cdot\Ups.
  \end{align*}
  Following the previous computation for $\rho$ we also obtain
  \begin{align*}
    \sigma = & \frac{1}{4} \dual\R(\Lb,L,\Lb,L) \\
    = & \sigma' + \Ups_A \dual \R(e_A,L,\Lb',L) + \Ups_A\Ups_B \dual\R(e_A,L,e_B,L) \\
    = & \sigma' - \left(\dual\beta'\right)^\dg\cdot(\Ups) - (\dual\alpha')^\dg\cdot\Ups\cdot\Ups.
  \end{align*}
  We have
  \begin{align*}
    \betab_A = & \half \R(e_A,\Lb,\Lb,L) \\
    = & \half\R(e'_A + \Ups_AL,\Lb' + 2\Ups_Be'_B + |\Ups|^2L,\Lb'+2\Ups_Ce'_C + |\Ups|^2L, L) \\
    = & \half\R(e'_A,\Lb',\Lb',L) + \half\Ups_A\R(L,\Lb',\Lb',L) + \Ups_B\R(e'_A,e'_B,\Lb',L) + \Ups_C\R(e'_A,\Lb',e'_C,L) \\
               & + \Ups_A \Ups_B\R(L,e'_B,\Lb',L) + \Ups_A\Ups_C\R(L,\Lb',e'_C,L) + 2\Ups_B\Ups_C\R(e'_A,e'_B,e'_C,L) \\
               & + \half |\Ups|^2\R(e'_A,L,\Lb',L) + 2\Ups_A\Ups_B\Ups_C \R(L,e'_B,e'_C,L) + |\Ups|^2\Ups_C\R(e'_A,L,e'_C,L) \\
    = & \betab'_A - 2\Ups_A\rho' + 2 \dual\Ups_A\sigma' + (-\Ups_A\rho'+\dual\Ups_A\sigma') - 2\Ups_A \Ups\cdot(\betab')^\dg  +2\Ups_A\Ups\cdot(\betab')^\dg \\
    & - 2\dual\Ups_A\Ups \cdot (\dual\beta')^\dg + |\Ups|^2\beta'_A- 2\Ups_A\Ups\cdot\Ups\cdot(\alpha')^\dg + |\Ups|^2\Ups\cdot(\alpha')^\dg_{A} \\
    = & \betab'_A -3\Ups_A\rho' + 3\dual\Ups_A\sigma' -2\dual\Ups_A \Ups\cdot(\dual\beta')^\dg + |\Ups|^2\beta'_A -2\Ups_A\Ups\cdot\Ups\cdot(\alpha')^\dg + |\Ups|^2\Ups\cdot(\alpha')^\dg_A.
  \end{align*}
  This finishes the proof of Proposition~\ref{prop:RR'}.
  We turn to the connection coefficients. We have immediately $\chi_{AB} = \chi'_{AB}$.
  We also have
  \begin{align*}
    \ze_A = & \half \g(\D_AL,\Lb) \\
    = & \half \g(\D_{e'_A + \Ups_A L}L,\Lb' +2\Ups_Be'_B +|\Ups|^2L) \\
    = & \half \g(\D_{e'_A}L,\Lb') + \Ups_B \g(\D_{e'_A}L,e'_B) \\
    = & \ze'_A + \Ups\cdot\chi_A,
  \end{align*}
  and
  \begin{align*}
    \etab_A = & \half \g(\D_{L}\Lb,e_A) \\
    = & \half\g(\D_L(\Lb'+2\Ups_Be'_B+|\Ups|^2L,e'_A +\Ups_AL)) \\
    = & \half\g(\D_L\Lb',e'_A) + \g(\D_L(\Ups_Be'_B),e'_A) \\
    = & \etab'_A + \Nd_L\Ups_A.
  \end{align*}
  Finally, we have
  \begin{align*}
    \chib_{AB} = & \g(\D_A\Lb,e_B) \\
    = & \g(\D_{e_A}(\Lb' +2\Ups_Ce_C - |\Ups|^2L),e_B) \\
    = & 2\Nd_A\Ups_B + \g(\D_{e_A}(\Lb'-|\Ups|^2L),e_B) \\
    = & 2\Nd_A\Ups_B + \g(\D_{e'_A+\Ups_A L}(\Lb'-|\Ups|^2L),e'_B+\Ups_BL) \\
    = & 2\Nd_A\Ups_B + \g(\D_{e'_A}\Lb',e'_B) + \Ups_A\g(\D_L\Lb',e'_B) + \Ups_B\g(\D_{e'_A}\Lb',L) - |\Ups|^2\g(\D_{e'_A}L,e'_B) \\
    = & 2 \Nd_A\Ups_B + \chib'_{AB} + 2\Ups_A\etab'_B - 2\Ups_B \ze'_A - |\Ups|^2\chi'_{AB}. \\
  \end{align*}
This finishes the proof of Proposition~\ref{prop:AA'}.  

\end{proof}           
\section{Proof of Lemmas~\ref{lem:Hs} and~\ref{lem:heat}}\label{app:proofSec3}
  \subsection{Proof of Lemma~\ref{lem:Hs}}
    This section is dedicated to the proof of Lemma~\ref{lem:Hs}.\\
    In fact, we prove the following more general estimate
    \begin{align}\label{est:Hsapp}
      \norm{\Nd F}_{H^s(S)} \les \norm{F}_{H^{s+1}(S)}, 
    \end{align}
    for $-1 < s < 0$.
    \begin{remark}
      As it is clear from what follows, the same proof does not work for other ranges of exponents $s$, and would require additional regularity assumptions on the $2$-sphere $S$.
    \end{remark}

  From Proposition 2.3 in~\cite{ShaoBesov}, we have the following characterisation of $H^s(S)$ using the Littlewood-Paley projectors defined in Section~\ref{sec:LPtheory} 
  \begin{align}\label{est:HsBes}
    \norm{F}^2_{H^s(S)} \simeq \sum_{k\geq 0} 2^{2sk}\norm{P_kF}^2_{L^2(S)} + \norm{P_{<0}F}_{L^2(S)}^2. 
  \end{align}
  From Section 2.2 and Proposition 2.1 in~\cite{ShaoBesov}, we recall the following properties of the Littlewood-Paley projection operators defined in Section~\ref{sec:LPtheory}.\\
  For all $k\in\ZZZ$, we have
  \begin{align}\label{eq:LPorthog}
    P_k = P_kP_{k-1} + P_kP_k + P_kP_{k+1},
  \end{align}
  and for $F$ an $S$-tangent tensor and for all $k\in\ZZZ$, we have
  \begin{align}\label{est:LPbound}
    \begin{aligned}
      \norm{P_kF}_{L^2(S)} \les \norm{F}_{L^2(S)}, \,\,\,\,\,\, \norm{P_{<0}F}_{L^2(S)} \les \norm{F}_{L^2(S)},
    \end{aligned}
  \end{align}
  and,
  \begin{align}\label{est:LPband}
    \begin{aligned}
    & \norm{P_k \Nd F}_{L^2(S)} \les 2^k \norm{F}_{L^2(S)}, \,\,\,\,\,\, & \norm{\Nd P_k F}_{L^2(S)} \les 2^k \norm{F}_{L^2(S)},\\
    & \norm{P_{<0} \Nd F}_{L^2(S)} \les \norm{F}_{L^2(S)},  & \norm{\Nd P_{<0} F}_{L^2(S)} \les \norm{F}_{L^2(S)}.
    \end{aligned}
  \end{align}
  
                                                              We turn to the proof of estimate~\eqref{est:Hsapp}.
                                                              Using~\eqref{eq:LPId},~\eqref{est:HsBes} and~\eqref{est:LPband}, we have
  \begin{align}\label{est:LP1}
    \begin{aligned}
    \norm{\Nd F}^2_{H^s(S)} \les & \sum_{k \geq 0}2^{2sk}\norm{P_k\Nd F}^2_{L^2(S)} + \norm{P_{< 0}\Nd F}^2_{L^2(S)} \\
    \les & \sum_{k \geq 0} 2^{2sk}\norm{P_k\Nd P_{> k}F}^2_{L^2} + \sum_{k\geq 0} 2^{2sk}\norm{P_k\Nd P_{\leq k}F}^2_{L^2(S)} + \norm{F}^2_{L^2(S)}.
    \end{aligned}
  \end{align}
  The first term in the right-hand side of~\eqref{est:LP1} can be estimated using~\eqref{eq:LPorthog},~\eqref{est:LPband} and that $-1<s<0$
  \begin{align*}
 \sum_{k\geq 0} 2^{2sk}\norm{P_k\Nd P_{>k}F}^2_{L^2(S)} \les & \sum_{k \geq 0}2^{2(s+1)k}\norm{P_{>k}F}_{L^2(S)}^2 \\
    \les & \sum_{k \geq 0}\sum_{l>k} 2^{2(s+1)k}\norm{P_lF}_{L^2(S)}^2 \\
    \les & \sum_{l >0} 2^{2(s+1)l}\norm{P_lF}_{L^2(S)}^2 \sum_{k=0}^{l-1}2^{2(s+1)(k-l)} \\
    \les & \sum_{l\geq 0}2^{2(s+1)l}\norm{P_lF}_{L^2(S)}^2.
  \end{align*}
  For the second term in the right-hand side of~\eqref{est:LP1}, using~\eqref{est:LPbound}, we first write the following decomposition
  \begin{align*}
    & \norm{\Nd P_{\leq k} F}^2_{L^2(S)} \\
    = & \sum_{0 \leq l\leq k}\sum_{0 \leq l'\leq k} \int_{S} \Nd P_{l}F \cdot \Nd P_{l'}F + 2 \sum_{0 \leq l\leq k}\int_S \Nd P_lF \cdot \Nd P_{<0}F + \int_S\Nd P_{<0}F \cdot \Nd P_{<0}F.
  \end{align*}
  The first term can be estimated, using that $\Ld$ preserves the support of the projectors $P_k$ (see also Section 2.2 in~\cite{ShaoBesov}) and~\eqref{eq:LPId}, 
  \begin{align*}
    \sum_{0 \leq l\leq k}\sum_{0 \leq l'\leq k} \int_{S} \Nd P_{l}F \cdot \Nd P_{l'}F = & -\sum_{0 \leq l\leq k}\sum_{0 \leq l'\leq k} \int_{S} P_lF \Ld P_{l'}F \\
    = & -\sum_{0 \leq l\leq k}\sum_{l' = l-1}^{l+1} \int_{S} P_lF \Ld P_{l'}F \\
    \les & \sum_{0 \leq l \leq k} \sum_{l'=l-1}^{l+1}2^{2l}\norm{P_lF}_{L^2(S)}\norm{P_{l'}F}_{L^2(S)} \\
    \les & \sum_{0 \leq l\leq k+1}2^{2l}\norm{P_lF}^2_{L^2(S)},
  \end{align*}
  and similarly, we deduce for the last two terms, using~\eqref{est:LPbound} and~\eqref{est:LPband}
  \begin{align*}
    & 2 \sum_{0 \leq l\leq k}\int_S \Nd P_lF \cdot \Nd P_{<0}F + \int_S\Nd P_{<0}F \cdot \Nd P_{<0}F \\
    \les & \norm{F}^2_{L^2(S)}.
  \end{align*}
  
  Using this, \eqref{est:LPbound} and that $-1<s<0$, we therefore deduce that for the second term of~\eqref{est:LP1} we have
           \begin{align*}
             \sum_{k\geq 0} 2^{2sk}\norm{P_k \Nd P_{\leq k}F}^2_{L^2(S)} \les & \sum_{k\geq 0} 2^{2sk}\norm{\Nd P_{\leq k} F}^2_{L^2(S)} \\
    \les & \sum_{k\geq 0}\sum_{0 \leq l\leq k+1}2^{2sk}2^{2l}\norm{P_lF}^2_{L^2(S)} + \sum_{k\geq 0}2^{2sk}\norm{F}^2_{L^2(S)} \\
    \les & \sum_{l\geq 0} 2^{2(s+1)l}\norm{P_lF}^2_{L^2(S)}\left(\sum_{k\geq l-1}2^{2s(k-l)}\right) + \norm{F}^2_{L^2(S)} \\
    \les & \sum_{l\geq 0} 2^{2(s+1)l}\norm{P_lF}^2_{L^2(S)} + \norm{F}^2_{L^2(S)}.
  \end{align*}
  
  Finally, plugging the above estimates into~\eqref{est:LP1} and using~(\ref{est:HsBes}), we obtain
  \begin{align*}
    \norm{\Nd F}_{H^s(S)}^2 \les & \sum_{l\geq 0} 2^{2(s+1)l}\norm{P_lF}^2_{L^2(S)} + \norm{F}^2_{L^2(S)} \\
    \les & \norm{F}_{H^{s+1}(S)}^2.
  \end{align*}
  This finishes the proof of Lemma~\ref{lem:Hs}.

           \subsection{Proof of Lemma~\ref{lem:heat}}
           This section is dedicated to the proof of Lemma~\ref{lem:heat}.
           We assume that $f$ is a scalar function satisfying the following elliptic equation
           \begin{align}\label{eq:ellf}
             \Ld f = \Divd P + h.
           \end{align}
           Multiplying equation~\eqref{eq:ellf} by $f-\overline{f}$ and integrating by part, we have
           \begin{align*}
             \norm{\Nd f}^2_{L^2(S)} \leq \norm{P}_{L^2(S)}\norm{\Nd f}_{L^2(S)} + \norm{h}_{L^{4/3}(S)}\norm{f-\overline{f}}_{L^4(S)}.
           \end{align*}
           Using Lemma~\ref{lem:Hodge}, we have the following Poincar\'e inequality
           \begin{align*}
             \norm{f-\overline{f}}_{L^2(S)} = \norm{(f-\overline{f}, 0)}_{L^2(S)} = \norm{(\Dds_1)^{-1}(\Nd f)}_{L^2(S)} \les \norm{\Nd f}_{L^2(S)}.
           \end{align*}
           Therefore, using Sobolev Lemma~\ref{lem:sob}, we have
           \begin{align*}
             \norm{\Nd f}^2_{L^2(S)} + \norm{f-\overline{f}}^2_{L^2(S)} \les & \norm{P}_{L^2(S)}\norm{\Nd f}_{L^2(S)} \\
                                                                             & + \norm{h}_{L^{4/3}(S)}\big(\norm{\Nd f}_{L^2(S)} + \norm{f-\overline{f}}_{L^2(S)}\big),
           \end{align*}
           and the bound holds by a standard absorption argument. The bound on $\HH$ follows by integration in $v$. This finishes the proof of Lemma~\ref{lem:heat}.


\end{document}